\documentclass{article}
\usepackage{amsmath,amssymb,graphicx,enumerate,latexsym,tabularx,xcolor,theorem}
\usepackage{algorithm}
\usepackage{algpseudocode}
\usepackage{arxiv}
\usepackage{comment}
\usepackage[utf8]{inputenc} 
\usepackage[T1]{fontenc}    
\usepackage{hyperref}       
\usepackage{url}            
\usepackage{booktabs}       
\usepackage{amsmath}
\usepackage{amssymb}
\usepackage{amsfonts}       
\usepackage{nicefrac}       
\usepackage{microtype}      
\usepackage{xcolor}
\usepackage{graphicx}
\usepackage{stmaryrd}
\SetSymbolFont{stmry}{bold}{U}{stmry}{m}{n} 
\usepackage{bm}
\usepackage[normalem]{ulem}
\usepackage{cancel}
\usepackage[tickmarkheight=0.1cm,textsize=tiny]{todonotes}
\usepackage{algorithmicx}
\usepackage{algorithm}
\usepackage{algpseudocode}
\usepackage{caption}
\usepackage{subcaption}
\usepackage{orcidlink}
\usepackage[inline,final]{showlabels}
\usepackage[tickmarkheight=0.1cm,textsize=tiny]{todonotes}

\usepackage{xpatch}
\makeatletter
\xpatchcmd{\algorithmic}
{\ALG@tlm\z@}{\leftmargin\z@\ALG@tlm\z@}
{}{}
\makeatother

\hypersetup{
	colorlinks=true, 
	linkcolor=blue, 
	urlcolor=red, 
	citecolor=magenta,
	linktoc=all 
}

\newcommand{\assign}{:=}

\newcommand{\cdummy}{\cdot}

\newcommand{\mathd}{\mathrm{d}}
\newcommand{\nocomma}{}
\newcommand{\nosymbol}{}
\newcommand{\tmcolor}[2]{{\color{#1}{#2}}}
\newcommand{\tmdummy}{$\mbox{}$}

\newcommand{\tmmathbf}[1]{\ensuremath{\boldsymbol{#1}}}
\newcommand{\tmop}[1]{\ensuremath{\operatorname{#1}}}

\newcommand{\tmtextbf}[1]{\text{{\bfseries{#1}}}}
\newcommand{\tmtextit}[1]{\text{{\itshape{#1}}}}

\newcommand{\tmverbatim}[1]{\text{{\ttfamily{#1}}}}

\newenvironment{proof}{\noindent\textbf{Proof\ }}{\hspace*{\fill}$\Box$\medskip}

\newenvironment{tmparmod}[3]{\begin{list}{}{\setlength{\topsep}{0pt}\setlength{\leftmargin}{#1}\setlength{\rightmargin}{#2}\setlength{\parindent}{#3}\setlength{\listparindent}{\parindent}\setlength{\itemindent}{\parindent}\setlength{\parsep}{\parskip}} \item[]}{\end{list}}
\newtheorem{definition}{Definition}
\usepackage[T3,T1]{fontenc}
\DeclareSymbolFont{tipa}{T3}{cmr}{m}{n}
\DeclareMathAccent{\invbreve}{\mathalpha}{tipa}{16}
{\theorembodyfont{\rmfamily}\newtheorem{remark}{Remark}}
\newtheorem{theorem}{Theorem}
\newcommand{\correction}[1]{{\color{black}{#1}}}

\newcommand{\pdv}[2]{\frac{\partial #1}{\partial #2}}

\newcommand{\dv}[2]{\frac{\mathd #1}{\mathd #2}}

\newcommand{\C}{\tmmathbf{\mathcal{C}}}

\newcommand{\ba}{\tmmathbf{a}}

\newcommand{\be}{\tmmathbf{e}}
\newcommand{\bc}{\tmmathbf{c}}

\newcommand{\bn}{\tmmathbf{n}}

\newcommand{\bp}{\tmmathbf{p}}

\newcommand{\bq}{\tmmathbf{q}}
\newcommand{\bzero}{\tmmathbf{0}}

\newcommand{\bv}{\tmmathbf{v}}
\newcommand{\bx}{\tmmathbf{x}}
\newcommand{\by}{\tmmathbf{y}}

\newcommand{\vxi}{{\ensuremath{\tmmathbf{\xi}}}}

\newcommand{\uu}{\tmmathbf{u}}
\newcommand{\uU}{\tmmathbf{U}}

\newcommand{\tu}{\widetilde{\tmmathbf{u}}}
\newcommand{\pf}{\tmmathbf{f}}
\newcommand{\F}{\tmmathbf{F}}

\newcommand{\tf}{\tilde{\tmmathbf{f}}}
\newcommand{\tF}{\tilde{\tmmathbf{F}}}

\newcommand{\Div}{\nabla_{\tmmathbf{x}} \cdot}
\newcommand{\II}{\mathbb{I}_N}

\newcommand{\InI}{I_N}

\newcommand{\xithree}{\xi^3}
\newcommand{\xii}{\xi^i}

\newcommand{\polyP}{P}
\newcommand{\proP}{\mathcal{P}}

\newcommand{\Vi}{\tmverbatim{V}}

\providecommand{\subindex}{=}
\newcommand{\re}{\mathbb{R}}

\newcommand{\half}{\frac{1}{2}}
\newcommand{\ud}{\mathrm{d}}
\newcommand{\pd}[2]{\frac{\partial #1}{\partial #2}}

\newcommand{\Uad}{\mathcal{U}_{\textrm{\tmop{ad}}}}

\providecommand{\mathd}{\mathrm{d}}

\providecommand{\tmcolor}[2]{{\color[HTML]{1}#2}}
\providecommand{\tmmathbf}[1]{\ensuremath{\tmmathbf{#1}}}
\providecommand{\tmop}[1]{\ensuremath{\textrm{#1}}}
\providecommand{\tmtextbf}[1]{\tmtextbf{#1}}
\providecommand{\boldsymbol}[1]{\tmtextbf{#1}}

\providecommand{\citep}{}
\providecommand{\citet}{}

\newcommand{\ijkc}{i, j, k}
\newcommand{\Vu}{\tmverbatim{u}}
\newcommand{\VR}{\tmverbatim{R}}
\newcommand{\Nph}{N + \frac{1}{2}}
\newcommand{\Nmh}{N-{\frac{1}{2}}}

\newcommand{\Cijk}{C_{\tmmathbf{p}}}
\newcommand{\Oo}{\Omega_o}
\newcommand{\Oip}{\partial \Omega_{o, i}^R}
\newcommand{\Ois}{\partial \Omega_{o, i}^S}
\newcommand{\Oim}{\partial \Omega_{o, i}^L}

\newcommand{\bnr}{\widehat{\tmmathbf{n}}}

\newcommand{\en}{\mathbb{E}}
\newcommand{\thresh}{\mathbb{T}}
\newcommand{\Leg}{L}

\newcommand{\emh}{e - \frac{1}{2}}

\newcommand{\eph}{e + \frac{1}{2}}
\newcommand{\epmh}{e \pm \frac{1}{2}}

\newcommand{\atu}{\invbreve{\tmmathbf{u}}}
\newcommand{\atF}{\invbreve{\tmmathbf{F}}}
\newcommand{\utilow}{\invbreve{\tmmathbf{u}}^{\tmop{low}, n + 1}}

\newcommand{\Nnd}{\mathbb{N}_N^d}
\newcommand{\Nod}{\mathbb{N}_1^d}

\newcommand{\bss}{\tmmathbf{s}}
\newcommand{\sphi}{\phi}
\newcommand{\G}{\Gamma}

\newcommand{\X}{\Xi}

\newcommand{\bet}{\tmmathbf{\eta}}
\newcommand{\bga}{\tmmathbf{\gamma}}
\newcommand{\nou}{\breve{\tmmathbf{u}}}
\newcommand{\uebp}{\tmmathbf{u}_{e, \tmmathbf{p}}}
\newcommand{\uebq}{\tmmathbf{u}_{e, \tmmathbf{q}}}
\newcommand{\uep}{\tmmathbf{u}_{e, p}}

\newcommand{\uez}{\tmmathbf{u}_{e, 0}}
\newcommand{\uepoz}{\tmmathbf{u}_{e + 1, 0}}
\newcommand{\ueN}{\tmmathbf{u}_{e, N}}
\newcommand{\pph}{p + \frac{1}{2}}
\newcommand{\pmh}{p - \frac{1}{2}}

\newcommand{\Cs}{C_\text{CFL}}

\makeatletter
\newcommand*{\mint}[1]{%
	\mint@l{#1}{}%
}
\newcommand*{\mint@l}[2]{%
	\@ifnextchar\limits{%
		\mint@l{#1}%
	}{%
		\@ifnextchar\nolimits{%
			\mint@l{#1}%
		}{%
			\@ifnextchar\displaylimits{%
				\mint@l{#1}%
			}{%
				\mint@s{#2}{#1}%
			}%
		}%
	}%
}
\newcommand*{\mint@s}[2]{%
	\@ifnextchar_{%
		\mint@sub{#1}{#2}%
	}{%
		\@ifnextchar^{%
			\mint@sup{#1}{#2}%
		}{%
			\mint@{#1}{#2}{}{}%
		}%
	}%
}
\def\mint@sub#1#2_#3{%
	\@ifnextchar^{%
		\mint@sub@sup{#1}{#2}{#3}%
	}{%
		\mint@{#1}{#2}{#3}{}%
	}%
}
\def\mint@sup#1#2^#3{%
	\@ifnextchar_{%
		\mint@sup@sub{#1}{#2}{#3}%
	}{%
		\mint@{#1}{#2}{}{#3}%
	}%
}
\def\mint@sub@sup#1#2#3^#4{%
	\mint@{#1}{#2}{#3}{#4}%
}
\def\mint@sup@sub#1#2#3_#4{%
	\mint@{#1}{#2}{#4}{#3}%
}
\newcommand*{\mint@}[4]{%
	\mathop{}%
	\mkern-\thinmuskip
	\mathchoice{%
		\mint@@{#1}{#2}{#3}{#4}%
		\displaystyle\textstyle\scriptstyle
	}{%
		\mint@@{#1}{#2}{#3}{#4}%
		\textstyle\scriptstyle\scriptstyle
	}{%
		\mint@@{#1}{#2}{#3}{#4}%
		\scriptstyle\scriptscriptstyle\scriptscriptstyle
	}{%
		\mint@@{#1}{#2}{#3}{#4}%
		\scriptscriptstyle\scriptscriptstyle\scriptscriptstyle
	}%
	\mkern-\thinmuskip
	\int#1%
	\ifx\\#3\\\else_{#3}\fi
	\ifx\\#4\\\else^{#4}\fi
}
\newcommand*{\mint@@}[7]{%
	\begingroup
	\sbox0{$#5\int\m@th$}%
	\sbox2{$#5\int_{}\m@th$}%
	\dimen2=\wd0 %
	\let\mint@limits=#1\relax
	\ifx\mint@limits\relax
	\sbox4{$#5\int_{\kern1sp}^{\kern1sp}\m@th$}%
	\ifdim\wd4>\wd2 %
	\let\mint@limits=\nolimits
	\else
	\let\mint@limits=\limits
	\fi
	\fi
	\ifx\mint@limits\displaylimits
	\ifx#5\displaystyle
	\let\mint@limits=\limits
	\fi
	\fi
	\ifx\mint@limits\limits
	\sbox0{$#7#3\m@th$}%
	\sbox2{$#7#4\m@th$}%
	\ifdim\wd0>\dimen2 %
	\dimen2=\wd0 %
	\fi
	\ifdim\wd2>\dimen2 %
	\dimen2=\wd2 %
	\fi
	\fi
	\rlap{%
		$#5%
		\vcenter{%
			\hbox to\dimen2{%
				\hss
				$#6{#2}\m@th$%
				\hss
			}%
		}%
		$%
	}%
	\endgroup
}

\usepackage{esint}
\newcommand{\qint}{\mint{-}}

\title{Lax-Wendroff Flux Reconstruction on adaptive curvilinear meshes with error based time stepping for hyperbolic conservation laws}

\author{
Arpit~Babbar \orcidlink{0000-0002-9453-370X} \\
Centre for Applicable Mathematics\\
Tata Institute of Fundamental Research\\
Bangalore -- 560065\\
\texttt{arpit@tifrbng.res.in} \\
\And
Praveen~Chandrashekar \orcidlink{0000-0003-1903-4107}\thanks{Corresponding author}\\
Centre for Applicable Mathematics\\
Tata Institute of Fundamental Research\\
Bangalore -- 560065\\
\texttt{praveen@math.tifrbng.res.in}
}

\begin{document}

\maketitle

\begin{abstract}
Lax-Wendroff Flux Reconstruction (LWFR) is a single-stage, high order, quadrature free method for solving hyperbolic conservation laws. This work extends the LWFR scheme to solve conservation laws on curvilinear meshes with adaptive mesh refinement (AMR). The scheme uses a subcell based blending limiter to perform shock capturing and exploits the same subcell structure to obtain admissibility preservation on curvilinear meshes. It is proven that the proposed extension of LWFR scheme to curvilinear grids preserves constant solution (free stream preservation) under the standard metric identities.
For curvilinear meshes, linear Fourier stability analysis cannot be used to obtain an optimal CFL number. Thus, an embedded-error based time step computation method is proposed for LWFR method which reduces fine-tuning process required to select a stable CFL number using the wave speed based time step computation.  The developments are tested on compressible Euler's equations, validating the blending limiter, admissibility preservation, AMR algorithm, curvilinear meshes and error based time stepping.
\end{abstract}


\section{Introduction}
Lax-Wendroff method is a single step method for time dependent problems in contrast to method of lines approach which combines a spatial discretization scheme with a Runge-Kutta method in time. The Lax-Wendroff idea has been used for hyperbolic conservation laws to develop single step finite volume and discontinuous Galerkin methods ~\cite{Qiu2003,Qiu2005b,Zorio2017,Burger2017,Carrillo2021}. Another approach to develop high order, single-stage schemes is based on ADER schemes~{\cite{Titarev2002,Dumbser2008}}.

Flux Reconstruction (FR) method introduced by Huynh~\cite{Huynh2007} is a finite-element type high order method which is quadrature-free. The key idea in this method is to construct a continuous flux approximation and then use collocation at solution points which leads to an efficient implementation that can exploit optimized matrix-vector operations and vectorization capabilities of modern CPUs. The continuous flux approximation requires a correction function whose choice affects the accuracy and stability of the method~\cite{Huynh2007,Vincent2011a,Vincent2015,Trojak2021}; by properly choosing the correction function and solution points, the FR method can be shown to be equivalent to some discontinuous Galerkin and spectral difference schemes~{\cite{Huynh2007,Trojak2021}}.

In~{\cite{babbar2022}}, a Lax-Wendroff Flux Reconstruction (LWFR) scheme was proposed which used the approximate Lax-Wendroff procedure of~{\cite{Zorio2017}} to obtain an element local high order approximation of the time averaged flux and then performs the FR procedure on it to perform evolution in a single stage. The numerical flux was carefully constructed in~{\cite{babbar2022}} to obtain enhanced accuracy and linear stability based on Fourier stability analysis. In~{\cite{babbar2023admissibility}}, a subcell based shock capturing blending scheme was introduced for LWFR based on the work of subcell based scheme of~{\cite{henneman2021}}. To enhance accuracy,~{\cite{babbar2023admissibility}} used Gauss-Legendre solution points and performed MUSCL-Hancock reconstruction on the subcells. Since the subcells used in~{\cite{babbar2023admissibility}} were inherently non-cell centred, the MUSCL-Hancock scheme was extended to non-cell centred grids along with the proof of~{\cite{Berthon2006}} for admissibility preservation. The subcell structure was exploited to obtain a provably admissibility preserving LWFR scheme by careful construction of the \tmtextit{blended numerical flux} at the element interfaces.

In this work, the LWFR scheme of~{\cite{babbar2022}} is further developed to incorporate three new features:
\begin{enumerate}
\item Ability to work on curvilinear, body-fitted grids
\item Ability to work on locally and dynamically adapted grids with hanging nodes
\item Automatic error based time step computation
\end{enumerate}

Curvilinear grids are defined in terms of a tensor product polynomial map from a reference element to the physical element. The conservation law is transformed to the coordinates of the reference element and then the LWFR procedure is applied leading to a collocation method that has similar structure as on Cartesian grids. This structure also facilitates the extension of the provably admissibility preserving subcell based blending scheme of~{\cite{babbar2023admissibility}} to curvilinear grids. The FR formulation on curvilinear grids is based on its equivalence with the DG scheme, see~{\cite{Kopriva2006}}, which also obtained certain metric identities that are required for preservation of constant solutions, that is, free stream preservation. See references in~{\cite{Kopriva2006}} for a review of earlier study of metric terms in the context of other higher order schemes like finite difference schemes.  The free stream preserving conditions for the LWFR scheme are proven to be the same discrete metric identities as that of~{\cite{Kopriva2006}}. The only requirement for the required metric identities in two dimensions is that the mappings used to define the curvilinear elements must have degree less than or equal to the degree of polynomials used to approximate the solution.

In many problems, there are non-trivial and sharp solution features only in some localized parts of the domain and these features can move with respect to time. Using a uniform mesh to resolve small scale features is computationally expensive and adaptive mesh refinement (AMR) is thus very useful. In this work, we perform adaptive mesh refinement based on some local error or solution smoothness indicator. Elements with high error indicator are flagged for refinement and those with low values are flagged for coarsening. A consequence of this procedure is that we get non-conformal elements with hanging nodes which is not a major problem with discontinuous Galerkin type methods, except that one has to ensure conservation is satisfied.  For discontinuous Galerkin methods based on quadrature, conservation is ensured by performing quadrature on the cell faces from the refined side of the face~\cite{schaal2015,Zanotti2015}. For FR type methods which are of collocation type, we need numerical fluxes at certain points on the element faces, which have to \correction{be} computed on a refined face without loss of accuracy and such that conservation is also satisfied. For the LWFR scheme, we use the Mortar Element Method~\cite{Kopriva1996,Kopriva2002} to compute the solution and fluxes at non-conformal faces. The resulting method is conservative and also preserves free-stream condition on curvilinear, adapted grids.

The choice of time step is restricted by a CFL-type condition in order to satisfy linear stability and some other non-linear stability requirements like maintaining positive solutions. Linear stability analysis  can be performed on uniform Cartesian grids only, leading to some CFL-type condition which depends on wave speed estimates. In practice these conditions are then also used for curvilinear grids but they may not be optimal and may require tuning the time step for each problem by adding a safety factor.  Thus, automatic time step selection methods based on some error estimates become very relevant for curvilinear grids. Error based time stepping methods are already developed for ODE solvers;  and by using a method of lines approach to convert partial differential equations to a system of ordinary differential equations, error-based time stepping schemes of ODE solvers have been applied to partial differential equations~{\cite{berzins1995,ketcheson2020,ware1995}} and recent application to CFD problems can be found in~{\cite{Ranocha2021,ranocha2023}}. The LWFR scheme makes use of a Taylor expansion in time of the time averaged flux; by truncating the Taylor expansion at one order lower, we can obtain two levels of approximation, whose difference is used as a local error indicator to adapt the time step. As a consequence the user does not need to specify a CFL number, but only needs to give some error tolerances based on which the time step is automatically decreased or increased.

The rest of the paper is organized as follows. In Section~\ref{sec:curvilinear.coords}, we review notations and the transformation of conservation laws from curved elements to a reference cube following~{\cite{Kopriva2006,kopriva2009}}. In Section~\ref{sec:cons.lw}, the LWFR scheme of~{\cite{babbar2022}} is extended to curvilinear grids. In Section~\ref{sec:fr}, we review FR on curvilinear grids and use it to construct LWFR on curvilinear grids in Section~\ref{sec:lwfr.curved}. Section~\ref{sec:free.stream.lwfr} shows that the free stream preservation condition of LWFR is the standard metric identity of~{\cite{Kopriva2006}}. In Section~\ref{sec:amr}, the Mortar Element Method for treatment of non-conformal interfaces in AMR of~{\cite{Kopriva1996}} is extended to LWFR. In Section~\ref{sec:time.stepping}, error-based time stepping methods are discussed; Section~\ref{sec:rk.error.section} reviews error-based time stepping methods for Runge-Kutta and Section~\ref{sec:error.lw} introduces an embedded error-based time stepping method for LWFR. In Section~\ref{sec:numerical.results}, numerical results are shown to demonstrate the scheme's capability of handling adaptively refined curved meshes and benefits of error-based time stepping. Section~\ref{sec:conclusions} gives a summary and draws conclusions from the work. In \correction{Appendix}~\ref{sec:shock.capturing}, the admissibility preserving subcell limiter for LWFR from~{\cite{babbar2023admissibility}} is reviewed and extended to curvilinear grids.

\section{Conservation laws and curvilinear grids}\label{sec:curvilinear.coords}

The developments in this work are applicable to a wide class of hyperbolic conservation laws but the numerical experiments are performed on 2-D compressible Euler's equations, which are a system of conservation laws given by
\begin{equation}
\label{eq:2deuler} \pd{}{t}  \left(\begin{array}{c}
\rho\\
\rho u\\
\rho v\\
E
\end{array}\right) + \pd{}{x}  \left(\begin{array}{c}
\rho u\\
p + \rho u^2\\
\rho uv\\
(E + p) u
\end{array}\right) + \pd{}{y}  \left(\begin{array}{c}
\rho v\\
\rho uv\\
p + \rho v^2\\
(E + p) v
\end{array}\right) = \bzero
\end{equation}
Here, $\rho, p$ and $E$ denote the density, pressure and total energy per unit
volume of the gas, respectively and $(u, v)$ are Cartesian components of the
fluid velocity. For a polytropic gas, an equation of state $E = E (\rho, u, v,
p)$ which leads to a closed system is given by
\begin{equation}
\label{eq:2dstate} E = \frac{p}{\gamma - 1} +
\frac{1}{2} \rho (u^2 + v^2)
\end{equation}
where $\gamma > 1$ is the adiabatic constant. For the sake of simplicity and
generality, we subsequently explain the development of the algorithms for a general hyperbolic conservation law written as
\begin{equation}
\uu_t + \Div \pf ( \uu ) = \bzero \label{eq:con.law}
\end{equation}
where $\uu \in \mathbb{R}^p$ is the vector of conserved quantities, $\pf
( \uu )=(\pf_1, \ldots, \pf_d) \in \mathbb{R}^{p \times d}$ is the corresponding physical flux, $\bx$ is in domain
$\Omega \subset \mathbb{R}^d$ and
\begin{equation}
\Div \pf = \sum_{i = 1}^d \partial_{x_i}  \pf_i \label{eq:defn1.div}
\end{equation}
Let us partition $\Omega$ into $M$ non-overlapping quadrilateral/hexahedral elements $\Omega_e$ such that
\[ \Omega = \bigcup_{e = 1}^M \Omega_e \]
The elements $\Omega_e$ are allowed to have curved boundaries in order to match curved boundaries of the problem domain $\Omega$. To construct the numerical approximation, we map each element $\Omega_e$ to a reference element $\Oo = [-1,1]^d$ by a bijective map $\Theta_e : \Oo \to \Omega_e$
\[
\bx = \Theta_e ( \vxi )
\]
where  $\vxi = ( \xii )_{i = 1}^d$ are the coordinates in the reference element,
and the subscript $e$ will usually be suppressed. We will denote a $d$-dimensional
multi-index as $\bp = (p_i)_{i = 1}^d$. In this work, the reference map is
defined using tensor product Lagrange interpolation of degree $N \ge 1$,
\begin{equation}
\Theta ( \vxi ) = \sum_{\bp \in \Nnd} \widehat{\bx}_{\bp}
\ell_{\bp} ( \vxi ) \label{eq:reference.map}
\end{equation}
where
\begin{equation}
\Nnd = \left\{ \bp = (p_1, \ldots, p_d) \ : \ p_i \in \{ 0, 1, \ldots, N \}, 1 \leq i \leq d \right\}
\label{eq:Nnd}
\end{equation}
and $\left\{ \ell_{\bp} \right\}_{\bp \in \Nnd}$ is the degree $N$ Lagrange
polynomial corresponding to the Gauss-Legendre-Lobatto (GLL) points $\left\{ \vxi_{\bp} \right\}_{\bp \in \Nnd}$ so that $\Theta ( \vxi_{\bp} ) = \widehat{\bx}_{\bp}$
for all $\bp \in \Nnd$. Thus, the points $\left\{ \vxi_{\bp} \right\}_{\bp \in
\Nnd}$ are where the reference map will be specified and they will also be taken to
be the solution points of the Flux Reconstruction scheme throughout this work.
The functions $\left\{ \ell_{\bp} \right\}_{\bp \in \Nnd}$ can be written as a tensor product
of the 1-D Lagrange polynomials $\{ \ell_{p_i} \}_{p_i \subindex 0}^N$ of degree $N$
corresponding to the GLL points $\{ \xi_{p_i} \}_{p_i = 0}^N$
\begin{equation}
\ell_{\bp} ( \vxi ) = \prod_{i = 1}^d \ell_{p_i} ( \xii
), \qquad
\ell_{p_i}(\xi^i) = \prod_{k=0, k \ne i}^N  \frac{\xi^i - \xi_{p_k}}{\xi_{p_i} - \xi_{p_k}}
\label{eq:lagrange.basis}
\end{equation}

The numerical approximation of the conservation law will be developed by first transforming the PDE in terms of the coordinates of the reference cell. To do this, we need to introduce covariant and contravariant basis vectors with respect to the reference coordinates.

\begin{definition}[Covariant basis]\label{defn:covariant.basis}
The coordinate basis vectors $\left\{ \ba_i \right\}_{i = 1}^d$ are defined
so that $\ba_i, \ba_j$ are tangent to $\{ \xi^k = \tmop{const} \}$ where $i,
j, k$ are cyclic. They are explicitly given as
\begin{equation}
\ba_i = (a_{i,1}, \ldots, a_{i,d}) = \pdv{\bx}{\xii}, \qquad 1 \leq i \leq d
\end{equation}
\end{definition}

\begin{definition}[Contravariant basis] \label{defn:contravariant.basis}
The contravariant basis vectors $\left\{ \ba^i \right\}_{i = 1}^d$ are the respective normal vectors to the coordinate planes $\{ \xi_i = \tmop{const} \}_{i = 1}^3$. They are explicitly given as
\begin{equation}
\ba^i = (a^i_1, \ldots, a^i_d) = \nabla_{\bx}  \xii, \qquad 1 \leq i \leq d
\end{equation}
\end{definition}

The covariant basis vectors $\ba_i$ can be computed by differentiating the reference map $\Theta(\xi)$. The contravariant basis vectors can be computed using~{\cite{Kopriva2006,kopriva2009}}
\begin{equation}
J \ba^i = J \nabla \xii = \ba_j \times \ba_k
\label{eq:contravariant.identity}
\end{equation}
where $( \ijkc )$ are cyclic, and $J$ denotes the Jacobian of the
transformation which also satisfies
\[
J = \textrm{det} \left[ \pdv{\bx}{\vxi} \right] = \ba_i \cdot ( \ba_j \times \ba_k ) \qquad (i, j, k) \text{ cyclic}
\]
The divergence of a flux vector can be computed in reference coordinates using the contravariant basis vectors
as~{\cite{Kopriva2006,kopriva2009}}
\begin{equation}
\Div \pf = \frac{1}{J}  \sum_{i = 1}^d \pdv{}{\xii}  ( J \ba^i \cdot
\pf ) \label{eq:divergence.transform.contravariant}
\end{equation}
Consequently, the gradient of a scalar function $\phi$ becomes
\begin{equation}
\nabla \phi = \frac{1}{J}  \sum_{i = 1}^d \pdv{}{\xii}  [ ( J \ba^i
) \phi ] \label{eq:gradient.transform.contravariant}
\end{equation}
Within each element $\Omega_e$, performing change of variables with the
reference map $\Theta_e$~\eqref{eq:divergence.transform.contravariant}, the
transformed conservation law is given by
\begin{equation}
\tu_t + \nabla_{\vxi} \cdot \tf = \bzero
\label{eq:transformed.conservation.law}
\end{equation}
where
\begin{equation}
\label{eq:contravariant.flux.defn}
\tu  = J \uu, \qquad
\tf^i  = J \ba^i \cdot \pf = \sum_{n = 1}^d Ja_n^i  \pf_n
\end{equation}
The flux $\tf$ is referred to as the contravariant flux.

The vectors $\{ J \ba^i \}_{i \subindex 1}^d$ are called the metric terms
and the metric identity is given by
\begin{equation}
\sum_{i = 1}^d \pdv{(J \ba^i)}{\xii} = \bzero
\label{eq:metric.identity.contravariant}
\end{equation}
The metric identity can be obtained by reasoning that the gradient of a
constant function is zero and
using~\eqref{eq:gradient.transform.contravariant} or that a constant solution
must remain constant in~\eqref{eq:transformed.conservation.law}. The metric
identity is crucial for studying free stream preservation of a
numerical scheme.

\begin{remark}
The equations for two dimensional case can be obtained by setting $( \Theta ( \vxi ) )_3 = x_3 ( \vxi
) = \xithree$ so that $\ba_3 = (0, 0, 1)$.
\end{remark}

\section{Lax-Wendroff Flux Reconstruction (LWFR) on curvilinear
grids}\label{sec:cons.lw}

The solution of the conservation law will be approximated by piecewise polynomial functions which are allowed to be discontinuous across the elements. In each element $\Omega_e$, the solution is approximated by
\begin{equation}
\widehat{\uu}_e^{\delta} ( \vxi ) =  \sum_{\bp} \uebp
\ell_{\bp} ( \vxi )
\end{equation}
where the $\ell_{\bp}$ are tensor-product polynomials of degree $N$ which have been already introduced before to define the map to the reference element. The hat will be used to denote functions written in terms of the reference coordinates and the delta denotes functions which are possibly discontinuous across the element boundaries. Note that the coefficients $\uebp$ are the values of the function at the solution points which are GLL points.

\subsection{Flux Reconstruction (FR)}\label{sec:fr}

Recall that we defined the multi-index $\bp = (p_i)_{i = 1}^d$~\eqref{eq:Nnd}
where $p_i \in \{ 0, 1 \ldots, N \}$. Let $i \in \{ 1, \ldots, d \}$ denote a
coordinate direction and $S \in \{ L, R \}$ so that $(S, i)$ corresponds to
the face $\Ois$ in direction $i$ on side $S$ which has the reference outward
normal $\bnr_{S, i}$, see Figure~\ref{fig:ref.map}. Thus, $\Oip$ denotes the
face where reference outward normal is $\bnr_{R, i} = \be_i$ and $\Oim$ has
outward unit normal $\bnr_{L, i} = - \bnr_{R, i}$.

The FR scheme is a collocation scheme at each of the solution points $\left\{
\vxi_{\bp} = (\xi_{p_i})_{i = 1}^d, p_i = 0, \ldots, N \right\}$. We will thus
explain the scheme for a fixed $\vxi = \vxi_{\bp}$ and denote $\vxi_i^S$ as
the projection of {\vxi} to the face $S = L, R$ in the $i^{\tmop{th}}$
direction (see Figure~\ref{fig:ref.map}), i.e.,
\begin{equation}
( \vxi_i^S )_j = \left\{\begin{array}{lll}
\xi_j, & \qquad & j \neq i\\
- 1, &  & j = i, S = L\\
+ 1, &  & j = i, S = R
\end{array}\right. \label{eq:xis.notation}
\end{equation}
\begin{figure}
\centering
{\includegraphics[width = 0.8\textwidth]{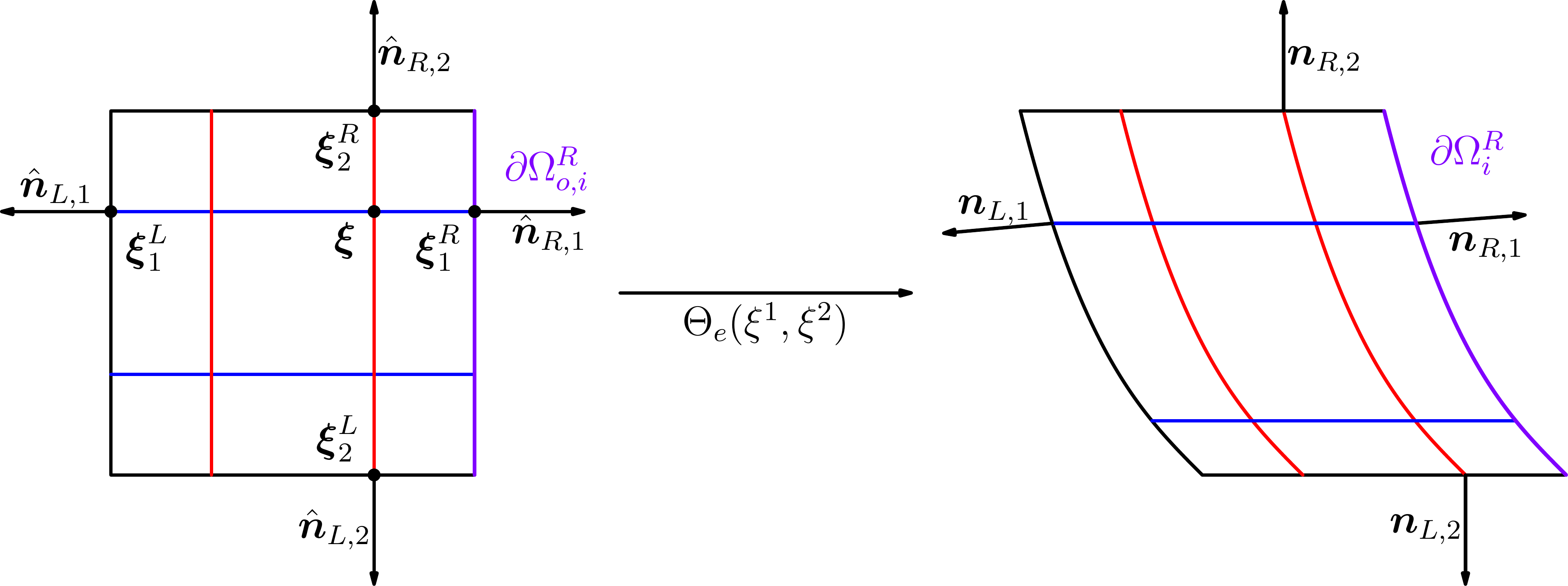}}
\caption{\label{fig:ref.map}Illustration of reference map, solution point
projections, reference and physical normals}
\end{figure}

The first step is to construct an approximation to the flux by interpolating at the solution points
\begin{equation}
( \tf_e^{\delta} )_i ( \vxi ) =  \sum_{\bp}
( J \ba^i \cdot {\pf} ) ( \vxi_{\bp}
) \ell_{\bp} ( \vxi )
\label{eq:flux.poly.defn}
\end{equation}
which may be discontinuous across
the element interfaces. In order to couple the neighbouring elements and ensure conservation property,
continuity of the normal flux at the interfaces is enforced by constructing
the \tmtextit{continuous flux approximation} using the FR correction functions
$g_L, g_R$~{\cite{Huynh2007}}. We construct this for the contravariant flux
$\tf^{\delta}$~\eqref{eq:flux.poly.defn} by performing correction along
each direction $i$,
\begin{equation}
( \tf_e ( \vxi ) )^i = ( \tf_e^{\delta} (
\vxi ) )^i + ( ( \tf_e \cdot \bnr_{R,i} )^{\ast} -
\tf_e^{\delta} \cdot \bnr_{R,i} ) ( \vxi_i^R ) g_R (\xi_{p_i})
- ( ( \tf_e \cdot \bnr_{L,i} )^{\ast} - \tf_e^{\delta} \cdot
\bnr_{L,i} ) ( \vxi_i^L ) g_L (\xi_{p_i})
\label{eq:cts.num.flux}
\end{equation}
where $\tf_e \cdot \bnr_{S,i} ( \vxi_i^S )$ denotes the trace value of
the normal flux in element $\Omega_e$ and $( \tf_e \cdot \bnr_i )^{\ast}
( \vxi_i^S )$ denotes the numerical flux. We will use Rusanov's numerical
flux~{\cite{Rusanov1962}} which for the face $(S, i)$ is given by
\begin{equation}
( \tf_e \cdot \bnr_{S, i} )^{\ast} = \tf^{\ast} ( \uu^-_{S,
i}, \uu^+_{S, i}, \bnr_{S, i} ) = \frac{1}{2}  [ (
\tf^{\delta} \cdot \bnr_{S, i} )^+ + ( \tf^{\delta} \cdot
\bnr_{S, i} )^- ] - \frac{\lambda_{S, i}}{2}  ( \uu^+_{S,
i} - \uu^-_{S, i} ) \label{eq:rusanov.flux}
\end{equation}
The $( \tf^{\delta} \cdot \bn_{S, i} )^{\pm}$ and $\uu_{S,
i}^{\pm}$ denote the trace values of the normal flux and solution from outer,
inner directions respectively; the inner direction corresponds to the element
$\Omega_e$ while the outer direction corresponds to its neighbour across the
interface $(S, i)$. The $\lambda_{S, i}$ is a local wave speed estimate at the
interface $(S, i)$. For compressbile Euler's equations~\eqref{eq:2deuler},
the wave speed is estimated as~\cite{Ranocha2022}
\begin{equation*}
 \lambda = \max (| v^- |, | v^+ |) + \max (| c^- |, | c^+ |),  \qquad
 v^{\pm} = \bv \cdot \bn^{\pm}, \qquad c^{\pm} = \sqrt{\gamma p^{\pm} /
\rho^{\pm}}
\end{equation*}
where $\bn$ is the physical unit normal at the interface. The FR correction
functions $g_L, g_R$ in the degree $N+1$ polynomial space $\polyP_{N + 1}$ are a crucial ingredient of the FR scheme and have the property
\[ g_L (- 1) = g_R (1) = 1, \qquad g_L (1) = g_R (- 1) = 0 \]
Reference~\cite{Huynh2007} gives a discussion on how the choice of correction
functions leads to equivalence between FR and variants of DG scheme. In this
work, the correction functions known as $g_2$ or $g_{\tmop{HU}}$
from~{\cite{Huynh2007}} are used since along with Gauss-Legendre-Lobatto
(GLL) solution points, they lead to an FR scheme which is equivalent to a DG scheme using the
same GLL solution and quadrature points. Once the continuous flux
approximation is obtained, the FR scheme is given by
\begin{equation}
\dv{\uebp^{\delta}}{t} + \frac{1}{J_{e, \bp}} \nabla_{\vxi} \cdot \tf_e
( \vxi_{\bp} ) = \bzero, \qquad \forall\bp \label{eq:semi.discrete.fr}
\end{equation}
where $J_{e, \bp}$ is the Jacobian of the transformation at solution points $\bx_{e,
\bp}$. The FR scheme is explicitly written as
\begin{equation}
\label{eq:fr.update.curvilinear}
\begin{split}
\dv{\uu_{e, \bp}^{\delta}}{t} &+ \frac{1}{J_{e, \bp}} \nabla_{\vxi} \cdot
\tf_e^{\delta} ( \vxi ) \\
&+ \frac{1}{J_{e, \bp}}  \sum_{i = 1}^d ( ( \tf_e \cdot
\bnr_{R,i} )^{\ast} - \tf_e^{\delta} \cdot \bnr_{R,i} ) ( \vxi_i^R
) g_R' (\xi_{p_i}) - ( ( \tf_e \cdot \bnr_{L,i} )^{\ast} -
\tf_e^{\delta} \cdot \bnr_{L,i} ) ( \vxi_i^L ) g_L' (\xi_{p_i})
= \bzero
\end{split}
\end{equation}

\subsection{Lax-Wendroff Flux Reconstruction (LWFR)}\label{sec:lwfr.curved}

The LWFR scheme is obtained by following the Lax-Wendroff procedure for
Cartesian domains~\cite{babbar2022} on the transformed
equation~\eqref{eq:transformed.conservation.law}. With $\uu^n$ denoting the
solution at time level $t = t_n$, the solution at the next time level can be
written using Taylor expansion in time as
\[ \uu^{n + 1} = \uu^n + \sum_{k \subindex 1}^{N+1} \frac{\Delta t^k}{k!}
\partial_t^{(k)}  \uu^n + O (\Delta t^{N + 2}) \]
where $N$ is the solution polynomial degree. Then, use $\uu_t = - \frac{1}{J}
\nabla_{\vxi} \cdot \tf$ from~\eqref{eq:transformed.conservation.law} to swap
a temporal derivative with a spatial derivative and retaining terms upto $O
(\Delta t^{N + 1})$
\[ \uu^{n + 1} = \uu^n - \frac{1}{J}  \sum_{k = 1}^{N+1} \frac{\Delta
t^k}{k!} \partial_t^{(k - 1)}  ( \nabla_{\vxi} \cdot \tf ) \]
Shifting indices and writing in a conservative form
\begin{equation}
\uu^{n + 1} = \uu^n - \frac{\Delta t}{J} \nabla_{\vxi} \cdot \tF
\label{eq:lw.update}
\end{equation}
where $\tF$ is a time averaged approximation of the contravaraint flux $\tf$
given by
\begin{equation}
\tF = \sum_{k = 0}^N \frac{\Delta t^k}{(k + 1) !} \partial_t^k  \tf
\approx \frac{1}{\Delta t}  \int_{t^n}^{t^{n + 1}} \tf  \ud t
\label{eq:time.averaged.flux}
\end{equation}
We first construct an element local order $N + 1$ approximation $\tF^{\delta}_e$ to
$\tF$~(Section~\ref{sec:alwp})
\[
{\tF}^\delta_e(\vxi) = \sum_{\bp} \tF_{e,\bp} \ell_{\bp} ( \vxi )
\]
and which will be in general discontinuous across the element
interfaces. Then, we construct the \tmtextit{continuous time averaged flux
approximation} by performing a correction along each direction $i$, analogous to the case of FR~\eqref{eq:cts.num.flux}, leading to
\begin{equation}
( \tF_e ( \vxi ) )^i = ( \tF_e^{\delta} (
\vxi ) )^i + ( ( \tF_e \cdot \bnr_{R,i} )^{\ast} -
\tF_e^{\delta} \cdot \bnr_{R,i} ) ( \vxi_i^R ) g_R (\xi_{p_i})
- ( ( \tF_e \cdot \bnr_{L,i} )^{\ast} - \tF_e^{\delta} \cdot
\bnr_{L,i} ) ( \vxi_i^L ) g_L (\xi_{p_i})
\label{eq:cts.time.avg.flux}
\end{equation}
where, as in~\eqref{eq:rusanov.flux}, the numerical flux~$( \tF_e \cdot
\bnr_{S, i} )^{\ast}$ is an approximation to the time average flux and
is computed by a Rusanov-type approximation,
\begin{equation}
( \tF_e \cdot \widehat{\bn}_{S, i} )^{\ast} = \frac{1}{2}  [
( \tF^{\delta} \cdot \bnr_{S, i} )^+ + ( \tF^{\delta} \cdot
\bnr_{S, i} )^- ] - \frac{\lambda_{S, i}}{2}  ( \uU^+_{S,
i} - \uU^-_{S, i} ) \label{eq:rusanov.flux.lw}
\end{equation}
where $\uU$ is the approximation of time average solution given by
\begin{equation}
\uU = \sum_{k = 0}^N \frac{\Delta t^k}{(k + 1) !} \partial_t^k  \uu
\approx \frac{1}{\Delta t}  \int_{t^n}^{t^{n + 1}} \uu  \ud t
\end{equation}
The computation of dissipative part of~\eqref{eq:rusanov.flux.lw} using the
time averaged solution instead of the solution at time $t_n$ was introduced
in~{\cite{babbar2022}} and was termed D2 dissipation. It is a natural choice
in approximating the time averaged numerical flux and does not add any
significant computational cost because the temporal derivatives of $\uu$ are
already available when computing the local approximation $\tF^{\delta}$. The
choice of D2 dissipation reduces to an upwind scheme in case of constant
advection equation and leads to enhanced Fourier CFL stability
limit~{\cite{babbar2022}}.

The Lax-Wendroff update is performed following~\eqref{eq:semi.discrete.fr}
for~\eqref{eq:lw.update}
\[ \uebp^{n + 1} = \uebp^n - \frac{\Delta t}{J_{e, \bp}} \nabla_{\vxi}
\cdot \tF_e ( \vxi_{\bp} )  \]
which can be explicitly written as
\begin{equation}
\label{eq:lwfr.update.curvilinear}
\begin{split}
\uebp^{n + 1} = \uebp^n & - \frac{\Delta t}{J_{e, \bp}}
\nabla_{\vxi} \cdot \tF^{\delta}_e ( \vxi_{\bp} ) \\
& - \frac{\Delta t}{J_{e, \bp}}  \sum_{i = 1}^d ( (
\tF_e \cdot \bnr_{R,i} )^{\ast} - \tF^{\delta}_e \cdot \bnr_{R,i} )
( \vxi_i^R ) g_R' (\xi_{p_i}) -
 ( ( \tF_e \cdot \bnr_{L,i} )^{\ast} -
\tF^{\delta}_e \cdot \bnr_{L,i} ) ( \vxi_i^L ) g_L' (\xi_{p_i})
\end{split}
\end{equation}
By multiplying~\eqref{eq:lwfr.update.curvilinear} by quadrature weights $J_{e,
\bp} w_{\bp}$ and summing over $\bp$, it is easy to see that the scheme is
conservative (see Appendix~\ref{sec:app.lwfr.conservative}) in the sense that
\begin{equation}
\overline{\uu}_e^{n + 1} = \overline{\uu}_e^n - \frac{\Delta t}{|
\Omega_e |}  \left( \sum_{i = 1}^d \int_{\Oip} ( \tF_e \cdot
\bnr_{R,i} )^{\ast}  \ud S_{\vxi} + \int_{\Oim} ( \tF_e
\cdot \bnr_{L,i} )^{\ast}  \ud S_{\vxi} \right)
\label{eq:conservation.lw}
\end{equation}
where the element mean value $\overline{\uu}_e$ is defined to be
\begin{equation}
\overline{\uu}_e = \frac{1}{| \Omega_e |}  \sum_{\bp} \uebp J_{e, \bp}
w_{\bp}
\label{defn.mean}
\end{equation}

\correction{The LWFR scheme~\eqref{eq:lwfr.update.curvilinear} gives a high order method for smooth problems, but there are many practical problems involving hyperbolic conservation laws that consist of non-smooth solutions containing shocks and other discontinuities. In such situations, using a higher order method is bound to produce Gibbs oscillations~{\cite{godunov1959}}. The cure is to non-linearly add dissipation in regions where the solution is non-smooth, with methods like artificial viscosity, limiters and switching to a robust lower order scheme; the resultant scheme will be non-linear even for linear equations. In this work, we use the blending scheme for LWFR proposed in~{\cite{babbar2023admissibility}} for Gauss-Legendre solution points. In order to be compatible with
\tmverbatim{Trixi.jl}~{\cite{Ranocha2022}} and make use of this excellent code, we introduce LWFR with blending
scheme for Gauss-Legendre-Lobatto solution points, which are also used in \tmverbatim{Trixi.jl}. As
in~{\cite{babbar2023admissibility}}, the blending scheme has to be constructed
to be provably admissibility preserving~(Definition~\ref{defn:adm.pres}). Since most of the description
is similar to~{\cite{babbar2023admissibility}}, we keep the details of the blending scheme
 in Appendix~\ref{sec:shock.capturing}.}

\subsubsection{Approximate Lax-Wendroff procedure}\label{sec:alwp}

We now illustrate how to approximate the time average flux at the solution points $\tF_{e,\bp}$ which is required to construct the element local approximation $\tF_e^\delta(\vxi)$ using the approximate Lax-Wendroff procedure~{\cite{Zorio2017}}. For $N = 1$,~\eqref{eq:time.averaged.flux} requires $\partial_t  \tf$ which is approximated as
\begin{equation}
\partial_t  \tf^{\delta}(\vxi_{\bp}) = \frac{\tf ( \uu_{\bp} + \Delta t (\uu_t)_{\bp} )
- \tf ( \uu_{\bp} - \Delta t (\uu_t)_{\bp} )}{2 \Delta t}
\label{eq:ft.defn}
\end{equation}
where element index $e$ is suppressed as all these operations are local to each element. The time index is also suppres\correction{s}ed as all quantities are used from time level $t_n$. The $\uu_t$ above is approximated using~\eqref{eq:transformed.conservation.law}
\begin{equation}
(\uu_t)_{\bp} = - \frac{1}{J_{\bp}} \nabla_{\vxi} \cdot \tf^{\delta}(\vxi_{\bp}) \label{eq:ut.defn}
\end{equation}
where $\tf^{\delta}_e$ is the cell local approximation to the flux $\tf$ given
in~\eqref{eq:flux.poly.defn}. For $N = 2$,~\eqref{eq:time.averaged.flux}
additionally requires $\partial_{t \nocomma t}  \tf$
\[
\partial_{t t}  \tf^\delta(\vxi_{\bp}) = \frac{1}{\Delta t^2}  \left[ \tf
\left( \uu_{\bp} + \Delta t (\uu_t)_{\bp} + \frac{\Delta t}{2}  (\uu_{tt})_{\bp} \right) - 2 \tf ( \uu_{\bp} ) + \tf \left( \uu_{\bp} -
\Delta t (\uu_t)_{\bp} + \frac{\Delta t}{2}  (\uu_{tt})_{\bp} \right) \right]
\]
where the element index $e$ is again suppressed. We approximate $\uu_{t \nocomma t}$ as
\begin{equation}
(\uu_{tt})_{\bp} = - \frac{1}{J_{\bp}} \nabla_{\vxi} \cdot \partial_t \tf^{\delta}(\vxi_{\bp})\label{eq:utt.defn}
\end{equation}
The procedure for other degrees will be similar and the derivatives $\nabla_\vxi$ are computed using a differentiation matrix. The implementation can be made efficient by accounting for cancellations of $\Delta t$ terms. Since this step is similar to that on Cartesian grids, the reader is referred to Section 4 of~{\cite{babbar2022}} for more details.

\subsection{Free stream preservation for LWFR}\label{sec:free.stream.lwfr}

Since the divergence in a Flux Reconstruction (FR) scheme~\eqref{eq:fr.update.curvilinear} is computed as the derivative of a polynomial, the following metric identity is required for our scheme to preserve a constant state
\begin{equation}
\sum_{i = 1}^d \pdv{}{\xii}  \InI ( J \ba^i ) = \bzero
\label{eq:metric.identity.contravariant.inter}
\end{equation}
where $\InI$ is the degree $N$ interpolation operator defined as
\begin{equation}
\InI (f)(\vxi) = \sum_{\bp} \ell_{\bp} ( \vxi ) f (
\vxi_{\bp} ) \label{eq:interpolation.defn}
\end{equation}
The study of free-stream preservation was made in~{\cite{Kopriva2006}} showing that
satisfying~\eqref{eq:metric.identity.contravariant.inter} gives free stream preservation. However, it was also shown that the identities impose additional constraints on the degree of the reference map $\Theta$. The remedy given in~\eqref{eq:metric.identity.contravariant.inter} is to replace the metric terms $J \ba^i$ by a different degree $N$ approximation $\II ( J \ba^i
)$ so that~\eqref{eq:metric.identity.contravariant.inter} reduces to
\begin{equation}
\sum_{i = 1}^d \pdv{}{\xii}  \InI  \II ( J \ba^i ) = \sum_{i =
1}^d \pdv{}{\xii}  \II ( J \ba^i ) = \bzero
\label{eq:metric.identity.contravariant.poly}
\end{equation}
In~{\cite{Kopriva2006}}, choices of $\II$ like the \tmtextit{conservative curl
form} were proposed which
ensured~\eqref{eq:metric.identity.contravariant.poly} without any additional
constraints on the degree of the reference map $\Theta$. Those choices are
only relevant in 3-D as, in 2-D, they are equivalent to interpolating $\Theta$
\correction{by a degree $N$ polynomial} before computing the metric terms which is the
choice of $\II$ we make in this work by defining the reference map as
in~\eqref{eq:reference.map}.

In this section, we show that the
identities~\eqref{eq:metric.identity.contravariant.inter} are enough to ensure
free stream preservation for LWFR. Throughout this section, we assume that the
mesh is \tmtextit{well-constructed}~{\cite{Kopriva2006}} which is a property
that follows from the natural assumption of global continuity of the reference
map \correction{defined as in~\eqref{eq:reference.map}, see Figure~\ref{fig:well.cons}}.
\begin{figure}
\centering
\begin{tabular}{ccc}
\includegraphics[width=0.315\textwidth]{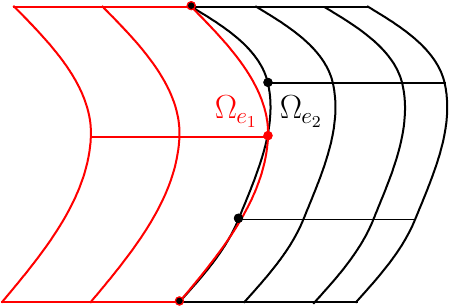} & \includegraphics[width=0.315\textwidth]{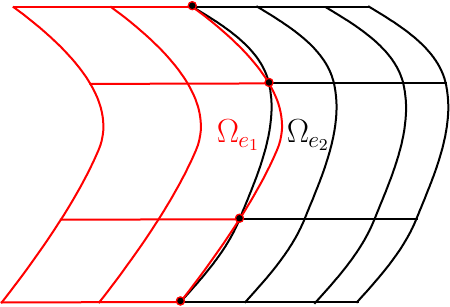} & \includegraphics[width=0.315\textwidth]{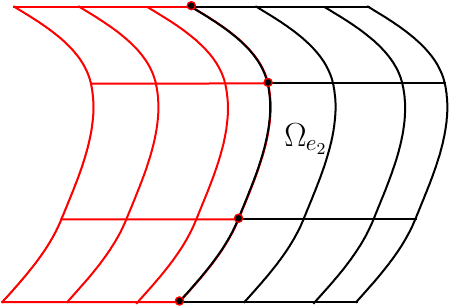}\\
(a) & (b) & (c)
\end{tabular}
\caption{\correction{(a) A non-well-constructed mesh where reference maps do not have the same solution points and polynomial degrees, contradicting~\eqref{eq:reference.map}, (b) a non-well-constructed mesh where reference maps agree at solution points but reference maps are still of different polynomial degrees contradicting~\eqref{eq:reference.map}, (c) a well-constructed mesh.} \label{fig:well.cons}}
\end{figure}
\begin{definition}
\label{defn:well.constructed.mesh}Consider a mesh where element faces in
reference element $\Oo$ are denoted as $\left\{ {\partial \Omega_{o, i}^S}
\right\}$ for coordinate directions $1 \leq i \leq d$ and $S = L / R$ chosen
so that the corresponding reference normals $\left\{ \bnr_{S, i} \right\}$
are $\bnr_{R, i} = \be_i$ and $\bnr_{L, i} = - \bnr_{R, i}$ where $\left\{
\be_i \right\}_{i = 1}^d$ is the Cartesian basis, see
Figure~\ref{fig:ref.map}. The mesh is said to be well-constructed if the
following is satisfied
\begin{equation}
\sum_{m = 1}^d ( \II ( J \ba^m )^+ - \II ( J \ba^m
)^- )  ( \bnr_{S, i} )_m = \bzero \qquad \forall 1
\leq i \leq d, \quad s = L, R \label{eq:well.constructed}
\end{equation}
where $\pm$ are used to denote trace values from $\Oo$ or from the
neighbouring element respectively.
\end{definition}

\begin{remark}
From~\eqref{eq:contravariant.identity}, the
identity~\eqref{eq:well.constructed} can be seen as a property of the
tangential derivatives of the reference map at the faces and is thus
obtained if the reference map is globally continuous. Also, since the unit
normal vector of an element at interface $i$ is given by $J \ba^i / \left\|
J \ba^i \right\|$,~\eqref{eq:well.constructed} also gives us continuity of
the unit normal across interfaces.
\end{remark}

Assuming the current solution is constant in space, $\uu^n = \bc$, we will begin by proving that the approximate time
averaged flux and solution satisfy
\begin{equation}
\tF^{\delta} = \tf^{\delta}(\bc), \qquad \uU = \uu^{\delta} = \bc
\label{eq:time.avg.is.physical.flux}
\end{equation}
For the constant physical flux $\pf ( \bc )$, the contravariant
flux $\tf$ will be
\[ \tf_i = \II ( J \ba^i ) \cdot \pf ( \bc ) = \sum_{n =
1}^d \II (Ja^i_n)  \pf_n ( \bc ) \]
Using~\eqref{eq:ut.defn}, we obtain at each solution point
\begin{align*}
\uu_t & = - \frac{1}{J} \nabla_{\vxi} \cdot \tf^{\delta} = - \frac{1}{J}
\sum_{i = 1}^d \partial_{\xii}  ( \tf^{\delta} )_i\\
& = - \frac{1}{J}  \sum_{i = 1}^d \partial_{\xii}  ( \II ( J
\ba^i ) \cdot \pf ( \bc ) ) = - \frac{1}{J}  \left(
\sum_{i = 1}^d \partial_{\xii}  \II ( J \ba^i ) \right) \cdot \pf
( \bc )\\
& = \bzero
\end{align*}
where the last equality follows from using the metric identity~\eqref{eq:metric.identity.contravariant.inter}. For polynomial degree $N = 1$, recalling \eqref{eq:ft.defn}, this proves that
\[ \partial_t  \tf^{\delta} = \frac{\tf ( \uu + \Delta t \uu_t
) - \tf ( \uu - \Delta t \uu_t )}{2 \Delta t}
= \bzero \]
Thus, we obtain
\[ \tF^{\delta} = \tf^{\delta} + \frac{\Delta t}{2} \partial_t
\tf^{\delta} = \tf^{\delta}, \qquad \uU = \uu^{\delta} + \frac{\Delta
t}{2} \partial_t  \uu^{\delta} = \uu^{\delta} \]
Building on this, for $N = 2$, by \eqref{eq:utt.defn},
\[ \uu_{t \nocomma t} = - \frac{1}{J} \nabla_{\vxi} \cdot \partial_t
\tf^{\delta} = \bzero \]
which will prove $\partial_{t \nocomma t}  \tf^{\delta} = \bzero$ and we
similarly obtain the following for all degrees
\begin{align}
\tF^{\delta} & = \sum_{k = 0}^N \frac{\Delta t^k}{(k + 1) !}
\partial_t^k  \tf^{\delta} = \tf^{\delta} = \left\{ J \ba^i \cdot \pf (
\bc ) \right\}_{i \subindex 1}^d \label{eq:Fisf} \\
\uU & = \sum_{k = 0}^N \frac{\Delta t^k}{(k + 1) !} \partial_t^k
\uu^{\delta} = \uu^{\delta} = \bc \label{eq:Uisu}
\end{align}
To prove free stream preservation, we argue that the
update~\eqref{eq:lwfr.update.curvilinear} vanishes as the volume terms
involving divergence of $\tF^{\delta}$ and the surface terms involving trace
values and numerical flux vanish. By~\eqref{eq:Fisf}, the volume terms
in~\eqref{eq:lwfr.update.curvilinear} are given by
\[ \frac{1}{J} \Delta t \left( \sum_{i = 1}^d \partial_{\xii}  \II
( J \ba^i ) \right) \cdot \pf ( \bc ) \]
and vanish by the metric
identity~\eqref{eq:metric.identity.contravariant.poly}. By~\eqref{eq:Uisu},
the dissipative part of the numerical flux~\eqref{eq:rusanov.flux.lw} is
computed with the constant solution $\uu^n = \bc$ and will thus vanish. For
the central part of the numerical flux, as the mesh is well-constructed
(Definition~\ref{defn:well.constructed.mesh}), the trace values are given by
\[ ( \tF^{\delta} \cdot \bnr_i )^+ = \sum_{m = 1}^d ( \II
( J \ba^m ) \cdot \pf ( \bc ) )^+ ( \bnr_i
)_m = \sum_{m = 1}^d ( \II ( J \ba^m ) \cdot \pf
( \bc ) )^- ( \bnr_i )_m = ( \tF^{\delta}
\cdot \bnr_i )^- \]
Thus, the numerical flux agrees with the physical flux at element interfaces, making
the surface terms in~\eqref{eq:lwfr.update.curvilinear} vanish.

\section{Adaptive mesh refinement}\label{sec:amr}

Adaptive mesh refinement helps resolve flows where the relevant features are localized to certain regions of the physical domain by increasing the mesh resolution in those regions and coarsening in the rest of the domain. In this work, we allow the adaptively refined meshes to be non-conforming, i.e., element neighbours need not have coinciding solution points at the interfaces~(Figure~\ref{fig:refine.coarsen}a). We handle the non-conformality using the \emph{mortar element method} first introduced for hyperbolic PDEs in~{\cite{Kopriva1996}}. {\correction{The procedure is described for general curvilinear meshes defined through element local reference maps~\eqref{eq:reference.map}. During the refinement stage, the reference maps for the finer elements are obtained by interpolation. The metric terms~\eqref{eq:contravariant.identity} and Jacobians are then recomputed.}}

In order to perform the transfer of solution during coarsening and refinement, we introduce some notations and operators. Define the 1-D reference elements
\begin{equation}
I_0 = [- 1, 0], \qquad I_1 = [0, 1], \qquad I = [- 1, 1], \qquad  \Nod = \{
0, 1 \}^d \label{eq:I0.defn}
\end{equation}
and the bijections $\phi_s : I_s \rightarrow I$ for $s = 0, 1$ as
\begin{equation}
\label{eq:submap}
\sphi_0 (\xi) = 2 \xi + 1, \quad \xi \in I_0, \qquad \sphi_1 (\xi) = 2
\xi - 1, \quad \xi \in I_1
\end{equation}
so that the inverse maps $\phi_s^{-1} : I \rightarrow I_s$ are given by
\begin{equation}
\label{eq:submap.inverse}
\sphi_0^{- 1} (\xi) = \frac{\xi - 1}{2}, \quad \xi \in I, \qquad
\sphi_1^{- 1} (\xi) = \frac{\xi + 1}{2}, \quad \xi \in I
\end{equation}
Denoting the 1-D solution points and Lagrange basis for $I$ as $\{ \xi_p \}_{p
\subindex 0}^N$ and $\{ \ell_p (\xi) \}_{p \subindex 0}^N$ respectively, the same
for $I_s$ are given by $\left\{ \sphi_s^{- 1} (\xi_p) \right\}_{p \subindex
0}^N$ and $\left\{ \ell_p ( \sphi_s (\xi) ) \right\}_{p \subindex 0}^N$
respectively. We also define $\qint$ to be integration under quadrature at solution points. Thus,
\[
\qint_I u(\xi) \ud \xi = \sum_{p \subindex 0}^N u(\xi_p) w_p, \qquad \qint_{I_s} u(\xi) \ud \xi = \sum_{p \subindex 0}^N 2 u(\sphi_s^{-1}(\xi_p)) w_p
\]

In order to get the solution point values of the refined
elements, we will perform interpolation. All integrals in this section are
approximated by quadrature at solution points which are the degree $N$
Gauss-Legendre-Lobatto points. The interpolation operator from $I$ to $\{ I_s
\}_{s \subindex 0, 1}$ is given by $\Vi_{\X_s}$ defined as the Vandermonde
matrix corresponding to the Lagrange basis
\begin{equation}
( \Vi_s )_{p \nocomma q} = \ell_q ( \sphi_s^{- 1} (\xi_p)
), \qquad 0 \leq p, q \leq N, \quad s = 0, 1 \label{eq:subint.op}
\end{equation}
For the process of coarsening, we also define the $L^2$ projection operators
$\left\{ \proP^s \right\}_{s \subindex 0, 1}$ which projects a polynomial $u$
defined on the Lagrange basis of $I_s$ to the Lagrange basis of $I$ as
\[
\qint_I \proP^s (u)(\xi) \ell_i (\xi)  \ud \xi = \qint_{I_s} u(\xi) \ell_i (\xi)  \ud
\xi \qquad 0 \leq i \leq N
\]
Approximating the integrals by quadrature on solution points, we obtain the
matrix representations corresponding to the basis
\begin{equation}
\proP^s_{p \nocomma q} = \frac{1}{2}  \frac{w_q}{w_p} \ell_p (
\sphi_s^{- 1} (\xi_q) ) \qquad 0 \leq i, j \leq N, \quad s = 0, 1
\label{eq:proP.defn}
\end{equation}
where $\{ w_p \}_{p \subindex 0}^N$ are the quadrature weights corresponding
to solution points. The transfer of solution during coarsening and refinement is performed by matrix vector operations using the operators~(\ref{eq:subint.op},~\ref{eq:proP.defn}). Thus, the operators~(\ref{eq:subint.op},~\ref{eq:proP.defn}) are stored as matrices for reference element at the beginning of the simulation and reused for the adaptation operations in all elements. Lastly, we introduce the notation of a product of matrix
operators $\{ A_i \}_{i = 1}^d$ acting on $\tmmathbf{b} = ( b_{\bp}
)_{\bp \in \Nnd} = (b_{p_1 \nocomma p_2 \nocomma p_3})_{\bp \in \Nnd}$
as
\begin{equation}
(A_i  \tmmathbf{b})_{\bp} = \sum_{\bq \in \mathbb{N}_N^d} \left( \prod_{i =
1}^d (A_i)_{p_i \nocomma q_i} \right) b_{\bq}
\label{eq:product.of.operators}
\end{equation}

\subsection{Solution transfer between element and subelements}\label{sec:soln.transfer}
\begin{figure}
\centering
\begin{tabular}{c}
\includegraphics[width = 0.5\textwidth] {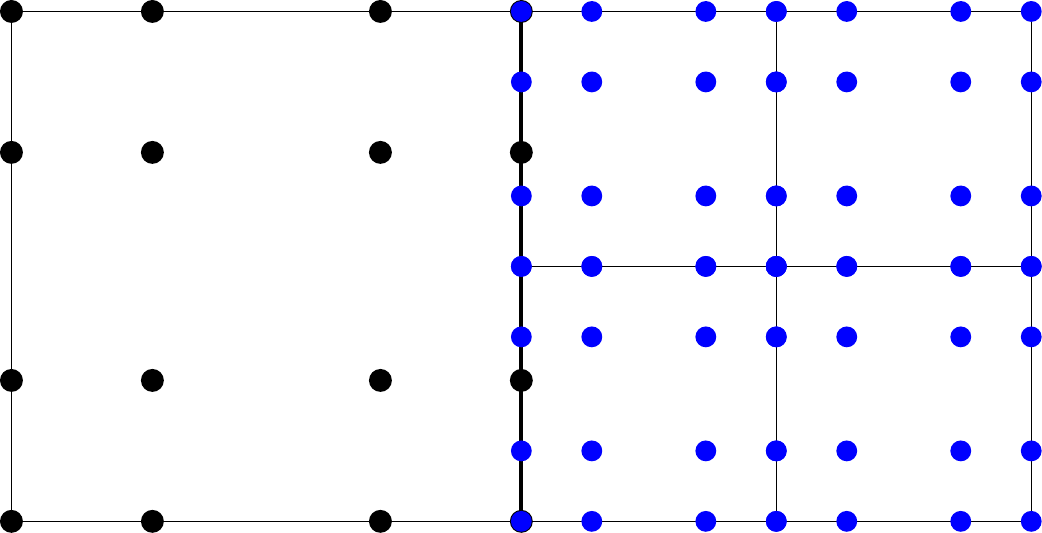} \\
(a) \\
\includegraphics[width = 0.65\textwidth]{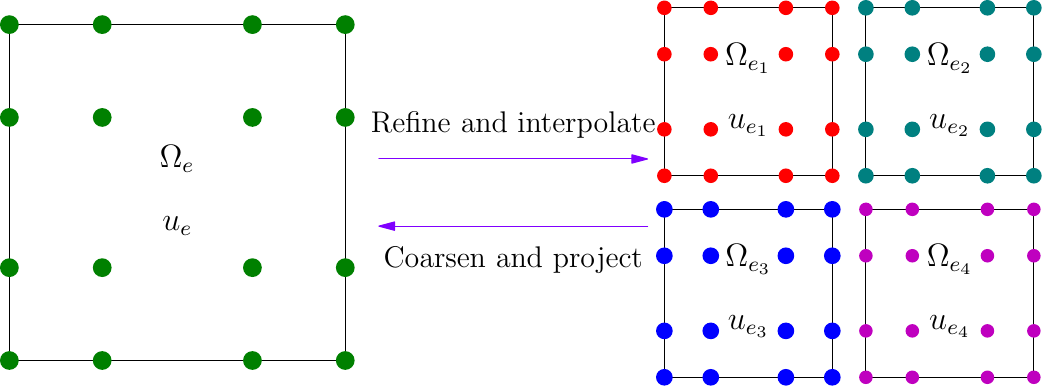} \\
(b)
\end{tabular}
\caption{(a) Neighbouring elements with hanging nodes (b) Illustration of refinement and coarsening \label{fig:refine.coarsen}}
\end{figure}

Corresponding to the element $\Omega_e$, we denote the $2^d$ subdivisions as~(Figure~\ref{fig:refine.coarsen}b)
\[
\Omega_{e_{{\bss}}} =
\Theta_e \left ( \prod_{i = 1}^d I_{s_i} \right ), \qquad \forall \bss \in \Nod
\]
where $I_s$ are defined in~\eqref{eq:I0.defn}. We also define
$\sphi_{\bss} ( \vxi ) = ( \sphi_{s_i} ( \xii )
)_{i = 1}^d$ so that $\sphi_{\bss}$ is a bijection between
$\Omega_{e_{\bss}}$ and $\Omega_e$. Recall that ${\left\{ \ell_{\bp}
\right\}_{\bp \in \Nnd}} $ are Lagrange polynomials of degree $N$ with
variables $\vxi = ( \xii )_{i = 1}^d$. Thus, the reference solution
points and Lagrange basis for $\Omega_{e_{\bss}}$ are given by $\left\{
\sphi_{\bss}^{- 1} ( \vxi_{\bp} ) \right\}_{\bp \in
\mathbb{N}_N^d}$ and ${\left\{ \ell_{\bp} ( \sphi_{\bss} ( \vxi )
) \right\}_{\bp \in \Nnd}} $, respectively. \correction{The conservation property of the LWFR scheme was mentioned in terms of evolution of the element mean approximations~\eqref{defn.mean} in~\eqref{eq:conservation.lw}. This property needs to be maintained during the adaptive mesh refinement, i.e., we must have the following during refinement and coarsening
\[
\sum_{\bss \in \Nod} \qint_{\Omega_{e_{\bss}}} \uu_{e_{\bss}} \ud \bx = \qint_{\Omega_e} \uu_e \ud x
\]
As shall be seen later, in order to maintain conservation,
transfer of solution during refinement and coarsening needs to be performed for the \textit{transformed solution} $\tu$~\eqref{eq:transformed.conservation.law} instead of the physical solution $\uu$.} The respective
representations of solution approximations in $\Omega_e, \Omega_{e_{\bss}}$ in
reference coordinates are thus given by
\begin{equation}
\uu_e (\vxi) = \sum_{\bq \in \Nnd} \ell_{\bq} (\vxi)
\uu_{e, \bq}, \qquad \uu_{e_{\bss}} ( \vxi ) =
\sum_{\bq \in \Nnd} \ell_{\bq} ( \sphi_{\bss} ( \vxi ) )
\uu_{e_{\bss}, \bq} \label{eq:ueues.defn}
\end{equation}
\correction{The respective transformed solutions are given by
\begin{align}
\tu_e (\vxi) &= \sum_{\bq \in \Nnd} \ell_{\bq} (\vxi)
J_{e, \bq} \uu_{e, \bq} =: \sum_{\bq \in \Nnd} \ell_{\bq} (\vxi)
\tu_{e, \bq} \label{eq:tue.defn} \\
\tu_{e_{\bss}} ( \vxi ) &=
\sum_{\bq \in \Nnd} \ell_{\bq} ( \sphi_{\bss} ( \vxi ) )
J_{e_{\bss}, \bq}\uu_{e_{\bss}, \bq} =: \sum_{\bq \in \Nnd} \ell_{\bq} ( \sphi_{\bss} ( \vxi ) )
\tu_{e_{\bss}, \bq} \label{eq:tues.defn}
\end{align}
}

\subsubsection{Interpolation for refinement}

After refining an element $\Omega_e$ into child elements $\left\{ \Omega_{e_{\bss}}
\right\}_{\bss \in \mathbb{N}_1^d}$, the solution $\uu_e$ has to be
interpolated on the solution points of child elements to obtain $\left\{
\uu_{e_{\bss}} \right\}_{\bss \in \mathbb{N}_1^d}$. The scheme will be
specified by writing $\uu_{e_{\bss}, \bq}$ in terms of $\uu_{e, \bq}$, which
were defined in~\eqref{eq:ueues.defn}. \correction{In order to maintain conservation, as a first step, we
construct the polynomial approximation of the transformed solution $\tu_e$ as in~\eqref{eq:tue.defn}.
The interpolation is performed to determine the transformed solution on the refined grid as
\begin{equation}
\begin{split}
\tu_{e_{\bss}, \bp} & = \sum_{\bq \in \Nnd} \ell_{\bq} (
\sphi_{\bss}^{- 1} ( \vxi_{\bp} ) )  \tu_{e, \bq} =
\sum_{\bq \in \Nnd} \left( \prod_{i = 1}^d \ell_{q_i} (
\sphi_{s_i}^{- 1} ( \vxi_{p_i} ) ) \right)  \tu_{e, \bq}
\\
& = \sum_{\bq \in \Nnd} \left( \prod_{i = 1}^d ( \Vi_{\X_{s_i}}
)_{p_i \nocomma q_i} \right)  \tu_{e, \bq}
\end{split} \label{eq:interpolation.ue}
\end{equation}
In the product of operators notation~\eqref{eq:product.of.operators}, the
interpolation can be written as
\[
\tu_{e_{\bss}} = \left( \prod_{i = 1}^d \Vi_{\X_{s_i}} \right)  \tu_e
\]
The solution values $\uu_{e_{\bss}, \bq}$~\eqref{eq:ueues.defn} are then obtained simply as
\[
\uu_{e_{\bss}, \bq} = \frac{1}{J_{e_{\bss}, \bq}} \tu_{e_{\bss}, \bq}
\]
The refinement step maintains solution conservation since
\begin{align*}
\sum_{\bss \in \Nod} \qint_{\Omega_{e_{\bss}}} \uu_{e_{\bss}} \ud \bx = \sum_{\bss \in \Nod} \qint_{\Theta_e^{-1}(\Omega_{e_{\bss}})} \tu_{e_{\bss}} \ud \xi &= \sum_{\bss \in \Nod} \qint_{\Theta_e^{-1}(\Omega_{e_{\bss}})} \tu_e \ud \xi = \qint_{\Omega_e} \uu_e \ud x
\end{align*}
}
\subsubsection{Projection for coarsening}\label{sec:proj.elem}
When $2^d$ elements are joined into one single bigger element $\Omega_e$, the solution transfer is performed using $L^2$ projection of $\left\{ \uu_{{e_{\bss}} } \right\}_{\bss \in \Nod}$ into
$\uu_e$. \correction{As in the case of refinement, projection will be performed through the transformed solution $\tu$ which we construct on the smaller element to obtain $\tu_{\be_{\bss}}$~\eqref{eq:tues.defn}, and then project these to obtain $\tu_{\be}$~\eqref{eq:tue.defn}. The projection is performed to obtain $\tu_e$ to satisfy
\begin{equation}
\sum_{\bss \in \Nod} \qint_{\Omega_{e_{\bss}}} \tu_{e_{\bss}} \ell_{\bp}
( \vxi ) \ud \bx = \qint_{\Omega_e} \tu_e \ell_{\bp} (\vxi)  \ud \bx, \qquad \forall \bp \in \Nnd \label{eq:4to1proj}
\end{equation}
Substituting~(\ref{eq:tue.defn},~\ref{eq:tues.defn}) into~\eqref{eq:4to1proj} gives
\begin{equation}
\sum_{\bss \in \Nod} \sum_{\bq \in \Nnd} \qint_{\Omega_{e_{\bss}}} \ell_{\bp}
( \vxi ) \ell_{\bq} ( \sphi_{\bss} ( \vxi )
)  \tu_{e_{\bss}, \bq}  \ud \bx = \sum_{\bq \in \Nnd} \qint_{\Omega_e}
\ell_{\bp} ( \vxi ) \ell_{\bq} ( \vxi )  \tu_{e, \bq}
\ud \bx \label{eq:pro4.with.lagrange}
\end{equation}
}
Note the 1-D identities
\begin{align*}
\qint_I \ell_p (\xi) \ell_q (\xi)  \ud \xi & = \delta_{p \nocomma q} w_p\\
\qint_{I_s} \ell_p (\xi) \ell_q ( \sphi_s (\xi) ) \ud \xi & =
\frac{1}{2}  \qint_{- 1}^1 \ell_p ( \sphi_s^{- 1} (\xi) ) \ell_q
(\xi)  \ud \xi = \frac{1}{2} \ell_p ( \sphi_s^{- 1} (\xi_q) ) w_q
=\mathcal{P}^s_{p \nocomma q} w_p
\end{align*}
where the projection operator $\left\{ \proP^s \right\}_{s = 0, 1}$ is defined
in~\eqref{eq:proP.defn}. Then, by change of variables, we have the following
\begin{align}
\qint_{\Omega_e} \ell_{\bp} ( \vxi ) \ell_{\bq} ( \vxi
) = J_{e, \bp}  \prod_{i = 1}^d w_{p_i} \delta_{p_i \nocomma q_i}, \qquad
\qint_{\Omega_{e_{\bss}}} \ell_{\bp} ( \vxi ) \ell_{\bq} (
\sphi_{\bss} ( \vxi ) ) = J_{e, \bp}  \prod_{i = 1}^d
w_{p_i} \mathcal{P}^{s_i}_{p_i \nocomma q_i} \label{eq:integral.identities}
\end{align}
Using~\eqref{eq:integral.identities} in~\eqref{eq:pro4.with.lagrange} and dividing both sides by $J_{e, \bp}$ gives
\[
\correction{\tu}_{e, \bp} = \sum_{\bss \in \Nod} \sum_{\bq \in \Nnd} \left( \prod_{i =
1}^d \mathcal{P}^{s_i}_{p_i \nocomma q_i} \right)  \correction{\tu}_{e_{\bss}, \bq} = \sum_{\bss \in \Nod} \left( \prod_{i = 1}^d \mathcal{P}^{s_i} \right)  \uu_{e_{\bss}}
\]
where the last equation follows using the product of operators notation~\eqref{eq:product.of.operators}. \correction{As in the interpolation case, the physical solution values $\uu_e$~\eqref{eq:ueues.defn} are then obtained as
\[
\uu_{e, \bp} = \frac{1}{J_{e, \bp}} \tu_{e, \bp}
\]
To show that refinement maintains conservation, we sum the identity~\eqref{eq:pro4.with.lagrange} over $\bp$ to get
\begin{align*}
\sum_{\bss \in \Nod} \qint_{\Omega_{e_{\bss}}} \tu_{e_{\bss}} \ud \bx = \qint_{\Omega_e} \tu_e \ud \bx
\end{align*}
}
\subsection{Mortar element method (MEM)}\label{sec:handling.mortars}

\subsubsection{Motivation and notation}

\begin{figure}
\centering
\begin{tabular}{c}
\includegraphics[width=0.85\textwidth]{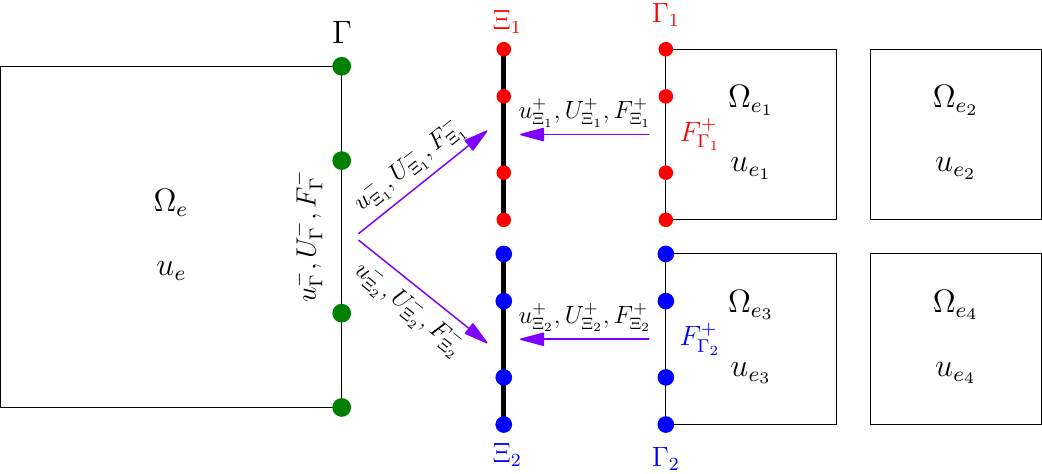}\\
(a)\\
\includegraphics[width=0.85\textwidth]{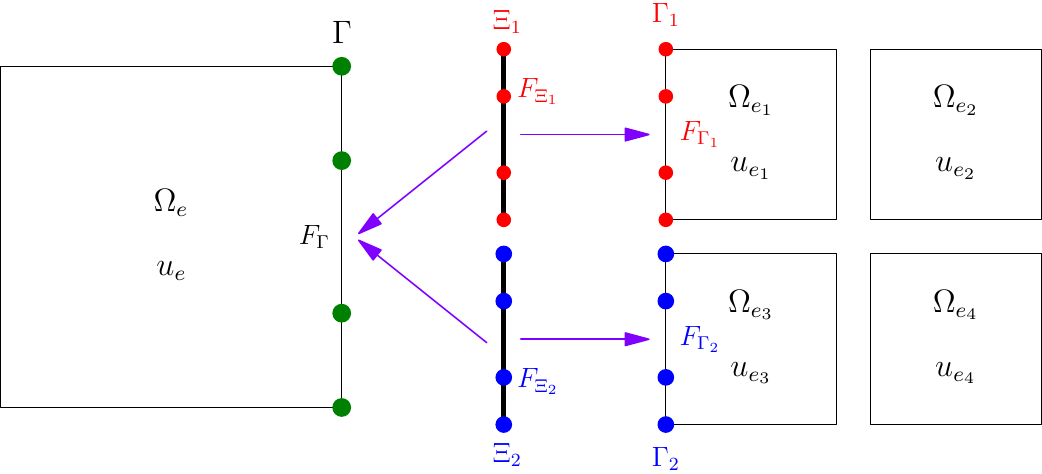}\\
(b)
\end{tabular}
\caption{(a) Prolongation to mortar and computation of numerical flux $\F_{\Xi_1}, \F_{\Xi_2}$, (b) Projection of numerical flux to interfaces\label{fig:mortar}}
\end{figure}

When the mesh is adaptively refined, there will be elements with different refinement levels sharing a face; in this work, we assume that the refinement levels of those elements only differs by 2~(Figure~\ref{fig:refine.coarsen}a). Since the neighbouring elements do not have a common face, the solution points on their faces do not coincide~(Figure~\ref{fig:mortar}). We will use the Mortar Element Method (MEM) for computing the numerical flux at all the required points on such a face, while preserving accuracy and the conservative property~\eqref{eq:conservation.lw}. There are two steps to the method

\begin{enumerate}
\item Prolong~$\tF^{\delta} \cdot \bnr_{S, i}, \uU_{S,
i}, \uu_{S,i}$~\eqref{eq:rusanov.flux.lw} from the neighbouring elements to a set of
common solution points known as mortar solution points~(Figure~\ref{fig:mortar}a).
\item Compute the numerical flux at the mortar solution points as in~\eqref{eq:rusanov.flux.lw} and map it back to the interfaces~(Figure~\ref{fig:mortar}b).
\end{enumerate}
In Sections~\ref{sec:el2mortar}, \ref{sec:proj.face}, we will explain these two steps through the specific case of Figure~\ref{fig:mortar} and we first introduce notations for the same.

Consider the multi-indices $\bss \in \mathbb{N}_1^{d - 1} = \{0, 1\}^{d - 1}$
and the interface in right (positive) $i = 1$ direction of element $\Omega_e$,
denoted as $\G$ (Figure~\ref{fig:mortar}). We assume that the elements
neighbouring $\Omega_e$ at the interface $\Gamma$ are finer and thus we have
non-conforming subinterfaces $\left\{ \G_{\bss} \right\}_{\bss \in
\mathbb{N}_1^{d - 1}}$ which, by continuity of the reference map, can be
written as $\G_{\bss} = \Theta_e ( \{ 1 \} \times \prod_{i = 1}^{d - 1}
I_{s_i} ) = \Theta_e ( \{ 1 \} \times \sphi_{\bss}^{- 1} (I^{d -
1}) )$. Thus, in reference coordinates, $\sphi_{\bss}$~\eqref{eq:submap}
is a bijection from $\Gamma_{\bss}$ to $\Gamma$. The interface $\Gamma$ can be
parametrized as $\by = \bga ( \bet ) = \Theta_e ( 1, \bet
)$ for $\bet \in I^{d - 1}$ and thus the reference variable of interface
is denoted $\bet = \bga^{- 1} ( \by )$. The subinterfaces can also
be written by using the same parametrization so that $\Gamma_{\bss} = \left\{
\gamma ( \bet ) : \bet \in \prod_{i = 1}^{d - 1} I_{s_i} \right\}$.
For the reference solution points on $\G$ being $\left\{ \bet_{\bss}
\right\}_{\bss \in \mathbb{N}_1^{d - 1}}$, the solution points in
$\Gamma_{\bss}$ are respectively given by $\left\{ \sphi_{\bss}^{- 1} (
\bet_{\bp} ) \right\}_{\bp \in N_1^{d - 1}}$ and for $\left\{ \ell_{\bp}
( \bet ) \right\}_{\bp \in \mathbb{N}_N^{d - 1}}$ being Lagrange
polynomials in $\G$, the Lagrange polynomials in $\Gamma_{\bss}$ are given by
$\left\{ \ell_{\bp} ( \sphi_{\bss} ( \bet ) )
\right\}_{\bp \in \mathbb{N}_1^{d - 1}}$ respectively. Since the solution
points between $\G$ and $\G_{\bss}$ do not coincide, they will be mapped to
common solution points in the mortars $\X_{\bss}$ and then back to $\G,
\G_{\bss}$ after computing the common numerical flux. The solution points in $\X_{\bss}$ are actually given by $\left\{ \sphi_{\bss}^{- 1} (
\bet_{\bp} ) \right\}_{\bp \in N_1^{d - 1}}$, i.e., they are the same as $\G_{\bss}$. The quantities with
subscripts $\X_{\bss}^-, \X_{\bss}^+$ will denote trace values from larger,
smaller elements respectively.

\subsubsection{Prolongation to mortars}\label{sec:el2mortar}

We will explain the prolongation procedure for a quantity $\F$ which could be the normal flux $\tF^{\delta} \cdot \bnr_{S, i}$, time average solution $\uU_{S,
i}$ or the solution $\uu_{S,i}$. The first step of MEM of mapping of solution point values from solution points at element interfaces $\G, \G_{\bss}$
to solution points at mortars $\X_{\bss}^-, \X_{\bss}^+$ is known as prolongation. The prolongation of $\{
\F^{\delta}_{\G_{\bss}} \}_{\bss \in \mathbb{N}_1^{d - 1}}$ from small elements $\G_{\bss}$ to mortar values $\{ \F_{\X_{\bss}^+} \}_{\bss \in \mathbb{N}_1^{d -
1}}$ is the identity map since both have the same solution points, and the prolongation of $\F_{\G}^\delta$ from the large interface $\G$ to the $\{\F_{\Xi_{\bss}}^-\}_{\bss \in
\mathbb{N}_1^{d - 1}}$ is an interpolation to the mortar solution points. Accuracy is maintained by the interpolation as the mortar elements are finer. Below, we explain the matrix operations used to perform the interpolation.

The prolongation of $\{
\F^{\delta}_{\G_{\bss}} \}_{\bss \in \mathbb{N}_1^{d - 1}}$ to the
mortar values $\{ \F_{\X_{\bss}^+} \}_{\bss \in \mathbb{N}_1^{d -
1}}$ is the identity map. The $\{ \F_{\Xi_{\bss}}^- \}_{\bss \in
\mathbb{N}_1^{d - 1}}$ in Lagrange basis are given by
\begin{equation}
\F_{\Xi_{\bss}^-} ( \bet ) = \sum_{\bp \in \mathbb{N}_N^{d -
1}} \ell_{\bp} ( \sphi_{\bss} ( \bet ) )
\F_{\X_{\bss}^-, \bp}, \qquad \bet \in \prod_{i = 1}^{d - 1} I_{s_i}
\label{eq:uXidefn}
\end{equation}
The coefficients $\{\F_{\X_{\bss}, \bp}^-\}_{\bp \in
\mathbb{N}_N^{d - 1}}$ are computed by interpolation
\begin{align}
\F_{\X_{\bss}^-, \bp} = \F_{\G^-}  ( \sphi_{\bss}^{- 1} (
\bet_{\bp} ) ) = & \sum_{\bq \in \mathbb{N}_N^{d - 1}}
\ell_{\bq} ( \sphi_{\bss}^{- 1} ( \bet_{\bp} ) )
\F_{\G}^{\delta} ( \bet_{\bq} ) \nonumber
=  \sum_{\bq \in \mathbb{N}_N^{d - 1}} \left( \prod_{i = 1}^{d - 1}
\ell_{q_i} ( \sphi_{s_i}^{- 1} ( \bet_{p_i} ) )
\right )  \F_{\G}^{\delta} ( \bet_{\bq} ) \nonumber\\
= & \sum_{\bq \in \mathbb{N}_N^{d - 1}}
\left( \prod_{i = 1}^{d - 1} ( \Vi_{s_i} )_{p_i \nocomma q_i} \right)  \F_{\G}^{\delta} ( \bet_{\bq} )
\label{eq:final.int.identity.fxi}
\end{align}
where the interpolation operators $\left\{ \Vi_{\X_s} \right\}_{s = 0, 1}$
were defined in~\eqref{eq:subint.op}. Using the product of operators
notation~\eqref{eq:product.of.operators}, we can compactly
write~\eqref{eq:final.int.identity.fxi} as
\begin{equation}
\F_{\X_{\bss}^-} = \left ( \prod_{i = 1}^{d - 1} \Vi_{s_i}
\right )  \F_{\G}^{\delta} \label{eq:compact.prolongation}
\end{equation}
The same procedure is performed for obtaining $\uU_{\X_{\bss}^{\pm}},
\uu_{\X_{\bss}^{\pm}}$. The numerical fluxes $\left\{ \F_{\X_{\bss}}^{\ast}
\right\}_{\bss {\in \mathbb{N}_1^{d - 1}} }$ are then computed as
in~\eqref{eq:rusanov.flux.lw}.

\subsubsection{Projection of numerical fluxes from mortars to
faces}\label{sec:proj.face}

In this section, we use the notation $\F^{\ast} := (\tF_e \cdot \bnr_{S,i} )^{\ast}$ to denote the numerical flux~\eqref{eq:rusanov.flux.lw}. In the second step of MEM, the numerical fluxes $\{ \F^{\ast}_{\X_{\bss}} \}_{\bss {\in \mathbb{N}_1^{d - 1}} }$ computed using values at $\{\Xi^\pm_{\bss}\}_{\bss {\in \mathbb{N}_1^{d - 1}}}$ are mapped back to interfaces $\G_{\bss}, \G$. Since the solution points on $\G_{\bss}$ are the same as those of $\Xi^\pm_{\bss}$, the mapping from $\{ \F^{\ast}_{\X_{\bss}} \}_{\bss {\in \mathbb{N}_1^{d - 1}} }$ to $\{ \F^{\ast}_{\G_{\bss}} \}_{\bss {\in \mathbb{N}_1^{d - 1}} }$ is the identity map. In order to maintain the conservation property, an $L^2$ projection is performed to map all the fluxes $\{ \F^{\ast}_{\X_{\bss}} \}_{\bss {\in \mathbb{N}_1^{d - 1}} }$ into one numerical flux $\F_{\G}^{\ast}$ on the larger interface.

An $L^2$ projection of these fluxes to $\F_{\G}^{\ast}$ on $\G$
is performed as
\begin{equation}
\sum_{\bss \in \mathbb{N}_1^{d - 1}} \qint_{\G_{\bss}} \F_{\X_{\bss}}^{\ast}
\ell_{\bp} = \qint_{\G} \F_{\G}^{\ast} \ell_{\bp}, \qquad \forall \bp \in
\mathbb{N}_N^{d - 1} \label{eq:proj.mortar.int.id}
\end{equation}
where integrals are computed with quadrature at solution points. As
in~\eqref{eq:uXidefn}, we write the mortar fluxes as
\begin{align*}
\F_{\Xi_{\bss}}^{\ast} ( \bet ) & = \sum_{\bq \in
\mathbb{N}_1^{d - 1}} \ell_{\bq} ( \sphi_{\bss} ( \bet )
)  \F_{\X_{\bss}, \bq}^{\ast}, \qquad \bet \in \X_{\bss} \\
\F_{\G}^{\ast} ( \bet ) & = \sum_{\bq \in \mathbb{N}_1^{d - 1}}
\ell_{\bq} ( \bet )  \F^{\ast}_{\G, \bq}, \qquad \bet \in \G
\end{align*}
Thus, the integral identity~\eqref{eq:proj.mortar.int.id} can be written as
\begin{equation}
\sum_{\bss \in \mathbb{N}_1^{d - 1}} \sum_{\bq \in \mathbb{N}_N^{d - 1}}
\qint_{\G_{\bss}} \ell_{\bp} ( \bet ) \ell_{\bq} (
\sphi_{\bss} ( \bet ) )  \F_{\X_{\bss}, \bq}^{\ast} =
\sum_{\bq \in \mathbb{N}_{\bp}^{d - 1}} \qint_{\G} \ell_{\bp} ( \bet
) \ell_{\bq} ( \bet )  \F^{\ast}_{\G, \bq}, \qquad \forall
\bp \in \mathbb{N}_N^{d - 1} \label{eq:projection.integral.equation}
\end{equation}
Using the identities~\eqref{eq:projection.integral.equation}, the
equations~\eqref{eq:projection.integral.equation} become
\begin{equation*}
\sum_{\bss \in \mathbb{N}_1^{d - 1}} \sum_{\bq \in \mathbb{N}_N^{d - 1}}
\left( \prod_{i = 1}^{d - 1} w_{p_i} \mathcal{P}^{s_i}_{p_i \nocomma q_i}
\right)  \F_{\X_{\bss}, \bq}^{\ast} J_{e, \bp}^S = w_{\bp}  \F^{\ast}_{\G,
\bp} J_{e, \bp}^S
\end{equation*}
where $J_{e, \bp}^S$ is the surface Jacobian, given by $||(J \ba^1)_{e,\bp}||$ in this case ((6.29) of~\cite{kopriva2009}). Then, dividing both sides by
$J_{e, \bp}^S w_{\bp}$ gives
\begin{equation}
\F^{\ast}_{\G, \bp} = \sum_{\bss \in \mathbb{N}_1^{d - 1}} \sum_{\bq \in \mathbb{N}_N^{d - 1}} \left( \prod_{i = 1}^{d - 1} \mathcal{P}^{s_i}_{p_i \nocomma q_i} \right)  \F_{\X_{\bss}, \bq}^{\ast} = \sum_{\bss \in \mathbb{N}_1^{d - 1}} \left( \prod_{i =
1}^{d - 1} \mathcal{P}^{s_i} \right)  \F_{\Xi_{\bss}}^{\ast}
\end{equation}
where the last identity is obtained by the product of operators notation~\eqref{eq:product.of.operators}. Note that the identity~\eqref{eq:proj.mortar.int.id} implies
\begin{equation*}
\sum_{\bss \in \mathbb{N}_1^{d - 1}} \qint_{\G_{\bss}} \F_{\X_{\bss}}^{\ast}
v = \qint_{\G} \F_{\G}^{\ast} v, \qquad v \in \polyP_N
\end{equation*}
Then, taking $v=1$ shows that the total fluxes over an interface $\Gamma$ are the same as over $\{\Gamma_{\bss}\}_{\bss \in \mathbb{N}_1^{d-1}}$ and thus the conservation property~\eqref{eq:conservation.lw} of LWFR is maintained by the LWFR scheme.

\begin{remark}[Freestream and admissibility preservation under AMR] \label{rmk:amr.metric.terms}
Under the adaptively refined meshes, free stream preservation and provable admissibility preservation are respectively ensured.
\begin{enumerate}
\item When refining/coarsening, there are two ways to compute the metric terms -
interpolate/project the metric terms directly or interpolate/project the
reference map $\Theta$ at solution points and use the newly obtained reference
map to recompute the metric terms. The latter, which is the approach taken in this work, can lead to violation of free
stream preservation as we can have $( \II )_{e_L} ( J \ba^i
) \neq ( \II )_{e_R} ( J \ba^i )$ where $\Omega_{e_L}$ and
$\Omega_{e_S}$ are two neighbouring large and small elements respectively. Thus, the
interface terms may not vanish in the update
equation~\eqref{eq:lwfr.update.curvilinear} with constant $\uu^n$ leading to a
violation of free stream preservation. This issue only occurs in 3-D and is
thus beyond the scope of this work, but some remedies are to
interpolate/project the metric terms when refining/coarsening or to use the
reference map $\Theta \in \polyP_{N / 2}$, as explained
in~{\cite{Kopriva2019}}. Another solution has been studied
in~{\cite{Kozdon2018}} where a common finite element space with mixed degree
$N - 1$ and $N$ is used with continuity at the non-conformal interfaces. Since this work only deals with problems in 2-D, we always have$( \II )_{e_L} ( J \ba^i
) = ( \II )_{e_R} ( J \ba^i )$ ensuring that the interface terms in~\eqref{eq:lwfr.update.curvilinear} vanish when $\uu = \bc$. Further, since the metric terms are recomputed in this work, the volume terms will vanish by the same arguments as in~Section~\ref{sec:free.stream.lwfr}. Thus, free stream preservation is maintained even with the non-conformal, adaptively refined meshes.

\item The flux limiting explained in Section~\ref{sec:flux.correction} ensures admissibility in means (Definition~\ref{defn:mean.pres}) and then uses the scaling limiter of~\cite{Zhang2010b} to enforce admissibility of solution polynomial at all solution points to obtain an admissibility preserving scheme (Definition~\ref{defn:adm.pres}). However, the procedure does not ensure that the polynomial is admissible at points which are not the solution points. Adaptive mesh refinement introduces such points into the numerical method and can thus cause a failure of admissibility preservation in the following situations: (a) mortar solution values $\{ \uu_{\X_{\bss}}^- \}$ obtained by interpolation as in~\eqref{eq:uXidefn} are not admissible, (b) mean values $\{ \overline{\uu}_{e_{\bss}} \}$ of the solution values $\{ \uu_{e_{\bss}} \}$ obtained by interpolating from the larger element as in~\eqref{eq:interpolation.ue} are not admissible. Since the scaling limiter~\cite{Zhang2010b} can be used to enforce admissibility of solution at any desired points, the remedy to both the issues is further scaling;  we simply perform scaling of solution point values $\{ \uu_{\X_{\bss}}^- \}, \{ \uu_{e_{\bss}} \}$ with the admissible mean value $\overline{\uu}_e$. This will ensure that the mortar solution point values and the mean values $\{\uu_{e_{\bss}} \}$ are admissible.
\end{enumerate}
\end{remark}

\subsection{AMR indicators}\label{sec:amr.indicator}

The process of adaptively refining and coarsening the mesh requires a solution
smoothness indicator. In this work, two smoothness indicators have been used
for adaptive mesh refinement. The first is the indicator
of~{\cite{henneman2021}}, explained in Section~\ref{sec:smooth.ind}. The
second is L{\"o}hner's smoothness indicator~{\cite{lohner1987}} which uses the central finite difference formula for second derivative, which is given by
\begin{equation*}
\begin{gathered}
\alpha_e = \max_{\bp \in \Nnd} \max_{1 \leq i \leq d} \frac{\left|
q(\uu_{\bp_{i +}}) - 2 q(\uu_{\bp}) + q(\uu_{\bp_{i -}}) \right|}{\left| q(\uu_{\bp_{i
+}}) - q(\uu_{\bp}) \right| + \left| q(\uu_{\bp}) - q(\uu_{\bp_{i -}}) \right| +
f_{\tmop{wave}}  ( \left| q(\uu_{\bp_{i +}}) \right| + 2 \left| q(\uu_{\bp})
\right| + \left| q(\uu_{\bp_{i -}}) \right| )} \label{eq:lohner.ind} \\
( \bp_{i \pm} )_m = \begin{cases}
p_m, \quad & m \neq i\\
p_{i \pm 1},  & m = i
\end{cases}
\end{gathered}
\end{equation*}
where $\left\{ \uu_{\bp} \right\}_{\bp \in \Nnd}$ are the degrees of freedom
in element $\Omega_e$ and $q$ is a derived quantity like the product of density and pressure used in Section~\ref{sec:smooth.ind}. The value $f_{\tmop{wave}} = 0.2$ has been chosen in all the
tests.

Once a smoothness indicator is chosen, the three level controller implemented
in \tmverbatim{Trixi.jl}~{\cite{Ranocha2021}} is used to determine the local
refinement level. The mesh begins with an initial refinement level and the
effective refinement level is prescribed by how much further refinement has
been done to the initial mesh, The mesh is created with two thresholds
\tmverbatim{med\_threshold} and \tmverbatim{max\_threshold} and three
refinement levels \tmverbatim{base\_level}, \tmverbatim{med\_level} and
\tmverbatim{max\_level}. Then, we have

\[
\tmverbatim{level}_e = \begin{cases}
\tmverbatim{base\_level},\qquad & \alpha_e \le \tmverbatim{med\_threshold} \\
\tmverbatim{med\_level},\qquad & \tmverbatim{med\_threshold} \le \alpha_e \le \tmverbatim{max\_threshold} \\
\tmverbatim{max\_level},\qquad & \tmverbatim{max\_threshold} \le \alpha_e
\end{cases}
\]

Beyond these refinement levels, further refinement is performed to make sure
that two neighbouring elements only differ by a refinement level of 1.

\section{Time stepping}\label{sec:time.stepping}

This section introduces an embedded error approximation method to compute the
time step size $\Delta t$ for the single stage Lax-Wendroff Flux
Reconstruction method. A standard way to compute the time step size
$\Delta t^n$ is to use~{\cite{babbar2022,babbar2023admissibility}}
\begin{equation}
\Delta t_n = C_s \min_{e,\bp} \frac{| J_{e,\bp} |}{\sigma ( \uebp^n )}
\tmop{CFL} (N) \label{eq:cfl.formula}
\end{equation}
where the minimum is taken over all elements $\{\Omega_e\}_e$, $J_e$ is the Jacobian of the change of variable map, $\sigma ( \uu^n_e )$ is the largest eigenvalue of the flux jacobian at state $\uu_e^n$, approximating the local wave speed, $\tmop{CFL} (N)$ is the optimal CFL number dependent on solution polynomial degree $N$ and $C_s \leq 1$ is a safety factor. In~{\cite{babbar2022}}, a Fourier stability analysis of the LWFR scheme was performed on Cartesian grids, and the optimal $\tmop{CFL}$ numbers were obtained for each degree $N$ which guaranteed the stability of the scheme. However, the Fourier stability analysis does not apply to curvilinear grids and formula~\eqref{eq:cfl.formula} need not guarantee $L^2$ stability which may require the CFL number to be fine-tuned for each problem. Along with the $L^2$ stability, the time step has to be chosen so that the scheme does not give inadmissible solutions. An error-based time stepping method inherently minimizes the parameter tuning process in time step computation. The parameters in an error-based time stepping scheme that a user has to specify are the absolute and relative error tolerances $\tau_a, \tau_r$, and they only affect the time step size logarithmically. In particular, because of the weak dependence, tolerances $\tau_a = \tau_r = 10^{- 6}$ worked reasonably for all tests with shocks; although, it was possible to enhance performance by choosing larger tolerances for some problems. Secondly, if inadmissibility is detected during any step in the scheme or if errors are too large, the time step is redone with a reduced time step size provided by the error estimate. The scheme also has the capability of increasing and decreasing the time step size.

We begin by reviewing the error-based time stepping scheme for the Runge-Kutta ODE solvers from~{\cite{Ranocha2021,ranocha2023}} in Section~\ref{sec:rk.error.section} and explain our extension of the same to LWFR in Section~\ref{sec:error.lw}.

\subsection{Error estimation for Runge-Kutta
schemes}\label{sec:rk.error.section}

Consider an explicit Runge-Kutta method used for solving ordinary differential equations by evolving the numerical solution from time level $n$ to $n+1$. For error-estimation, the method is constructed to have an embedded lower order update $\widehat{\uu}^{n + 1}$, as described in equation~(3) of~\cite{Ranocha2021}. The difference in the two updates, $\uu^{n+1} - \widehat{\uu}^{n + 1}$, gives an indication of the time integration error, which is used to build a  Proportional Integral Derivative (PID) controller to compute the new time step size,
\begin{equation}
\widetilde{\Delta t}_{n + 1} = \kappa (\varepsilon_{n + 1}^{\beta_1 /
k} \varepsilon_n^{\beta_2 / k} \varepsilon_{n - 1}^{\beta_3 / k})
\Delta t_n \label{eq:dtnp1.formula}
\end{equation}
where for $q$ being the order of main method, $\hat{q}$ being the order of
embedded method, we have
\[ k = \min (q, \hat{q}) + 1 \]
and $\beta_i$ are called control parameters which are optimized for the
particular Runge-Kutta scheme~{\cite{Ranocha2021}}. For $m$ being the degrees
of freedom in $\uu$, we pick absolute and relative tolerances $\tau_a, \tau_r$
and then error approximation is made as
\begin{equation}
\varepsilon_{n + 1} = \frac{1}{w_{n + 1}}, \qquad w_{n + 1} = \left(
\frac{1}{M}  \sum_{i = 1}^M \left ( \frac{\uu_i^{n + 1} - \widehat{\uu}_i^{n
+ 1}}{\tau_a + \tau_r \max \left\{ \left| \uu_i^{n + 1} \right|, \left| \nou
\right| \right\}} \right )^2 \right)^{\frac{1}{2}} \label{eq:error.estimator}
\end{equation}
where the sum is over all degrees of freedom, including solution points and
conservation variables. The tolerances are to be chosen by the user but their
influence on the scheme is logarithmic, unlike the CFL based
scheme~\eqref{eq:cfl.formula}.

The limiting function $\kappa (x) = 1 + \tan^{- 1} (x - 1)$ is used to prevent
sudden increase in time step sizes. For normalization, \tmverbatim{PETSc} uses
$\nou = \widehat{\uu}^{n + 1}$ while \tmverbatim{OrdinaryDiffEq.jl} uses $\nou
= \uu^n$. Following~{\cite{Ranocha2021}}, if the time step factor
$\widetilde{\Delta t}^{n + 1} / \Delta t^n \geq 0.9^2$, the new
time step is accepted and used in the next level as $\Delta t^{n + 1} =
\widetilde{\Delta t}^{n + 1}$. If not, or if admissibility is violated,
evolution is redone with time step size $\Delta t^n =
\widetilde{\Delta t}^{n + 1}$ computed from~\eqref{eq:dtnp1.formula}.

\subsection{Error based time stepping for Lax-Wendroff flux
reconstruction}\label{sec:error.lw}

Consider the LWFR scheme~\eqref{eq:lwfr.update.curvilinear} with polynomial degree $N$ and formal order
of accuracy $N+1$
\begin{equation*}
\uebp^{n + 1} = \uebp^n - \frac{\Delta t}{J_{e, \bp}}
\nabla_{\vxi} \cdot \tF^{\delta}_e ( \vxi_{\bp} ) - \C_{e,\bp}
\end{equation*}
where $\C_{e,\bp}$ contains contributions at element interfaces. In order to construct a lower order embedded scheme without requiring additional inter-element communication, consider an evolution where the
interface correction terms $\C_{e,\bp}$ are not used, i.e., consider the element local update
\begin{equation}
\uu^{n + 1}_{\tmop{loc}, e, \bp} = \uu^n_{e, \bp} - \frac{\Delta t}{J_{e,\bp}}
 \nabla_{\vxi} \cdot \tF^{\delta}_e ( \vxi_{\bp} )
\label{eq:uloc.high}
\end{equation}
Truncating the locally computed time averaged flux
$\tF^{\delta}_e$~\eqref{eq:time.averaged.flux} at one order lower
\begin{equation}
\widehat{\F^{\delta}_e} = \sum_{k = 0}^{N - 1} \frac{\Delta t^k}{(k +
1) !} \partial_t^k  \tf_e^{\delta} \label{eq:Fdelta.low}
\end{equation}
we can consider another update
\begin{equation}
\widehat{\uu^{n + 1}_{\tmop{loc}, e, \bp}} = \uu^n_{e, \bp} - \frac{\Delta t}{J_{e,\bp}}
 \nabla_{\vxi} \cdot \widehat{\F^{\delta}_e} ( \vxi_{\bp}
) \label{eq:unp1.low}
\end{equation}
which is also locally computed but is one order of accuracy lower. We thus use $\uu^{n + 1}_e = \uu^{n + 1}_{\tmop{loc}, e}$ and $\widehat{\uu}^{n + 1}_e = \widehat{\uu^{n + 1}_{\tmop{loc}, e}}$ in the formula~\eqref{eq:error.estimator} along with $\nou = \widehat{\uu}^{n + 1}$; then we  use the same procedure of redoing the time step sizes as in Section~\ref{sec:rk.error.section}. That is, after using the error estimate~\eqref{eq:error.estimator} to compute $\widetilde{\Delta t}_{n + 1}$~\eqref{eq:dtnp1.formula} we redo the time step if $\widetilde{\Delta t}^{n + 1} / \Delta t^n \geq 0.9^2$ or if admissibility is violated; otherwise we set $\Delta t^{n+1}$ to be used at the next time level. The complete process is also detailed in Algorithm~\ref{alg:time.stepping}. In this work, we have used the control parameters $\beta_1 = 0.6, \beta_2 = - 0.2, \beta_3 = 0.0$ for all numerical results which are the same as those used in~{\cite{Ranocha2021}} for $\text{BS}3(2)3_F$, the third-order, four-stage RK method of~{\cite{Bogacki1989}}. We tried the other control parameters from~{\cite{Ranocha2021}} but found the present choice to be either superior or only slightly different in performance, measured by the number of iterations taken to reach the final time.
\begin{algorithm}
\caption{High level overview of element residual computation of order $N + 1$ including error approximation using $\uu^{n + 1}_{\tmop{loc}}, \widehat{\uu^{n + 1}_{\tmop{loc}}}$ \label{alg:cell.residual}}
\begin{algorithmic} 
\For{ $e$ in \texttt{eachelement(mesh)} }
\State Compute $\left\{ \partial_t^k  \tf_e^{\delta} \right\}_{k = 0}^{N -1}$ using the approximate Lax-Wendroff procedure~\eqref{eq:ft.defn} to obtain $\widehat{\F^{\delta}_e}$~\eqref{eq:Fdelta.low}

\State Compute $\widehat{\uu^{n + 1}_{\tmop{loc}, e}}$ using $\widehat{\F^{\delta}_e}$~\eqref{eq:unp1.low}

\State Compute \ $\partial_t^{N + 1}  \tf_e^{\delta}$ using the approximate Lax-Wendroff procedure~\eqref{eq:ft.defn} to obtain $\tF^{\delta}_e$~\eqref{eq:time.averaged.flux}

\State Compute $\uu^{n + 1}_{\tmop{loc}, e}$ using $\F^{\delta}_e$ as in~\eqref{eq:uloc.high}

\State \texttt{temporal\_error[e]=}$\displaystyle \sum\limits_i \left( \frac{\uu^{n + 1}_{\tmop{loc}, e, i} - \widehat{\uu^{n + 1}_{\tmop{loc}, e, i}}}{\tau_a + \tau_r \max \left\{ \left| \uu^{n + 1}_{\tmop{loc}, e, i} \right|, \left| \widehat{\uu^{n + 1}_{\tmop{loc}, e, i}} \right| \right\}} \right)^2$ where the sum is over dofs in $e$

\State Compute and add local contribution of $\F^{\delta}_e$ to
the residual~\eqref{eq:lwfr.update.curvilinear}
\EndFor
\end{algorithmic}
\end{algorithm}

\begin{algorithm}
\caption{High level overview of LWFR residual (Within time integration) \label{alg:lw.residual}}
\begin{algorithmic} 
\State Compute $\{ \alpha_e \}$ (Section~\ref{sec:smooth.ind})
\State Assemble cell residual (Algorithm~\ref{alg:cell.residual})
\For{ $\G$ in \tmverbatim{eachinterface(mesh)}}
\State Compute $\F_{\G}^{\text{LW}}, \pf_{\G}$ and blend them into
$\F_{\G}$ (Algorithm~\ref{alg:blended.flux})
\EndFor

\For{$e$ in \tmverbatim{eachelement(mesh)}}
\State Add contribution of numerical fluxes to residual of element $e$ (common to high, low residual, see Remark~\ref{rmk:common.contri})
\EndFor

\State Update solution
\State Apply positivity limiter
\end{algorithmic}
\end{algorithm}

\begin{algorithm}
\caption{\label{alg:time.stepping}Lax-Wendroff Flux Reconstruction at a high level to explain error based time stepping}
\begin{algorithmic} 
	\State Initialize $t \leftarrow 0$, time step number $n \leftarrow 0$, and
	initial state $\uu^0$
	\State Initialize PID controller with $\varepsilon_0 \leftarrow 1,
	\varepsilon_{- 1} \leftarrow 1$
	\State Initialize $\Delta t_0 = \widetilde{\Delta t}$ with a user
	supplied value
	\State Initialize \tmverbatim{accept\_step} $\leftarrow$ \tmverbatim{false}
	\While{$t < T$}
	\If{\tmverbatim{accept\_step}}
	\State \tmverbatim{accept\_step} $\leftarrow$ \tmverbatim{false}
	\State $t \leftarrow \tilde{t}$

	\State $\Delta t_{n + 1} \leftarrow \widetilde{\Delta t}$

	\State $n \leftarrow n + 1$
	\Else
	\State $\Delta t_n \leftarrow \widetilde{\Delta t}$
	\EndIf
	\If{$t + \Delta t_n >$\tmverbatim{final\_time}}
	\State $\Delta t_n$ $\leftarrow$ \tmverbatim{final\_time}$- t$
	\EndIf

\State $\uu^n \rightarrow \uu^{n + 1}$ with $\Delta t_n$
(Algorithm~\ref{alg:lw.residual}, \ref{alg:cell.residual}) computing \tmverbatim{temporal\_error} and checking
admissibility

\State $w_{n + 1} \gets \left( \frac{1}{M}  \sum\limits_e \text{\tmverbatim{temporal\_error[e]}} \right)^{\frac{1}{2}}$
\Comment{$M$ is the total dofs}

\State $w_{n + 1} \gets \max \{ w_{n + 1}, 10^{- 10} \}$
\Comment{To avoid division by zero}

\State $\varepsilon_{n + 1} \leftarrow \frac{1}{w_{n + 1}}$

\State \tmverbatim{dt\_factor} $\leftarrow$ $\kappa (\varepsilon_{n + 1}^{\beta_1 / k} \varepsilon_n^{\beta_2 / k} \varepsilon_{n - 1}^{\beta_3 / k})$
\Comment{$\kappa (x) = 1 + \tan^{- 1} (x - 1)$}

\State $\widetilde{\Delta t} \leftarrow$\tmverbatim{dt\_factor}$\cdot
\Delta t_n$

\If{\tmverbatim{dt\_factor} $\geq$ \tmverbatim{accept\_safety} \tmverbatim{\&\&} \tmverbatim{no inadmissibility}}
\State \tmverbatim{accept\_step} $\leftarrow$ \tmverbatim{true}
\Else
\State \tmverbatim{accept\_step} $\leftarrow$ \tmverbatim{false}
\EndIf

\If{\tmverbatim{accept\_step}}
\State $\tilde{t} \leftarrow t + \Delta t_n$

\If{$\tilde{t} \approx$ \tmverbatim{final\_time}}
\State $\tilde{t} \leftarrow$\tmverbatim{final\_time}
\EndIf
\State Apply callbacks
\Comment{Analyze and postprocess solution, AMR}

\State Positivity correction for AMR (Remark~\ref{rmk:amr.metric.terms})
\EndIf
\EndWhile
\end{algorithmic}
\end{algorithm}

\section{Numerical results}\label{sec:numerical.results}

The numerical experiments are performed on 2-D Euler's equations~\eqref{eq:2deuler}. Unless specified otherwise, the adiabatic constant $\gamma$ will be taken as $1.4$ in the numerical tests, which is the typical value for air. The CFL based time stepping schemes use the following formula for the time step (see 2.5 of~\cite{ranocha2023}, but also~\cite{Rasha2021, Ranocha2021})
\begin{equation}
\Delta t_n = \frac{2}{N + 1} \Cs \min_{e, \bp}  \left( \frac{1}{\left| J_{e, \bp} \right|}  \sum_{i = 1}^d \tilde{\lambda}_{e, \bp}^i \right), \qquad \Cs \leq 1 \label{eq:cfl.time.step}
\end{equation}
where $\{ \tilde{\lambda}_{e, \bp}^i \}_{i \subindex 1}^d$ are wave
speed estimates computed by the transformation
\[
\tilde{\lambda}_{e, \bp}^i = \sum_{n = 1}^d (J a^i_n)_{e,\bp} \lambda_{e, \bp}^i
\]
where $\{ J \ba^i \}_{i \subindex 1}^d$ are the contravariant
vectors~\eqref{defn:contravariant.basis} and $\lambda_{e,\bp}^i$ is the absolute maximum eigenvalue of $\pf'_i ( \uu_{e, \bp} )$. For Euler's
equations with velocity vector $\bv = \{ v_i \}$ and sound speed $c$,
$\lambda^i = | v_i | + c$. The $\Cs$ in~\eqref{eq:cfl.time.step} may need to be
fine-tuned depending on the problem. Other than the convergence test~(Section~\ref{sec:isentropic}), the results shown below have been generated with error-based time stepping~(Section~\ref{sec:error.lw}). The scheme is implemented in a Julia package {\tt TrixiLW.jl}\correction{~\cite{babbar2024trixilw}} written using {\tt Trixi.jl}~{\cite{Ranocha2022,schlottkelakemper2021purely,schlottkelakemper2020trixi}}
as a library. {\tt Trixi.jl} is a high order PDE solver package in {\tt Julia}~\cite{Bezanson2017} and uses the Runge-Kutta Discontinuous Galerkin method; {\tt TrixiLW.jl} uses {\tt Julia}'s multiple dispatch to borrow features like curved meshes support and postprocessing from {\tt Trixi.jl}. {\tt TrixiLW.jl} is not a fork of {\tt Trixi.jl} but only uses it through {\tt Julia}'s package manager without modifying its internal code. \correction{Most of the curvilinear meshes are generated using {\tt HOHQMesh.jl}~\cite{hohqmesh}, whose construction process we briefly describe. The mesh generation in {\tt HOHQMesh.jl} begins with a Cartesian mesh on which each curved boundary is supplied (as analytic expressions, splines or even as sets of points) within the Cartesian domain (e.g., the top/bottom of the airfoil in Figure~\ref{fig:naca}a) followed by application of a smoother on elements nearing the boundary. Smoothing is a process where elements are relocated to improve mesh quality (e.g., minimum/maximum angle, aspect ratio). Currently, {\tt HOHQMesh.jl} always uses a linear spring smoothing algorithm from~\cite{Minoli2011,Mavriplis2006} that produces straight-sided quadrilaterals. Thus, for the meshes generated using {\tt HOHQMesh.jl}, only boundary elements have curved faces. However, the solver can handle general meshes where the inner elements can also have curved boundaries. Some numerical experiments like the isentropic vortex (Figure~\ref{fig:isentropic}) are performed on meshes generated using a transformation from a square and have such curved interfaces in the interior also.} The setup files for the numerical experiments in this work are available at~\cite{curvedrepo}.
\subsection{Results on Cartesian grids}

\subsubsection{Mach 2000 astrophysical jet}

In this test, a hypersonic jet is injected into a gas at rest with a Mach number of 2000 relative to the sound speed in the gas at rest. Following~{\cite{ha2005,zhang2010c}}, the domain is taken to be $[0, 1] \times [- 0.5, 0.5]$, the ambient gas in the interior has state $\tmmathbf{u}_a$ defined in primitive variables as
\[
(\rho, u, v, p)_a = (0.5, 0, 0, 0.4127)
\]
and inflow state $\tmmathbf{u}_j$ is defined in primitive variables as
\[
(\rho, u, v, p)_j = (5, 800, 0, 0.4127)
\]
On the left boundary, we impose the boundary conditions
\[
\uu_b = \begin{cases}
\uu_a, & \quad \text{if} \quad y \in [- 0.05, 0.05]\\
\uu_j, & \quad \text{otherwise}
\end{cases}
\]
and outflow conditions on the right, top and bottom of the computational domain.

The simulation is performed on a uniform $512^2$ element mesh. This test requires admissibility preservation to be enforced to avoid solutions with negative pressure. This is a cold-start problem as the solution is constant with zero velocity in the domain at time $t=0$. However, there is a high speed inflow at the boundary, which the standard wave speed estimate for time step approximation~\eqref{eq:cfl.effective} does not account for. Thus, in order to use the CFL based time stepping, lower values of $\Cs$~\eqref{eq:cfl.time.step} have to be used in the first few iterations of the simulations. Once the high speed flow has entered the domain, the this value needs to be raised since otherwise, the simulation will use much smaller time steps than the linear stability limit permits. Error based time stepping schemes automate this process by their adaptivity and ability to redo the time steps. The simulation is run till $t = 10^{- 2}$ and the log scaled density plot for degree $N = 4$ solution obtained on the uniform mesh is shown in Figure~\ref{fig:m2000}a.  For an error-based time stepping scheme, we define the effective $\Cs$ as
\begin{equation}
\Cs := \Delta t_n  \left[ \frac{2}{N + 1} \min_{e, \bp}  \left(\frac{1}{\left| J_{e, \bp} \right|}  \sum_{i = 1}^d \tilde{\lambda}_{e, \bp}^i \right) \right]^{- 1} \label{eq:cfl.effective}
\end{equation}
which is a reverse computation so that its usage in~\eqref{eq:cfl.formula} will get $\Delta t_n$
chosen in the error-based time stepping scheme~(Algorithm~\ref{alg:time.stepping}). In Figure~\ref{fig:m2000}b, time $t$ versus effective $\Cs$~\eqref{eq:cfl.effective} is plotted
upto $t = 10^{- 5}$ to demonstrate that the scheme automatically uses a smaller $\Cs$
of ${\sim} 10^{- 3}$ at the beginning which later increases and stabilizes at ${\sim} 10^{- 1}$. Thus, the error based time stepping is automatically doing what would have to be manually implemented for a CFL based time stepping scheme which would be problem-dependent and requiring user intervention.
\begin{figure}
\begin{tabular}{cc}
{\includegraphics[width=0.44\textwidth]{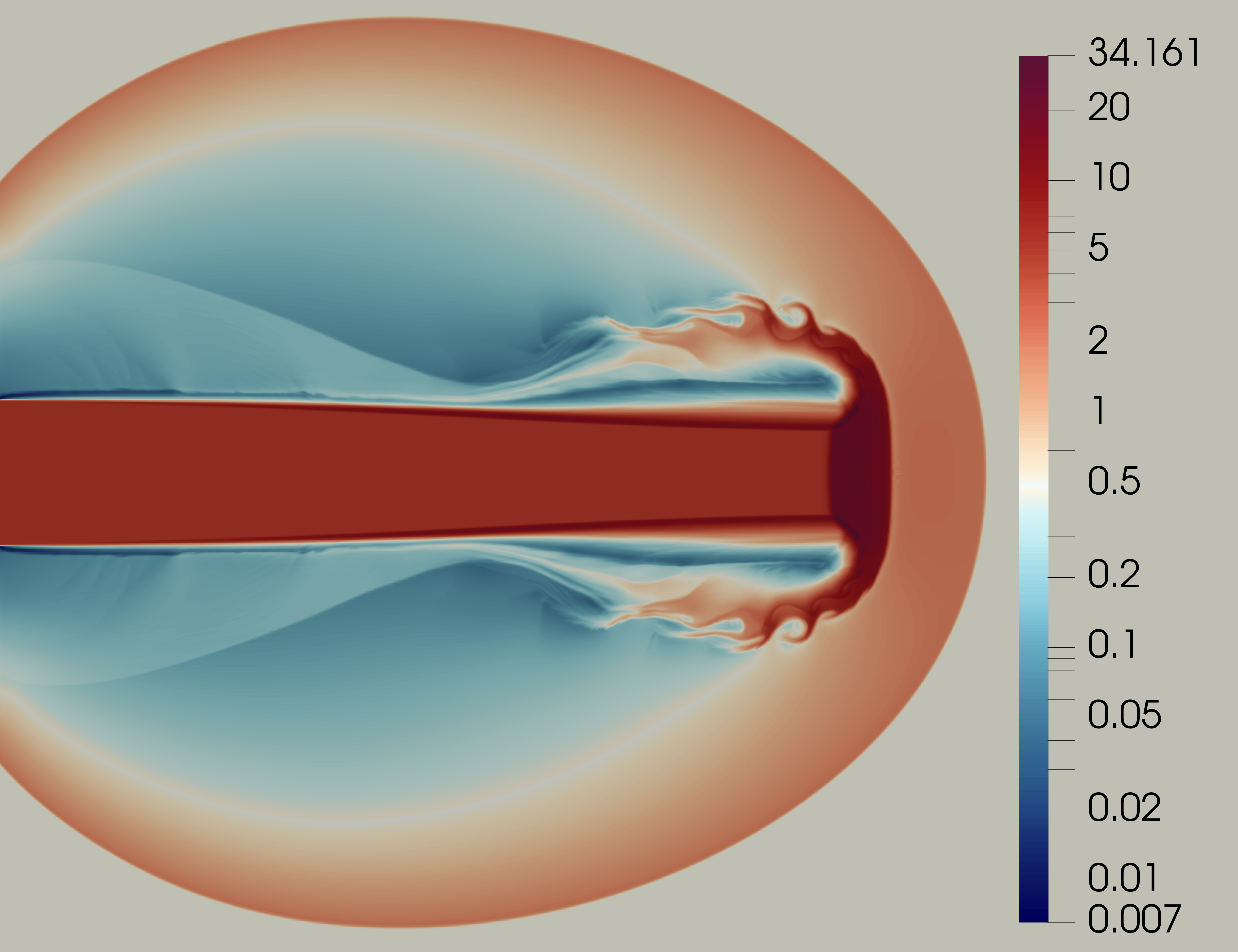}} & {\includegraphics[width=0.44\textwidth]{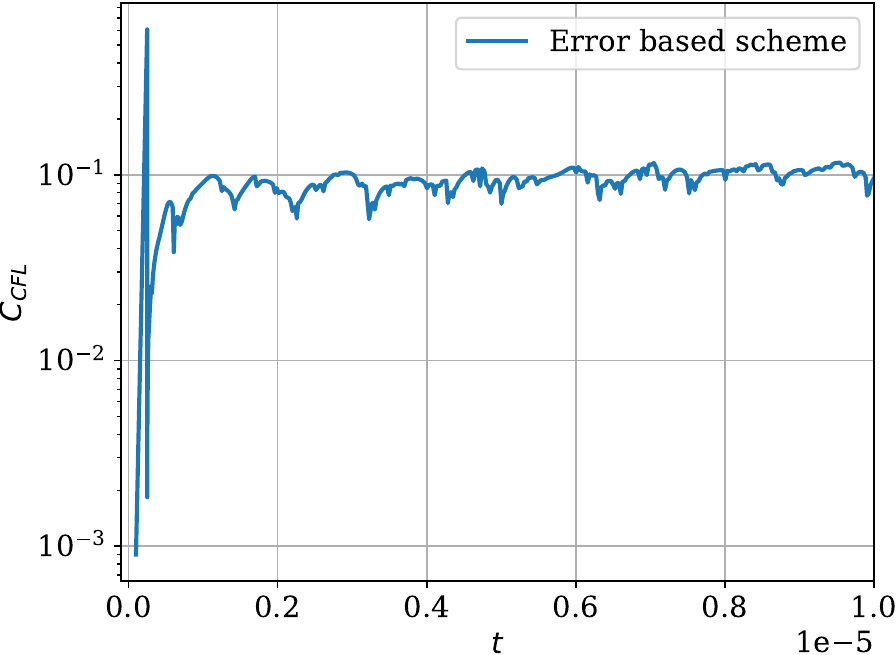}}\\
(a) & (b)
\end{tabular}
\caption{Mach 2000 astrophysical jet (a) Density plot (b) Effective $\Cs$\label{fig:m2000}}
\end{figure}

\subsubsection{Kelvin-Helmholtz instability}

The Kelvin-Helmholtz
instability is a common fluid instability that occurs across density and
tangential velocity gradients leading to a tangential shear flow. This
instability leads to the formation of vortices that grow in amplitude and can
eventually lead to the onset of turbulence. The initial condition is given
by~{\cite{ramirez2021}}
\[ (\rho, u, v, p) = \left( \frac{1}{2} + \frac{3}{4} B, \frac{1}{2}  (B - 1),
\frac{1}{10} \sin (2 \pi x), 1 \right) \]
with $B = \tanh (15 y + 7.5) - \tanh (15 y - 7.5)$ in domain
$\Omega = [- 1, 1]^2$ with periodic boundary conditions. The initial condition has a Mach number $M \leq 0.6$
which makes compressibility effects relevant but does not cause shocks to
develop. Thus, a very mild shock capturing scheme is used by setting $\alpha_e
= \min \{ \alpha_e, \alpha_{\max} \}$~(Section~\ref{sec:smooth.ind}) where
$\alpha_{\max} = 0.002$. The same smoothness indicator of
Section~\ref{sec:smooth.ind} is used for AMR indicator with parameters from
Section~\ref{sec:amr.indicator} chosen to be
\[
(\tmverbatim{base{\_}level}, \tmverbatim{med{\_}level}, \tmverbatim{max{\_}level}) = (4, 0, 8), \qquad
(\tmverbatim{{med}{\_}{threshold}}, \tmverbatim{max{\_}threshold}) = (0.0003, 0.003)
\]
where $\tmverbatim{base\_level} = 0$ refers to a $2 \times 2$ mesh. This test case, along with indicators' configuration was taken from the examples of \tmverbatim{Trixi.jl}~{\cite{Ranocha2021}}. The simulation is run till $t = 3$ using polynomial degree $N = 4$. There is a shear layer at $y=\pm 0.5$ which rolls up and develops smaller scale structures as time progresses. The results are shown in Figure~\ref{fig:kh} and it can be seen that the AMR indicator is able to track the small scale structures. The simulation starts with a mesh of $1024$ elements which steadily increases to 13957 at the final time; the mesh is adaptively refined or coarsened at every time step. The solution has non-trivial variations in small regions around the rolling structures which an adaptive mesh algorithm can capture efficiently, while a uniform mesh with similar resolution would require 262144 elements.
\begin{figure}
\centering
\begin{tabular}{cc}
{\includegraphics[width=0.46\textwidth]{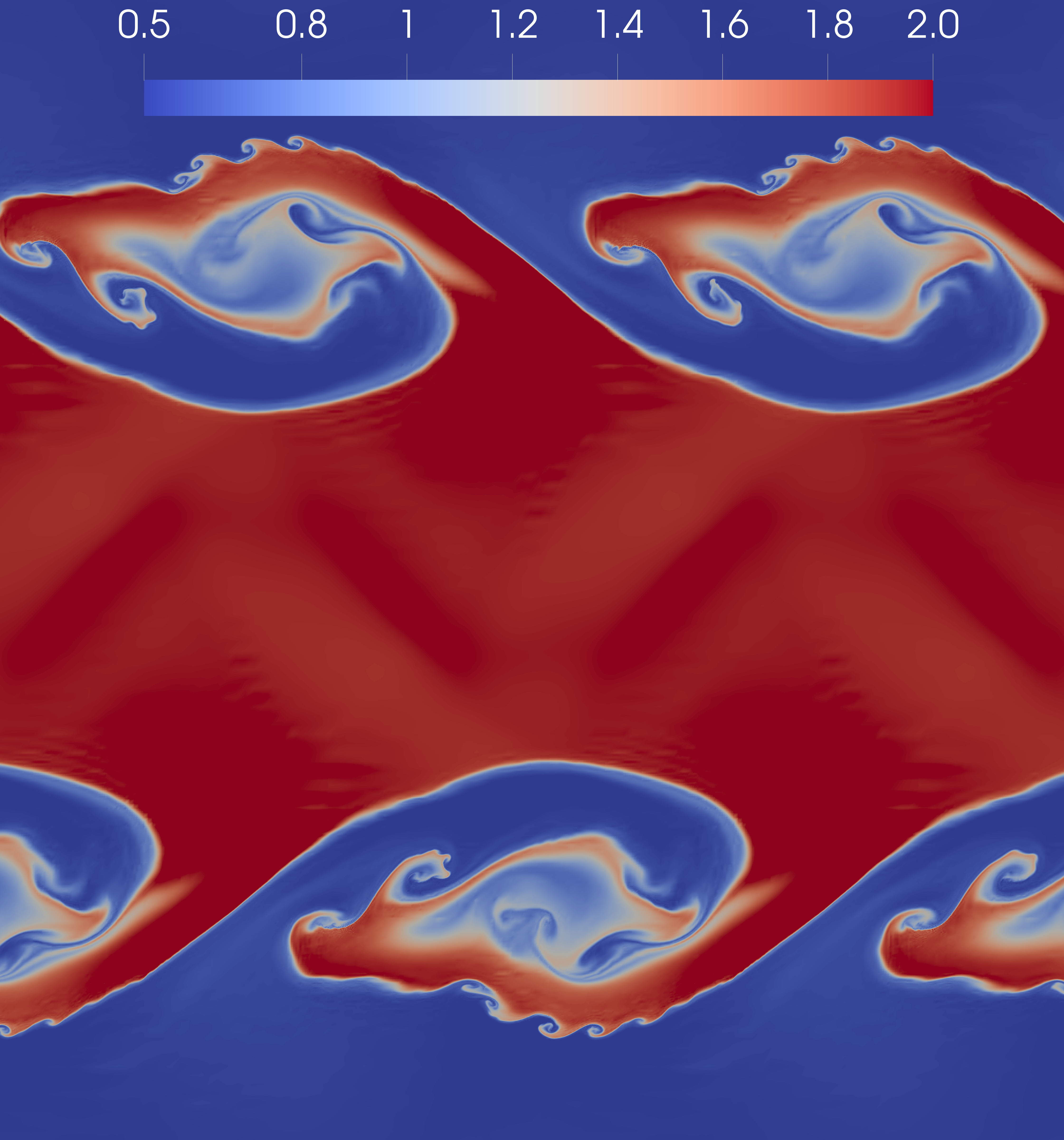}} & {\includegraphics[width=0.46\textwidth]{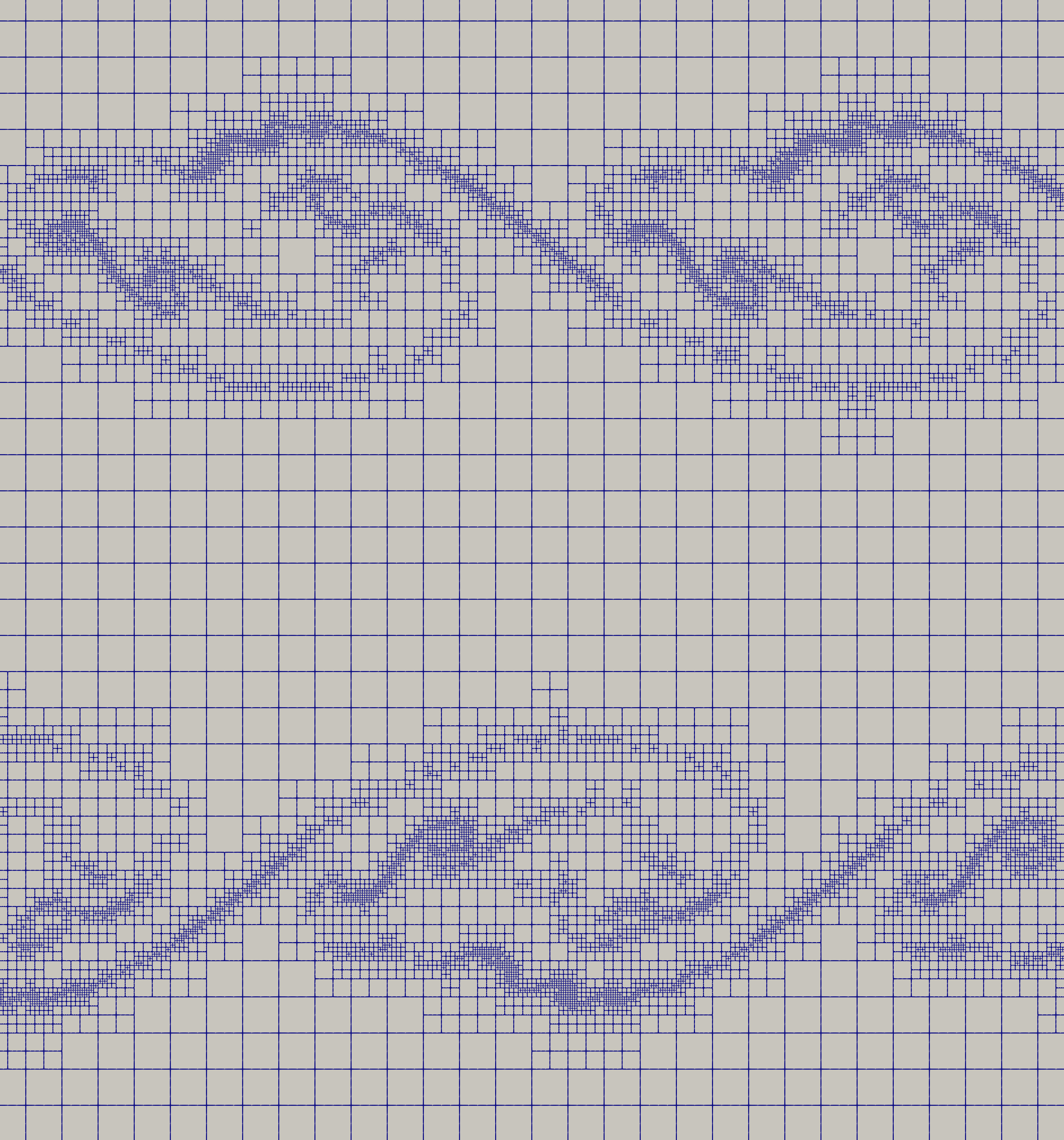}}\\
(a) & (b)
\end{tabular}
\caption{\label{fig:kh}Kelvin-Helmholtz instability at $t = 3$ using
polynomial degree $N = 4$ (a) density plots, (b) adaptively refined mesh}
\end{figure}

\subsubsection{Double mach reflection}\label{sec:dmr}

This test case was originally proposed by Woodward and Colella~{\cite{Woodward1984}} and consists of a shock impinging on a wedge/ramp which is inclined by 30 degrees. An equivalent problem is obtained on the rectangular domain $\Omega = [0,4] \times [0,1]$ obtained by rotating the wedge so that the initial condition now consists of a shock angled at 60 degrees. The solution consists of a self similar shock structure with two triple points. Define $\uu_b = \uu_b(x, y, t)$ in primitive variables as
\[
(\rho, u, v, p) = \left\{\begin{array}{ll}
(8, 8.25 \cos ( \frac{\pi}{6} ), - 8.25 \sin (
\frac{\pi}{6} ), 116.5), & \text{if } x < \frac{1}{6} + \frac{y +
20 t}{\sqrt{3}}\\
(1.4, 0, 0, 1), & \text{if } x > \frac{1}{6} + \frac{y + 20 t}{\sqrt{3}}
\end{array}\right.
\]
and take the initial condition to be $\uu_0 (x, y) = \uu_b (x, y, 0)$. With
$\uu_b$, we impose inflow boundary conditions at the left side $\{0\} \times
[0, 1]$, outflow boundary conditions both on $[0, 1 / 6] \times \{0\}$ and
$\{4\} \times [0, 1]$, reflecting boundary conditions on $[1 / 6, 4] \times
\{0\}$ and inflow boundary conditions on the upper side $[0, 4] \times \{1\}$.

The setup of L{\"o}hner's smoothness indicator~\eqref{eq:lohner.ind} is taken from an example of \tmverbatim{Trixi.jl}~{\cite{Ranocha2021}}
\[
(\tmverbatim{base{\_}level}, \tmverbatim{med{\_}level}, \tmverbatim{max{\_}level}) = (0, 3, 6), \qquad
(\tmverbatim{{med}{\_}{threshold}}, \tmverbatim{max{\_}threshold}) = (0.05, 0.1)
\]
where $\tmverbatim{base{\_}level} = 0$ corresponds to a $16 \times 5$ mesh. The density solution obtained using polynomial degree $N = 4$ is shown in Figure~\ref{fig:dmr}
where it is seen that AMR is tracing the shocks and small scale shearing
well. The initial mesh consists of 80 elements and is refined in first iteration in the vicinity of the shock to get 2411 elements. In later iterations, the mesh is refined and coarsened in each iteration, and the number of elements keeps increasing up to 7793 elements at the final time $t=0.2$. In order to capture the same effective refinement, a uniform mesh will require 327680 elements.

\begin{figure}
\centering
\begin{tabular}{l}
{\includegraphics[width=0.9\textwidth]{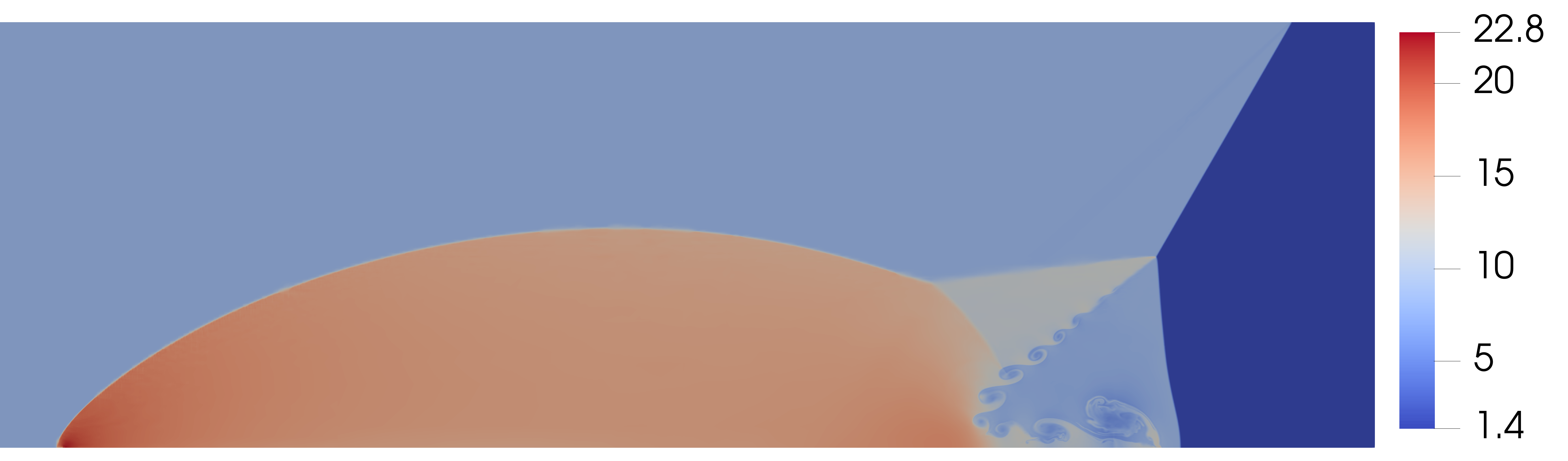}} \\
\hspace{0.4\textwidth} (a) \\
\includegraphics[width=0.8\textwidth]{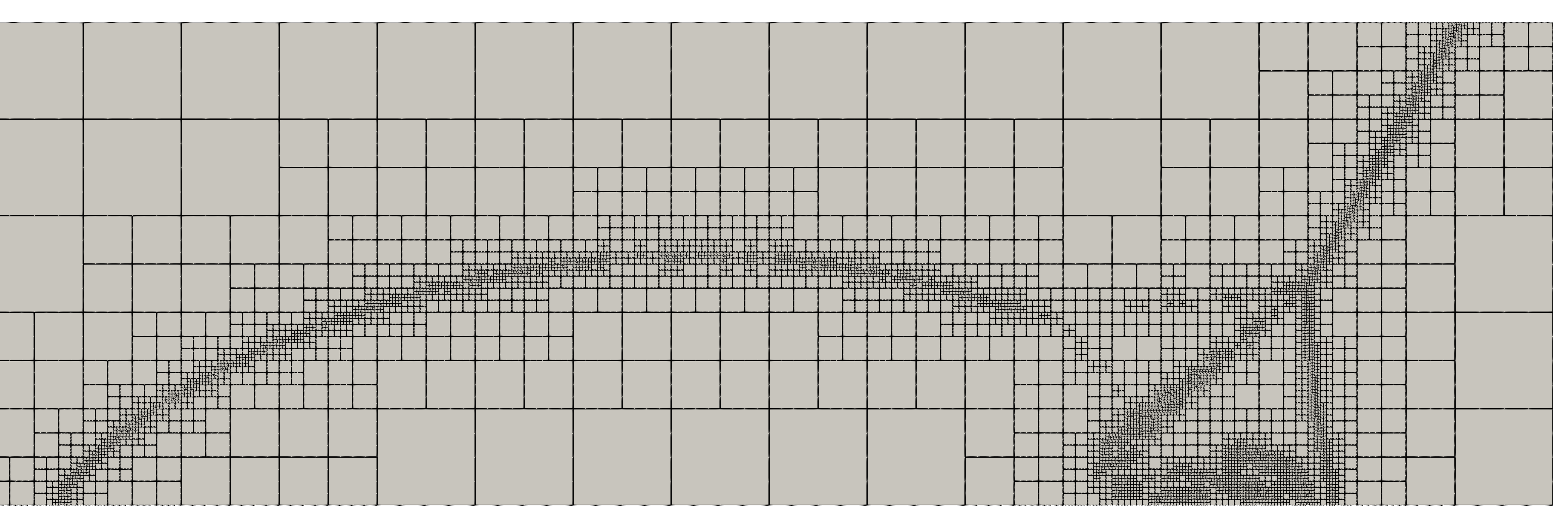} \\
\hspace{0.4\textwidth} (b)
\end{tabular}
\caption{\label{fig:dmr}Double Mach reflection with solution polynomial degree $N = 4$ at $t = 0.2$ (a) Density plot, (b) Adaptively refined mesh at
final time}
\end{figure}

\subsubsection{Forward facing step}\label{sec:forward.step}

Forward facing step is a classical test case
from~{\cite{emery1968,Woodward1984}} where a uniform supersonic flow passes
through a channel with a forward facing step generating several phenomena like
a strong bow shock, shock reflections and a Kelvin-Helmholtz instability. It
is a good test for demonstrating a shock capturing scheme's capability of
capturing small scale vortex structures while suppressing spurious
oscillations arising from shocks. The step is simulated in the domain $\Omega = ([0, 3] \times [0, 1]) \setminus ([0.6, 3] \times [0, 0.2])$
and the initial conditions are taken to be
\[
(\rho, u, v, p) = (1.4, 3, 0, 1) \qquad \text{in} \quad \Omega
\]
The left boundary condition is taken as an inflow and the right one
is an outflow, while the rest are solid walls. The corner $(0.6, 0.2)$ of the step
is the center of a rarefaction fan and can lead to large errors and the formation of a spurious boundary layer, as shown in Figure 7a-7d of~\cite{Woodward1984}. These errors can be reduced by refining the mesh near the corner, which is automated here with the AMR algorithm.

The setup of L{\"o}hner's smoothness indicator~\eqref{eq:lohner.ind} is taken from an example of \tmverbatim{Trixi.jl}~{\cite{Ranocha2021}}
\[
(\tmverbatim{base{\_}level}, \tmverbatim{med{\_}level}, \tmverbatim{max{\_}level}) = (0, 2, 5), \qquad
(\tmverbatim{{med}{\_}{threshold}}, \tmverbatim{max{\_}threshold}) = (0.05, 0.1)
\]
The density at $t = 3$ obtained using polynomial
degree $N = 4$ and L{\"o}hner's smoothness indicator~\eqref{eq:lohner.ind}
is plotted in Figure~\ref{fig:forward.step}. The shocks have been well-traced
and resolved by AMR and the spurious boundary layer and Mach stem do not
appear. The simulations starts with a mesh of 198 elements and the number peaks at 6700 elements during the simulation then and decreases to 6099 at the final time $t=3$. The mesh is adaptively refined or coarsened once every 100 time steps. In order to capture the same effective refinement, a uniform mesh will require 202752 elements.

\begin{figure}
\centering
\begin{tabular}{c}
{\includegraphics[width=0.8\textwidth]{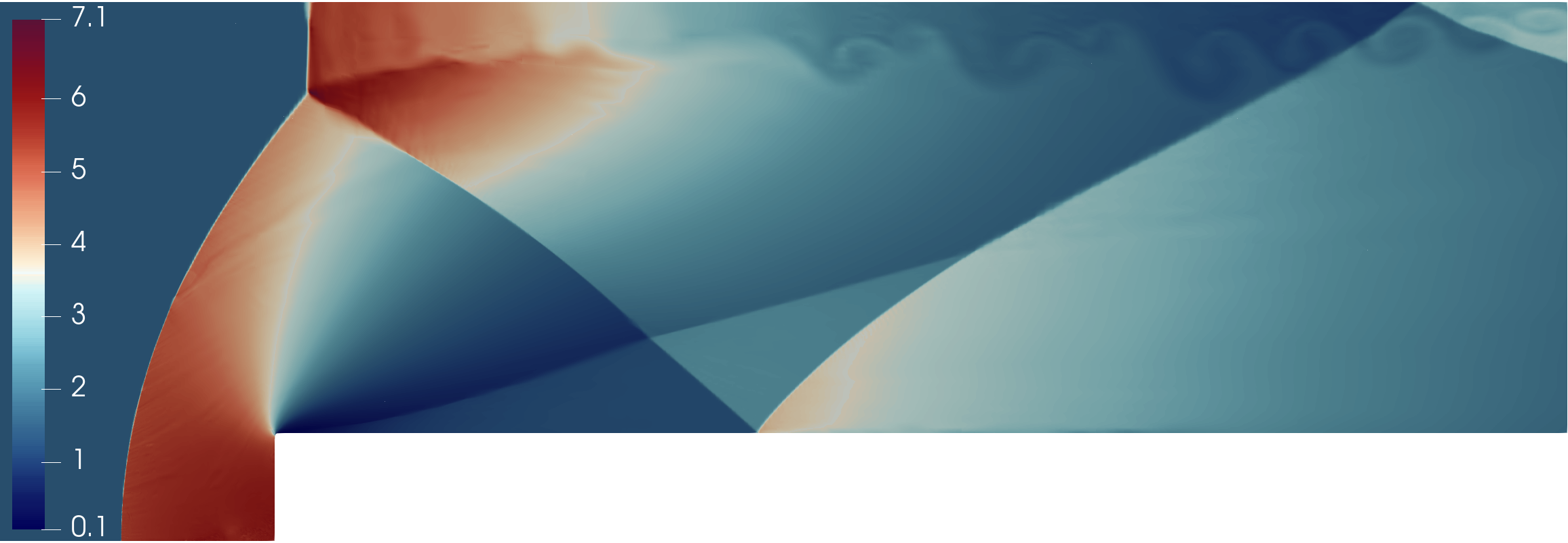}} \\
(a) \\
\includegraphics[width=0.8\textwidth]{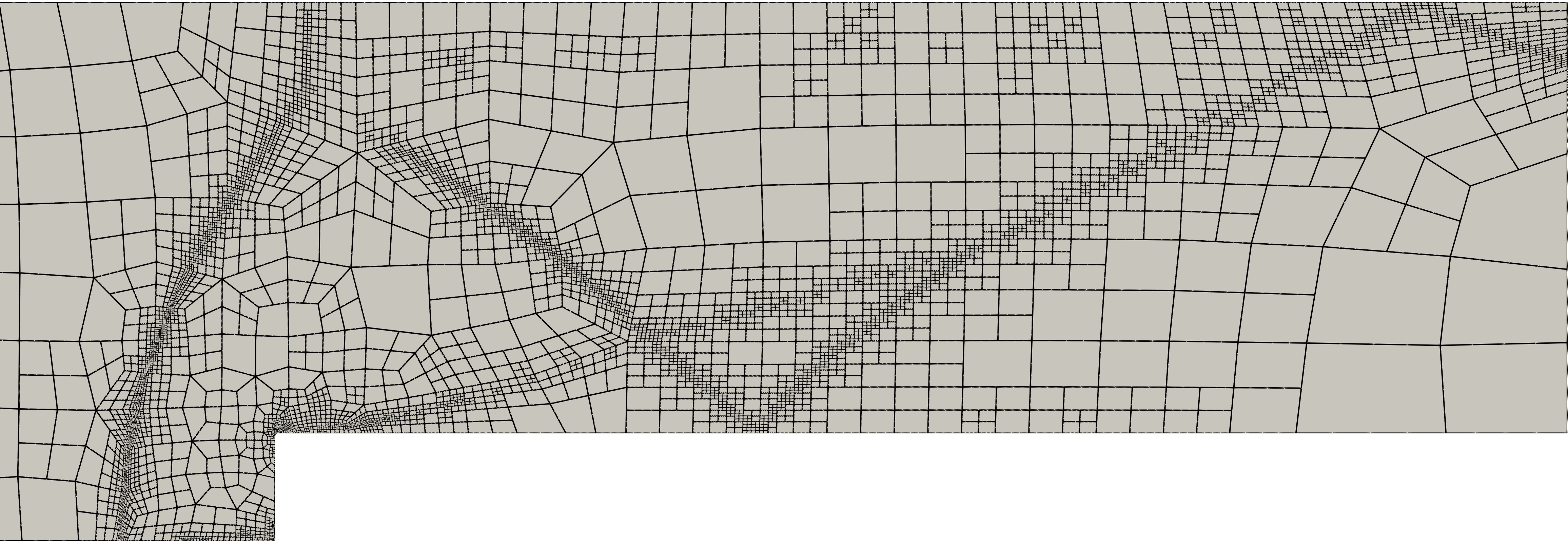} \\
(b)
\end{tabular}
\caption{\label{fig:forward.step}Mach 3 flow over forward facing step at
time $t = 3$ using solution polynomial degree $N = 4$ with L{\"o}hner's
indicator for mesh refinement. (a) Density plot (b) Adaptively refined mesh}
\end{figure}

\subsection{Results on curved grids}


\subsubsection{Free stream preservation}

In this section, free stream preservatio\textbf{}n is tested for meshes with curved elements. Since we use a reference map of degree $N$ in~\eqref{eq:reference.map}, free stream will be preserved following the discussion in Section~\ref{sec:free.stream.lwfr}. We numerically verify the same for the meshes taken from \tmverbatim{Trixi.jl} which are shown in Figure~\ref{fig:free.stream}. The mesh in Figure~\ref{fig:free.stream}a consists of curved boundaries and only the elements adjacent to the boundary are curved, while the one in Figure~\ref{fig:free.stream}b is a non-conforming mesh with curved elements everywhere, and is used to verify that free stream preservation holds with adaptively refined meshes. The mesh in Figure~\ref{fig:free.stream}b is a 2-D reduction of the one used in Figure 3 of~{\cite{Rueda2021}} and is defined by the global map $ (\xi, \eta) \mapsto (x, y)$ from $[0, 3]^2 \to \Omega$ described as
\begin{equation*}
x  = \xi + \frac{3}{8}
\cos \left( \frac{\pi}{2} \frac{2 \xi - 3}{3} \right)
\cos \left( 2 \pi \frac{2 y - 3}{3} \right), \qquad
y  = \eta + \frac{3}{8}
\cos \left( \frac{3\pi}{2} \frac{2 \xi - 3}{3} \right)
\cos \left( \frac{\pi}{2} \frac{2 \eta - 3}{3}\right)
\end{equation*}
The free stream preservation is verified on these meshes by solving the Euler's equation with constant initial data
\[
(\rho, u, v, p) = (1, 0.1, -0.2, 10)
\]
and Dirichlet boundary conditions. Figure~\ref{fig:free.stream} shows the density at time $t=10$ which is constant throughout the domain.
\begin{figure}
\centering
\includegraphics[width=0.8\textwidth]{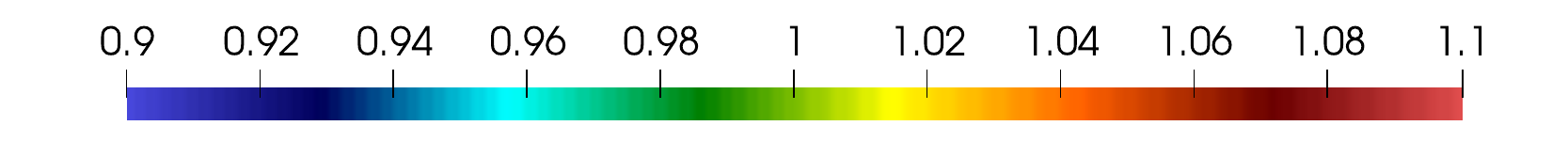}
\begin{tabular}{cc}
{\includegraphics[width=0.34\textwidth]{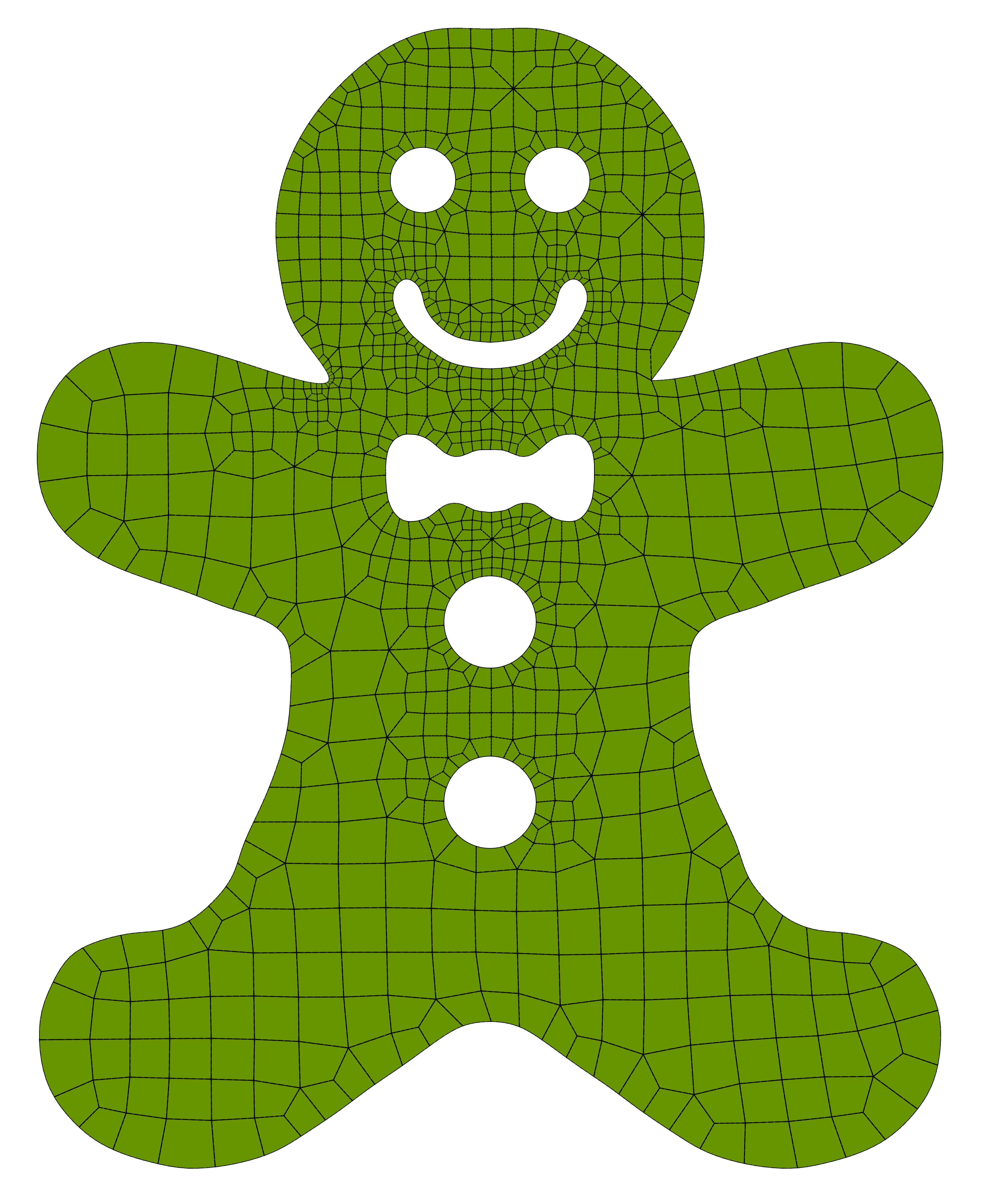}} & {\includegraphics[width=0.4\textwidth]{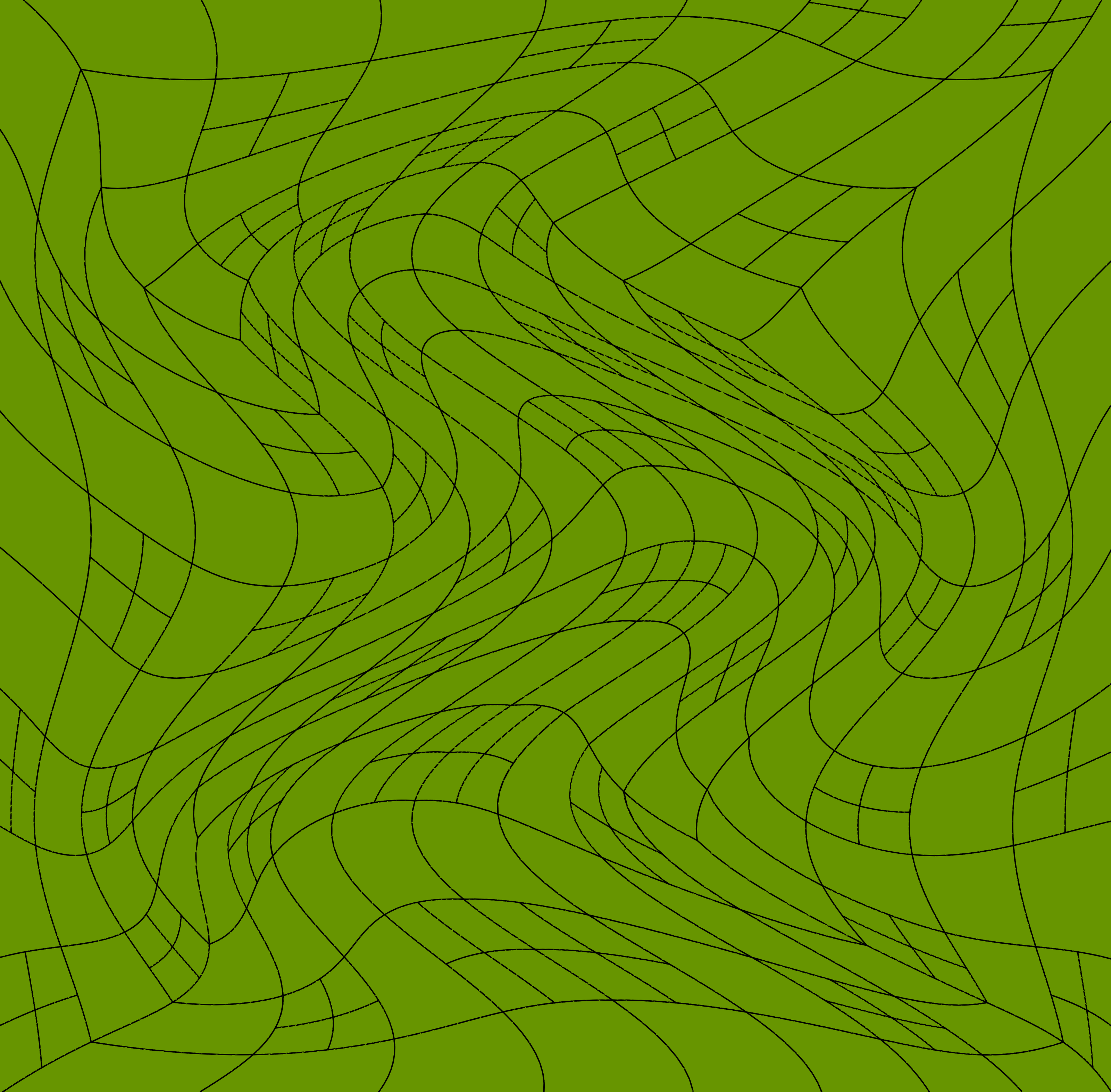}}\\
(a) & (b)
\end{tabular}
\caption{\label{fig:free.stream}Density plots of free stream tests with mesh and solution polynomial degree $N = 6$ at $t = 10$ on (a) mesh with curved boundaries, (b) mesh with refined curved elements}
\end{figure}

\subsubsection{Rotating Couette flow}\label{sec:couette}
\correction{This test from~\cite{Hesthaven2008} consists of a smooth steady state solution with a rotational velocity between two concentric circles. This is actually a steady solution of both the Navier-Stokes and Euler equations as the viscous forces vanish in the steady state. The simulation is run for a long time to reach steady solution and then measure order of accuracy with non-periodic (Dirichlet) boundary conditions on the curved boundaries prescribed by the steady state solution. The steady state, which is also used to specify the boundary values, is given by
\[
(\rho, u, v, p) = \left(1, -\sin \theta, \cos \theta, 1+\frac{1}{75^2}\left(\frac{r^2}{2} - 32 \ln r - \frac{128}{r^2}\right)\right),\qquad \theta = \tan^{-1}\left(\frac y x\right), r = \sqrt{x^2+y^2}
\]
The mesh is generated using a reference map that is defined as $(\xi, \eta)$ from $(0,1) \mapsto \Omega$ given by
\[
(x,y) = (r \cos \theta, r \sin \theta), \qquad r = R_1 e^{\xi \log(R_2/R_1)}, \theta = 2 \pi \eta, R_1 = 1, R_2 = 4
\]
where the values $R_1, R_2$ are chosen to ensure that the mesh is finer near the inner circle.
The mesh and the grid resolution versus $L^2$ error analysis are shown in Figure~\ref{fig:couette} where the optimal order of accuracy is seen.
\begin{figure}
\centering
\begin{tabular}{cc}
{\includegraphics[width=0.45\textwidth]{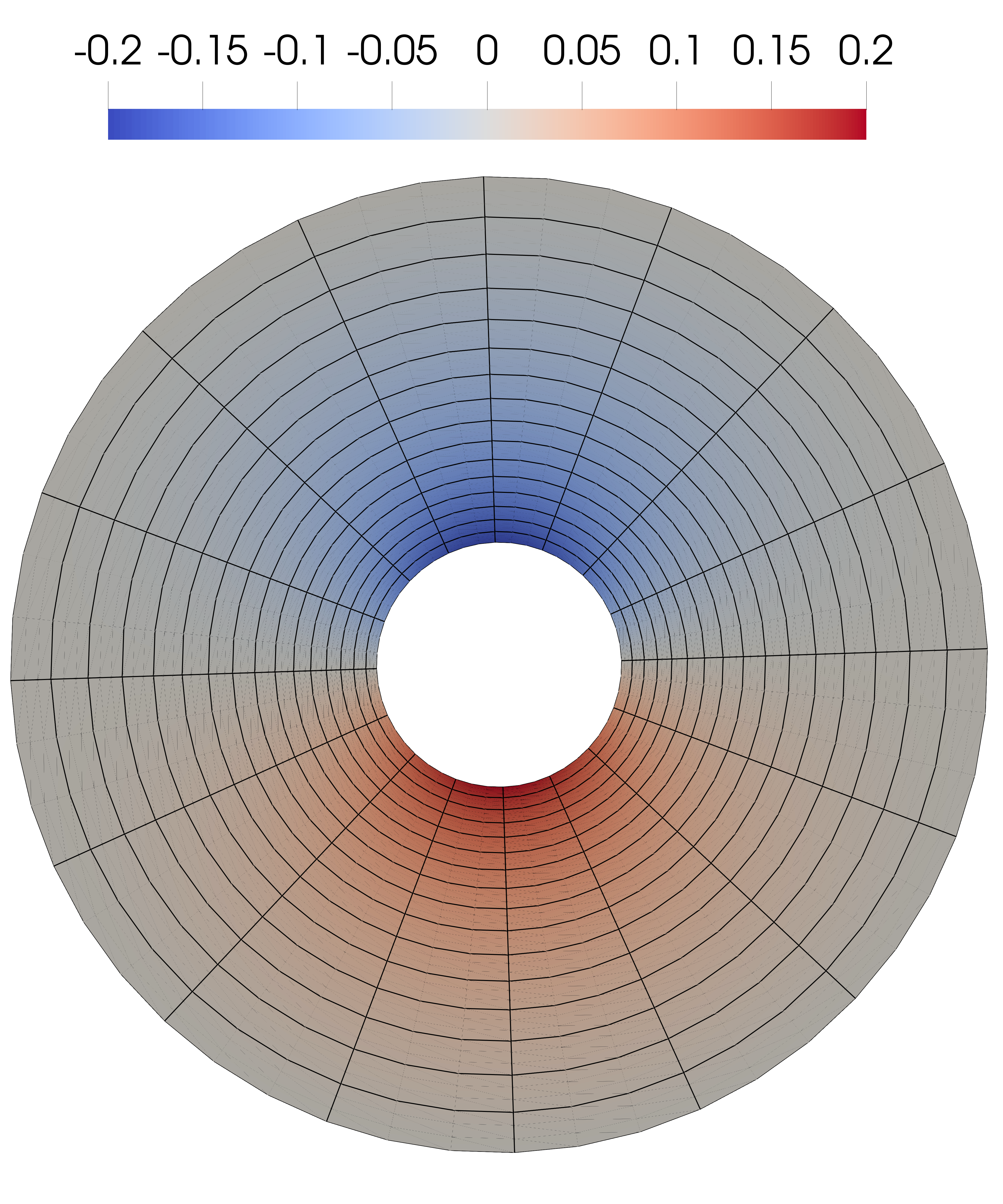}} & {\includegraphics[width=0.48\textwidth]{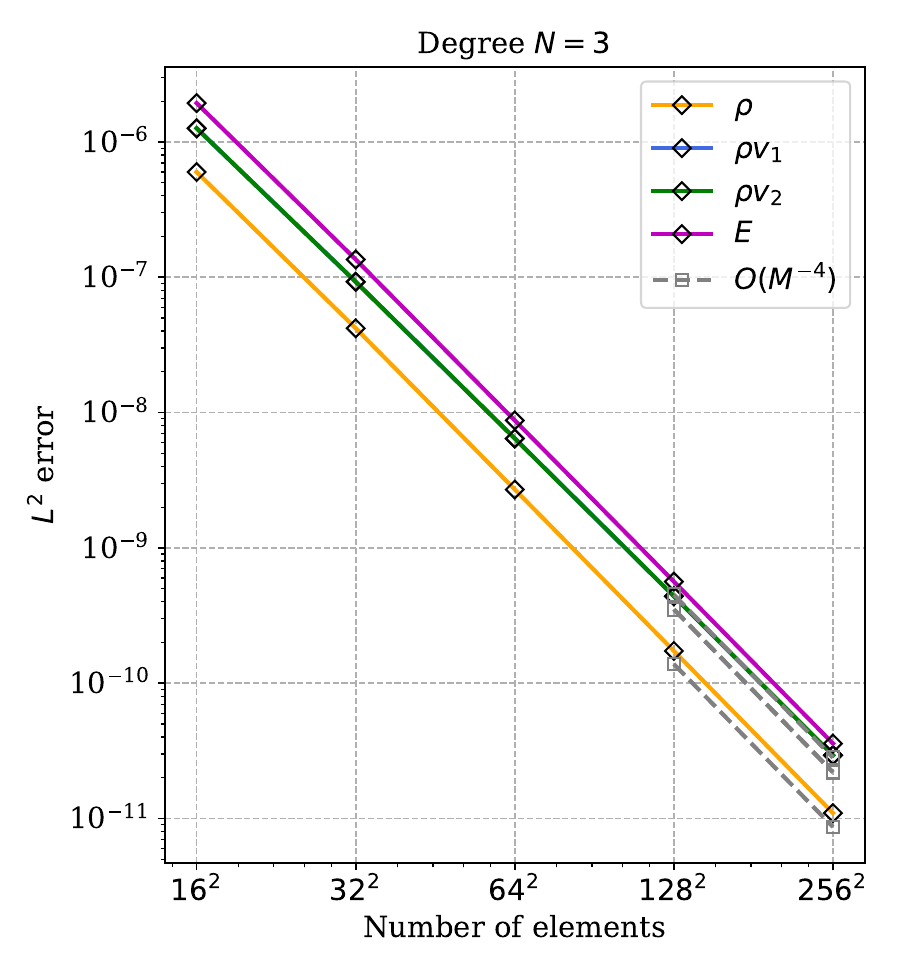}}\\
(a) & (b)
\end{tabular}
\caption{Numerical solution of the Couette flow at $t=1$. (a) The $x$ component of velocity profile, (b) Grid resolution versus $L^2$ error for the conservative variables. \label{fig:couette}}
\end{figure}
}

\subsubsection{Isentropic vortex}\label{sec:isentropic}

This is a test with exact solution taken from~{\cite{henneman2021}} where the domain
is specified by the following transformation from $[0, 1]^2 \to \Omega$
\[ \bx (\xi, \eta) = \left(\begin{array}{c}
\xi L_x - A_x L_y \sin (2 \pi \eta)\\
\eta L_y + A_y L_x \sin (2 \pi \xi)
\end{array}\right) \]
which is a distortion of the square $[0, L_x] \times [0, L_y]$ with sine waves
of amplitudes $A_x, A_y$. Following~{\cite{henneman2021}}, we choose length
$L_x = L_y = 0.1$ and amplitudes $A_x = A_y = 0.1$. The boundaries are set to
be periodic. A vortex with radius $R_v = 0.005$ is initialized in the curved
domain with center $(x_v, y_v) = (L_x / 2, L_y / 2)$. The gas constant is
taken to be $R_{\tmop{gas}} = 287.15$ and specific heat ratio $\gamma = 1.4$
as before. The free stream state is defined by the Mach number $M_0 = 0.5$,
temperature $T_0 = 300$, pressure $p_0 = 10^5$, velocity $u_0 = M_0
\sqrt{\gamma R_{\tmop{gas}} T_0}$ and density $\rho_0 =
\frac{p_0}{R_{\tmop{gas}} T_0}$. The initial condition $\uu_0$ is given by
\begin{equation*}
\begin{gathered}
(\rho, u, v, p) =  \left(
 \rho_0  \left( \frac{T}{T_0} \right)^{\frac{1}{\gamma - 1}},
 u_0  \left( 1 - \beta \frac{y - y_v}{R_v} e^{\frac{- r^2}{2}} \right),
u_0 \beta \frac{x - x_v}{R_v} e^{\frac{-r^2}{2}}, \rho (x, y) R_{\tmop{gas}} T
\right) \\
T(x, y) = T_0 - \frac{(u_0 \beta)^2}{2 C_p} e^{- r^2},\qquad r = \sqrt{(x - x_v)^2 + (y - y_v)^2} / R_v
\end{gathered}
\end{equation*}
where $C_p = R_{\tmop{gas}} \gamma / (\gamma - 1)$ is the heat capacity at constant pressure and $\beta = 0.2$ is the vortex strength. The vortex moves in the positive $x$ direction with speed $u_0$ so that the exact solution at time $t$ is $\uu(x,y,t) = \uu_0(x-u_0 t, y)$ where $\uu_0$ is extended outside $\Omega$ by periodicity. We simulate the propagation of the vortex for one time period $t_p = L_x / u_0$ and perform numerical convergence analysis for degree $N=3$ in Figure~\ref{fig:isentropic}b, showing optimal rates in grid versus $L^2$ error norm for all the conserved variables.

\begin{figure}
\centering
\begin{tabular}{cc}
{\includegraphics[width=0.51\textwidth]{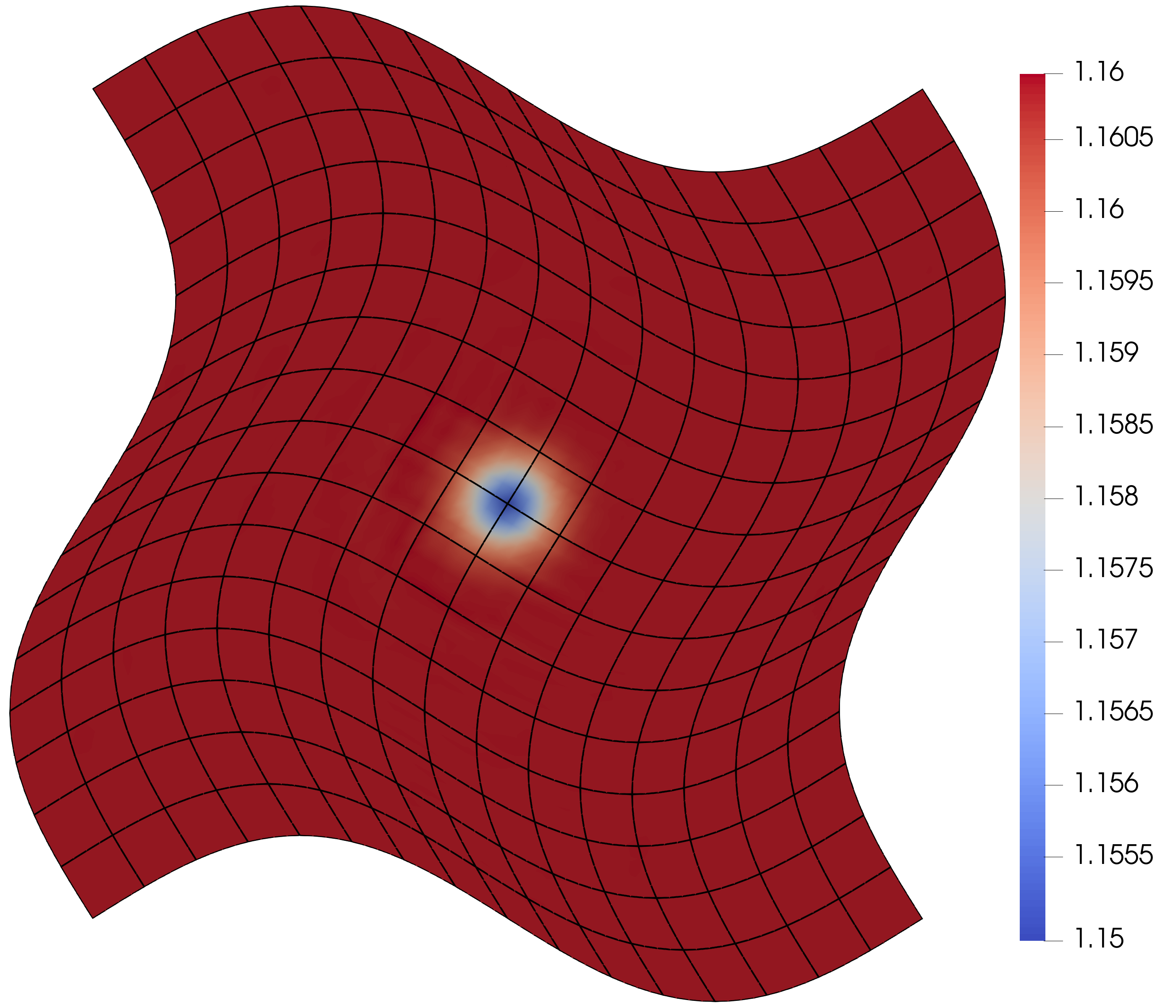}} & {\includegraphics[width=0.43\textwidth]{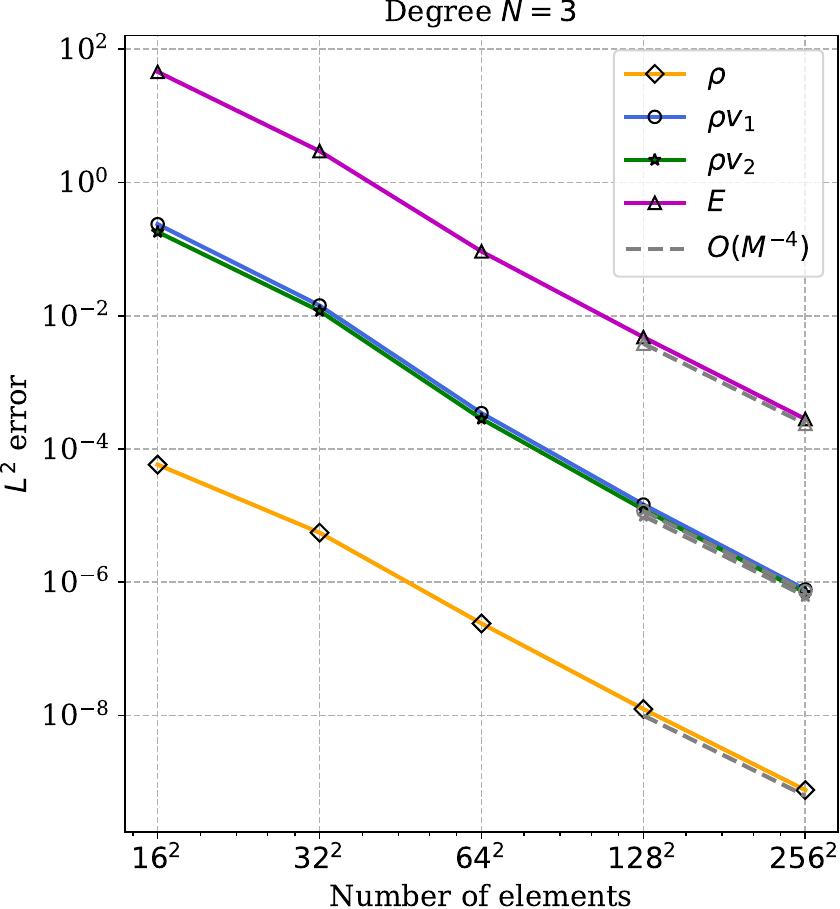}}\\
(a) & (b)
\end{tabular}
\caption{Convergence analysis for isentropic vortex problem with polynomial degree $N = 3$. (a) Density plot, (b) $L^2$ error norm of conserved variables\label{fig:isentropic}}
\end{figure}

\subsubsection{Supersonic flow over cylinder}\label{sec:cylinder}

Supersonic flow over a cylinder is computed at a free stream Mach number of $3$ with the initial condition
\[
(\rho, u, v, p) = (1.4, 3, 0, 1)
\]
Solid wall boundary conditions are used at the top and bottom boundaries. A bow shock forms which reflects across the solid walls and interacts with the small vortices forming in the wake of the cylinder.  The setup of L{\"o}hner's smoothness indicator~\eqref{eq:lohner.ind} is taken from an example of \tmverbatim{Trixi.jl}~{\cite{Ranocha2021}}
\[
(\tmverbatim{base{\_}level}, \tmverbatim{med{\_}level}, \tmverbatim{max{\_}level}) = (0, 3, 5), \qquad
(\tmverbatim{{med}{\_}{threshold}}, \tmverbatim{max{\_}threshold}) = (0.05, 0.1)
\]
where $\tmverbatim{base\_level} = 0$ refers to mesh in Figure~\ref{fig:supersonic.cylinder}a. The flow consists of a strong a shock and thus the positivity limiter had to be used to enforce admissibility. The flow behind the cylinder is highly unsteady, with reflected shocks and vortices interacting continuously. The density profile of the numerical solution at $t = 10$ is shown in Figure~\ref{fig:supersonic.cylinder} with mesh and solution polynomial degree $N = 4$ using L{\"o}hner's indicator~\eqref{eq:lohner.ind} for AMR. The AMR indicator is tracing the shocks and the vortex structures forming in the wake well. The initial mesh has 561 elements which first increase to 63000 elements followed by a fall to 39000 elements and then a steady increase to the peak of 85000 elements from which it steadily falls to 36000 elements by the end of the simulation. The mesh is refined or coarsened once every 100 time steps. In order to capture the same effective refinement, a uniform mesh will require 574464 elements.

\begin{figure}
\centering
\begin{tabular}{c}
{\includegraphics[width = 0.7\textwidth]{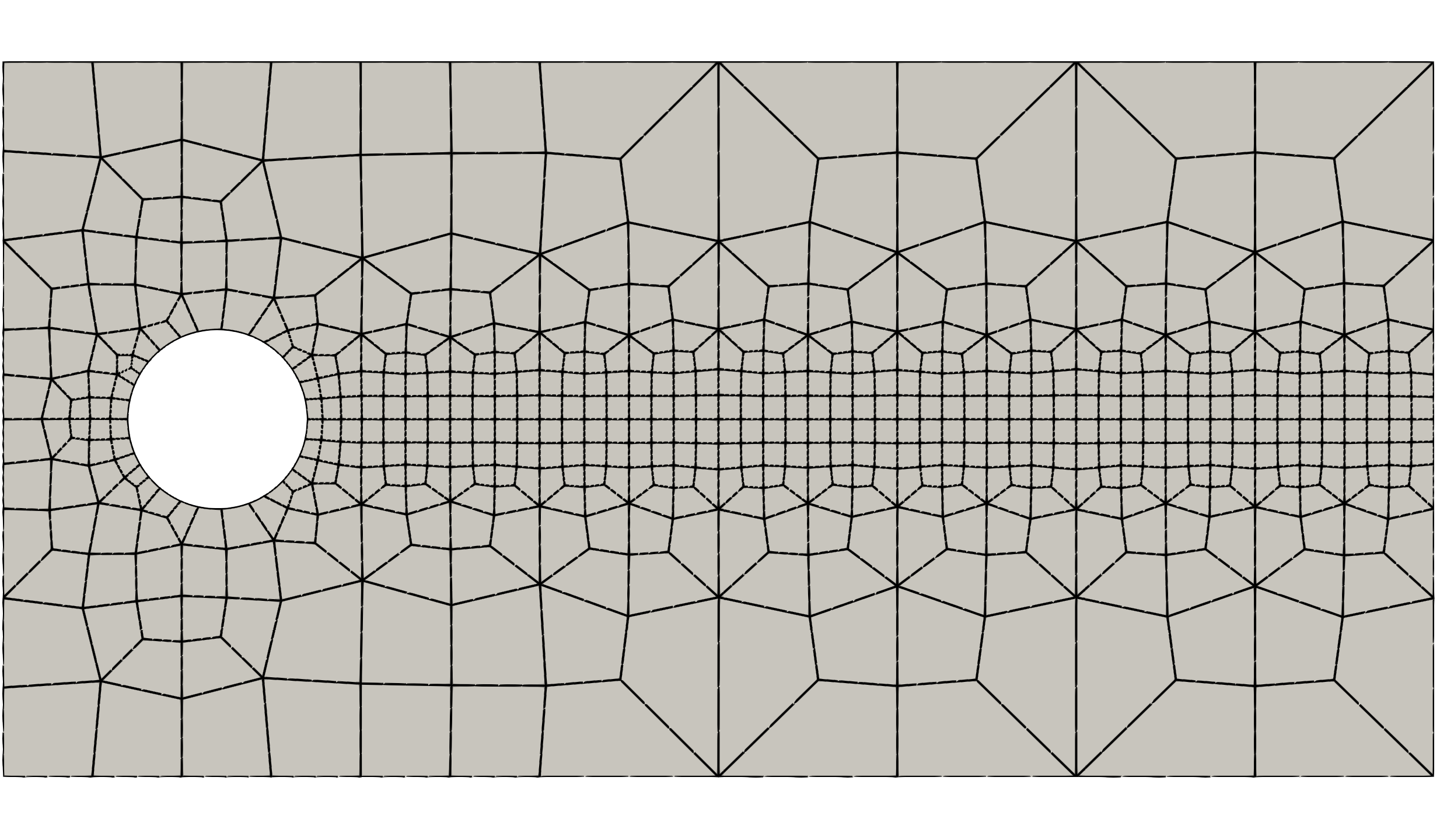}}\\
(a)\\
{\includegraphics[width = 0.7\textwidth]{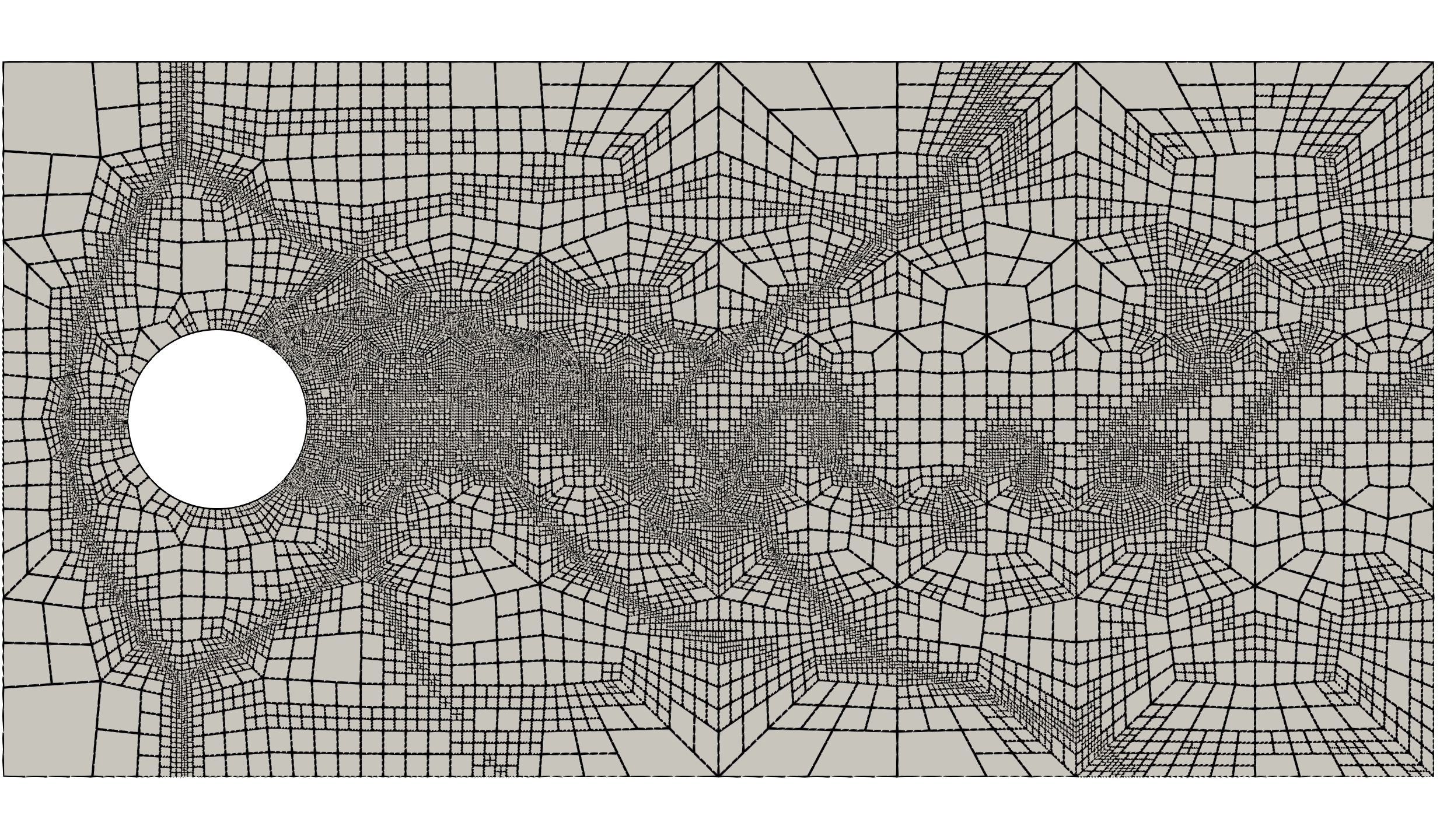}}\\
(b)\\
{\includegraphics[width = 0.7\textwidth]{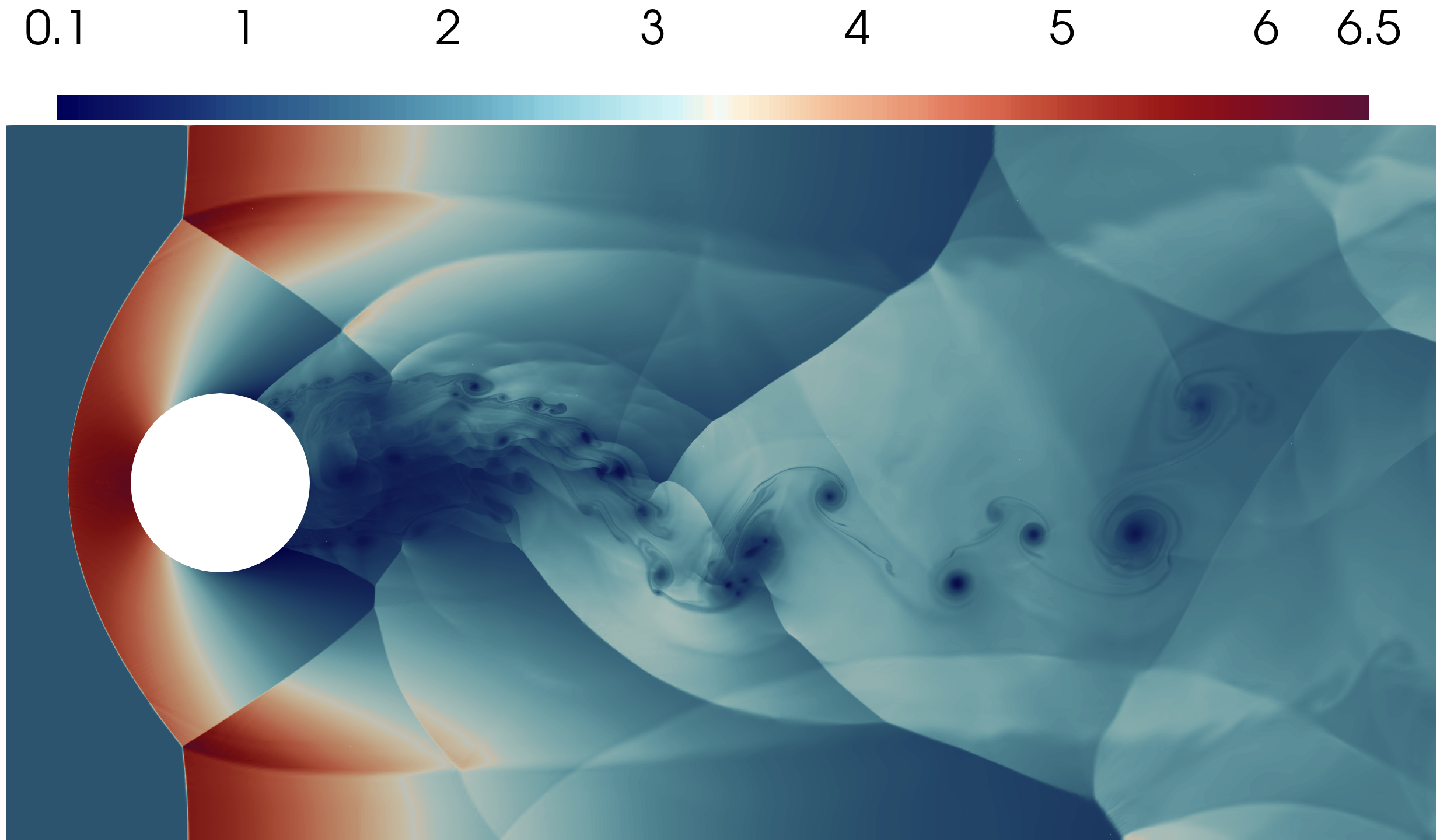}}\\
(c)
\end{tabular}
\caption{\label{fig:supersonic.cylinder}Mach 3 flow over cylinder using solution and mesh polynomial degree $N = 4$ at $t = 10$ (a) Initial mesh, (b) adaptively refined mesh at final time, (c) density plot at final time}
\end{figure}

\subsubsection{Inviscid bow shock upstream of a blunt body}\label{sec:blunt}

This test simulates steady supersonic flow over a blunt body and is taken from~{\cite{henneman2021}} which followed the description proposed by the high order computational fluid dynamics workshop~{\cite{Cenaero2017}}. The domain, also shown in Figure~\ref{fig:blunt} consists of a left and a right boundary. The left boundary is an arc of a circle with origin $(3.85, 0)$ and radius $5.9$ extended till $x = 0$ on both ends. The right boundary consists of (a) the blunt body and (b) straight-edged outlets. The straight-edged outlets are $\{ (0, y) : | y | > 0.5 \}$ extended till the left boundary arc. The blunt body consists of a front of length $1$ and two quarter circles of radius $0.5$. The domain is initialized with a Mach 4 flow, which is given in primitive variables by
\begin{equation}
(\rho, u, v, p) = (1.4, 4, 0, 1)
\end{equation}
The left boundary is set as supersonic inflow, the blunt body is a reflecting wall and the straight edges at $x = 0$ are supersonic outflow boundaries. L{\"o}hner's smoothness indicator~\eqref{eq:lohner.ind} for AMR is set up as
\[
(\tmverbatim{base{\_}level}, \tmverbatim{med{\_}level}, \tmverbatim{max{\_}level}) = (0, 1, 2), \qquad
(\tmverbatim{{med}{\_}{threshold}}, \tmverbatim{max{\_}threshold}) = (0.05, 0.1)
\]
where $\tmverbatim{base\_level} = 0$ refers to mesh in Figure~\ref{fig:blunt}a. Since this is a test case with a strong bow shock, the positivity limiter had to be used to enforce admissibility. The pressure obtained with polynomial degree $N=4$ is shown in Figure~\ref{fig:blunt} with adaptive mesh refinement performed using L{\"o}hner's smoothness indicator~\eqref{eq:lohner.ind} where the AMR procedure is seen to be refining the mesh in the region of the bow shock. The initial mesh~(Figure~\ref{fig:blunt}a) has 244 elements which steadily increases to ${\sim} 1600$ elements till $t\approx 1.5$ and then remains nearly constant as the solution reaches steady state. The mesh is adaptively refined or coarsened at every time step.
\begin{figure}
\centering
\begin{tabular}{cccc}
{\includegraphics[width = 0.095\textwidth]{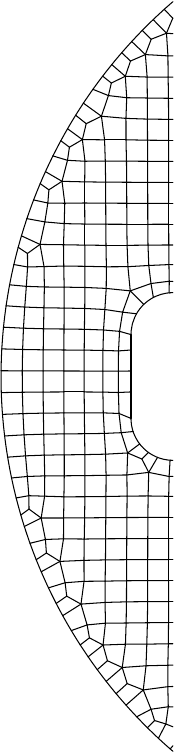}} \qquad \qquad
&
{\includegraphics[width = 0.095\textwidth]{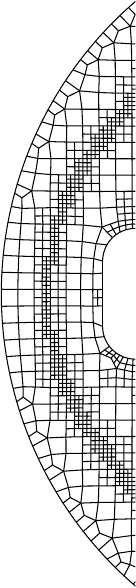}} \qquad \qquad
&
{\includegraphics[width = 0.21\textwidth]{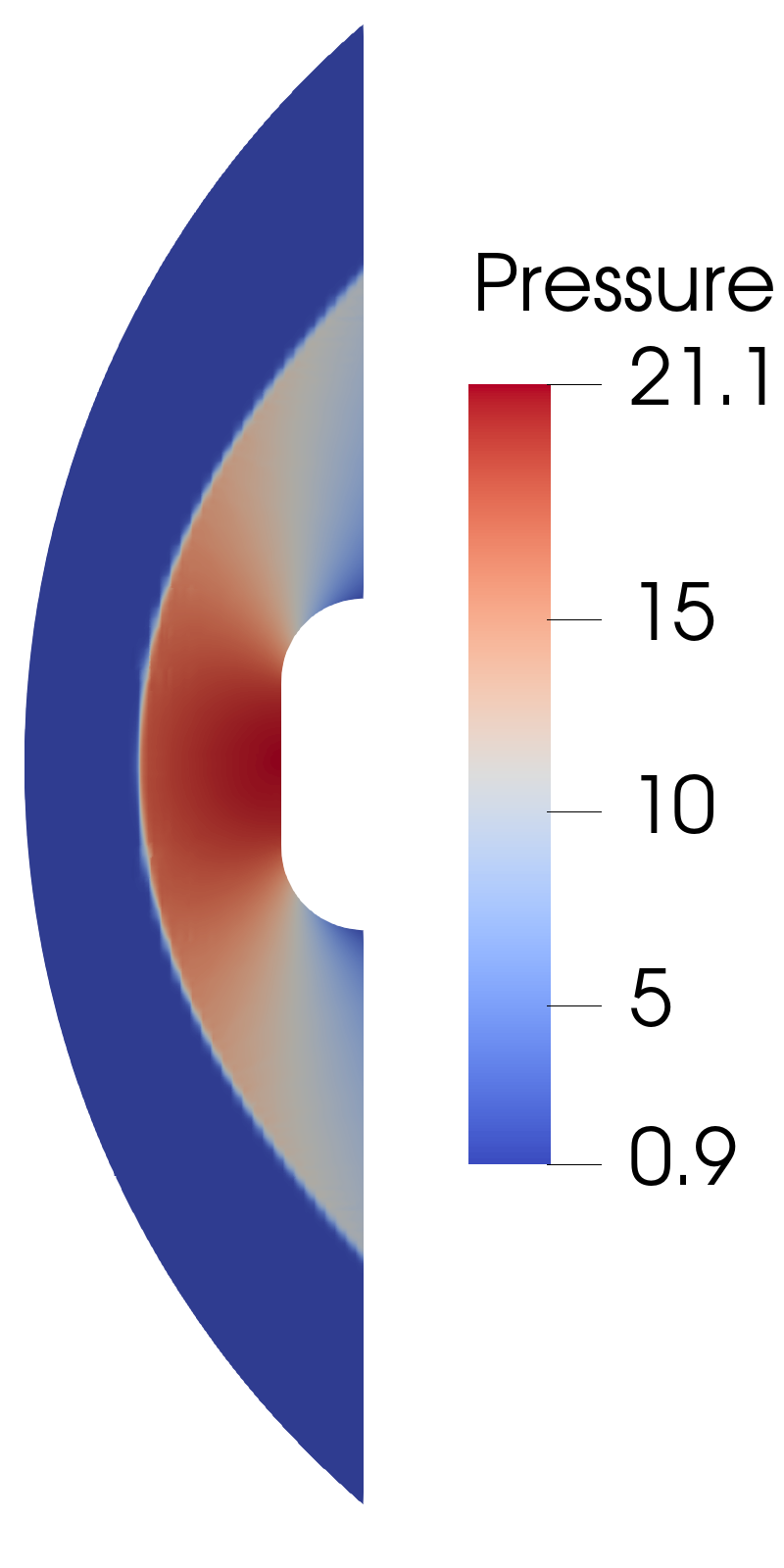}} \qquad
&
{\includegraphics[width = 0.21\textwidth]{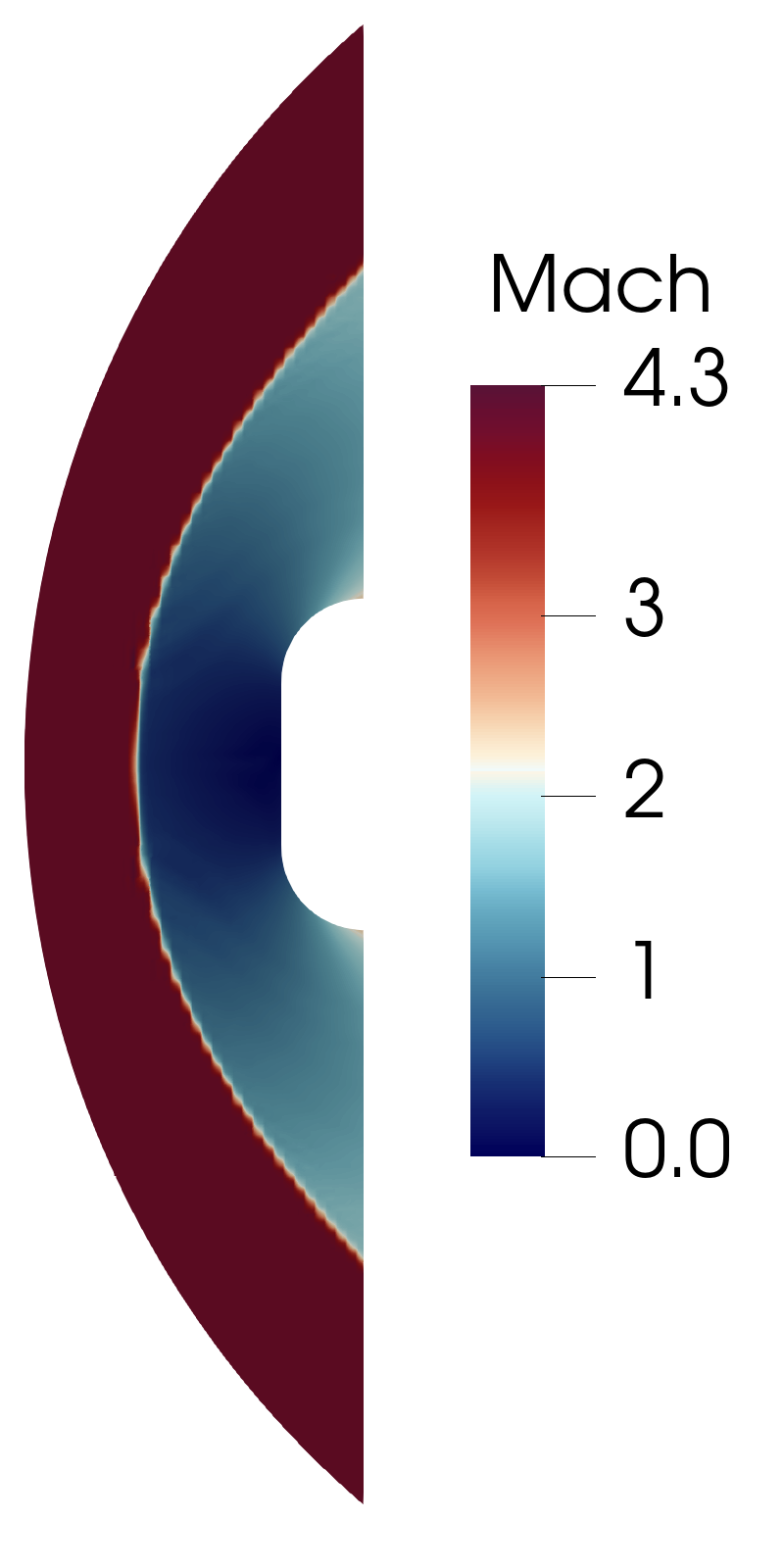}}
\\
(a) & (b) & (c) & (d)
\end{tabular}
\caption{Mach 4 flow over blunt body using polynomial degree $N=4$ showing (a) initial mesh, (b) adaptively refined mesh, (c) pressure plot, (d) Mach number plot \label{fig:blunt} }
\end{figure}

\subsubsection{Transonic flow over NACA0012 airfoil}\label{sec:naca}

This is a steady transonic flow over the symmetric NACA0012 airfoil. The initial condition is taken to have Mach number $M_0 = 0.85$ and it is given in primitive variables as
\[
(\rho, u, v, p) =
\left(
\frac{p_0}{T_0R}, U_0 \cos \theta, U_0 \sin \theta,
p_0
\right)
\]
where $p_0 = 1, T_0 = 1, R = 287.87, \theta = \pi / 180$, $U_0 = M_0 c_0$ and sound speed $c_0 = \sqrt{\gamma p_0 / \rho_0}$. The airfoil is of length $1$ unit located in the rectangular domain $[-20,20]^2$ and the initial mesh has 728 elements. We run the simulation with mesh and solution polynomial degree $N = 6$ using L{\"o}hner's smoothness indicator~\eqref{eq:lohner.ind} for AMR with the setup
\[
(\tmverbatim{base{\_}level}, \tmverbatim{med{\_}level}, \tmverbatim{max{\_}level}) = (1, 3, 4), \qquad
(\tmverbatim{{med}{\_}{threshold}}, \tmverbatim{max{\_}threshold}) = (0.05, 0.1)
\]
where $\tmverbatim{base\_level} = 1$ refers to the mesh in Figure~\ref{fig:naca.mesh}a . In Figure~\ref{fig:naca.mesh}, we show the initial and adaptively refined mesh. In Figure~\ref{fig:naca}, we show the Mach number and compare the coefficient of pressure $C_p$ on the surface of airfoil with SU2~{\cite{su2}} results, seeing reasonable agreement in terms of the values and shock locations. The AMR procedure is found to steadily increase the number of elements till they peak at ${\sim} 4200$ and decrease to stabilize at ${\sim} 3750$; the region of the shocks is being refined by the AMR process. The mesh is adaptively refined or coarsened once every 100 time steps. In order to capture the same effective refinement, a uniform mesh will require 186368 elements.
\begin{figure}
\centering
\begin{tabular}{cc}
{\includegraphics[width=0.47\textwidth]{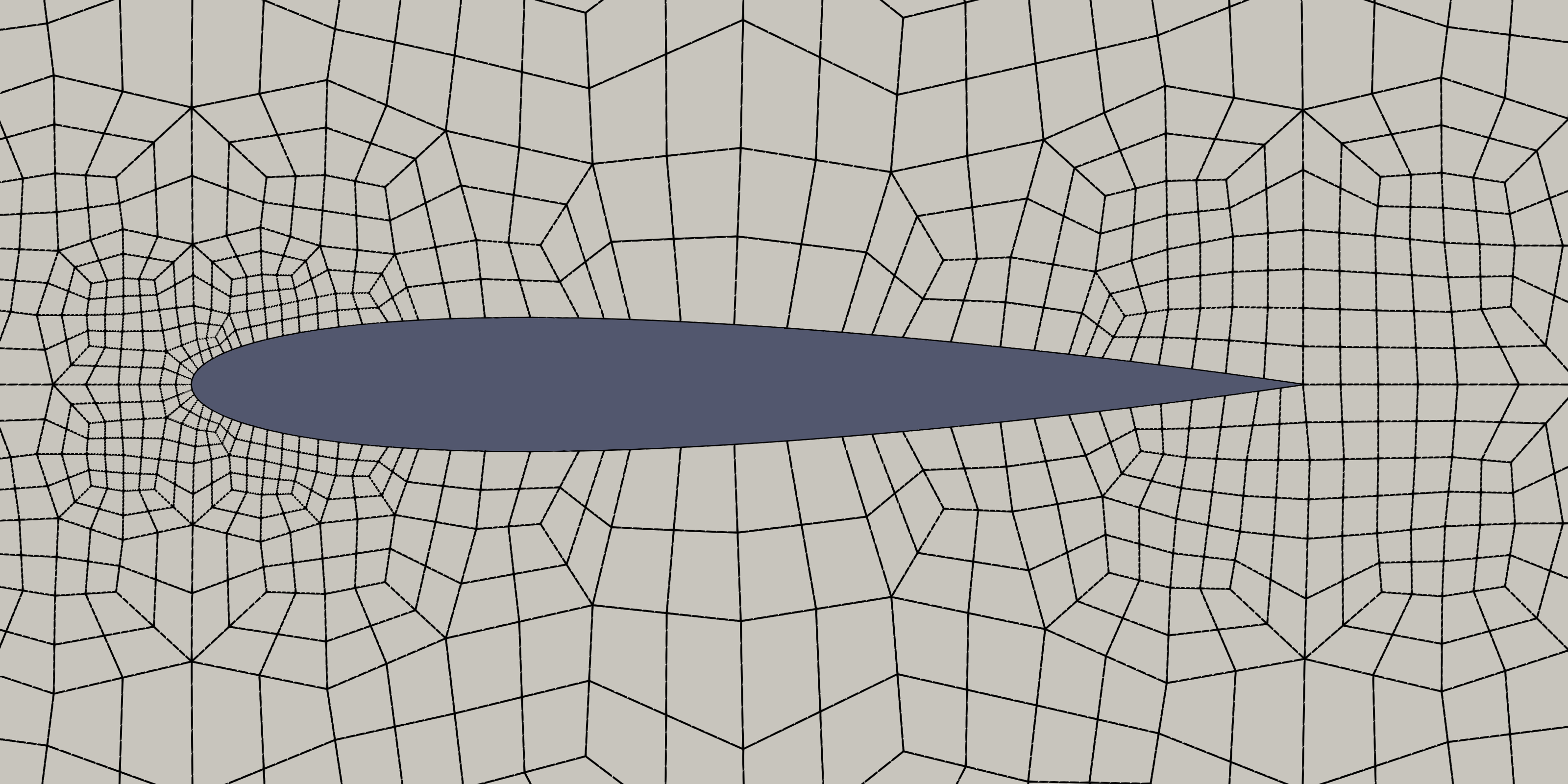}}
&
\includegraphics[width=0.47\textwidth]{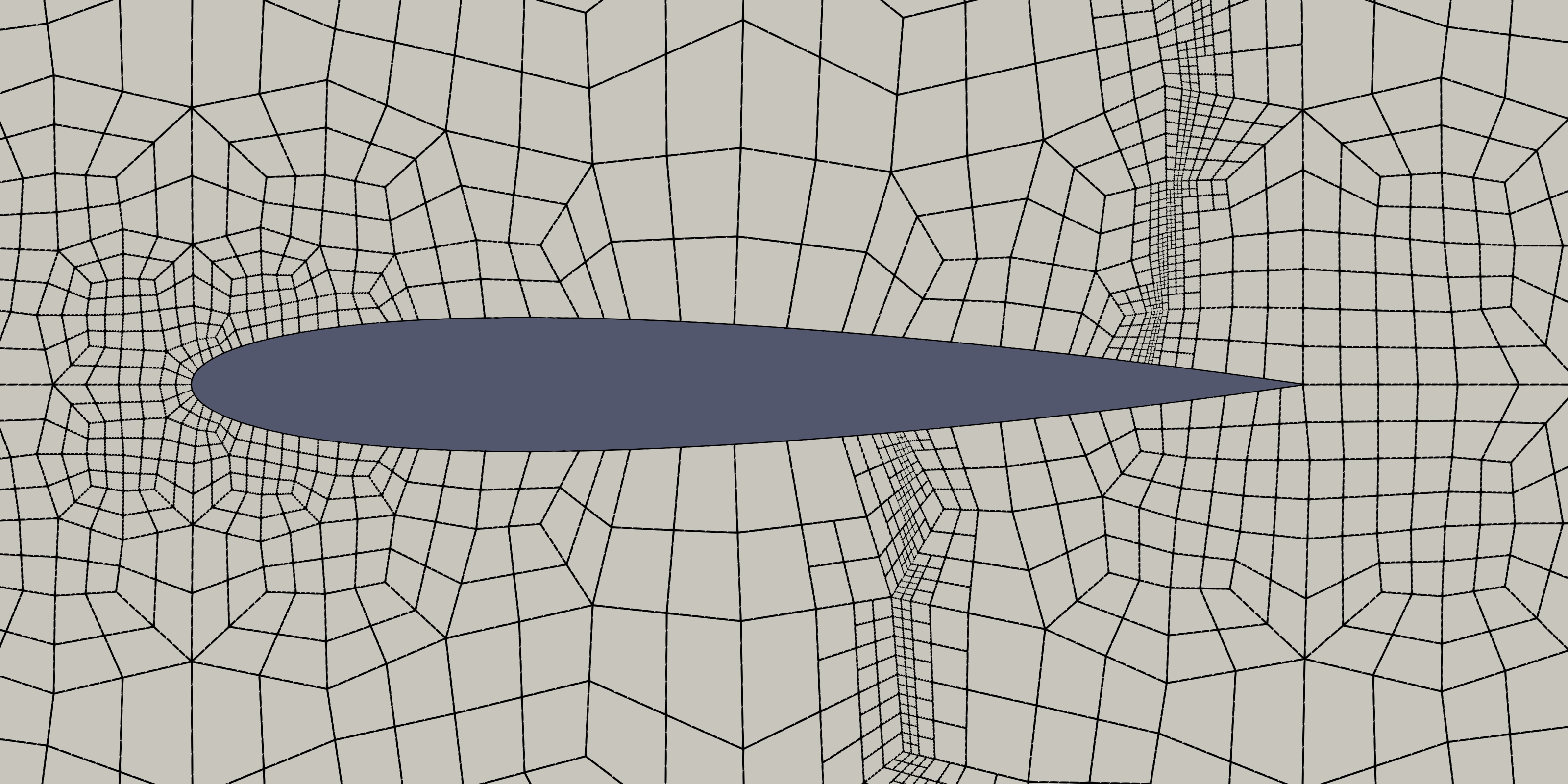} \\
(a) & (b)
\end{tabular}
\caption{\label{fig:naca.mesh}Meshes for transonic flow over NACA0012 airfoil. (a) Initial mesh (b) adaptively refined mesh}
\end{figure}

\begin{figure}
\begin{tabular}{cc}
{\includegraphics[width=0.52\textwidth]{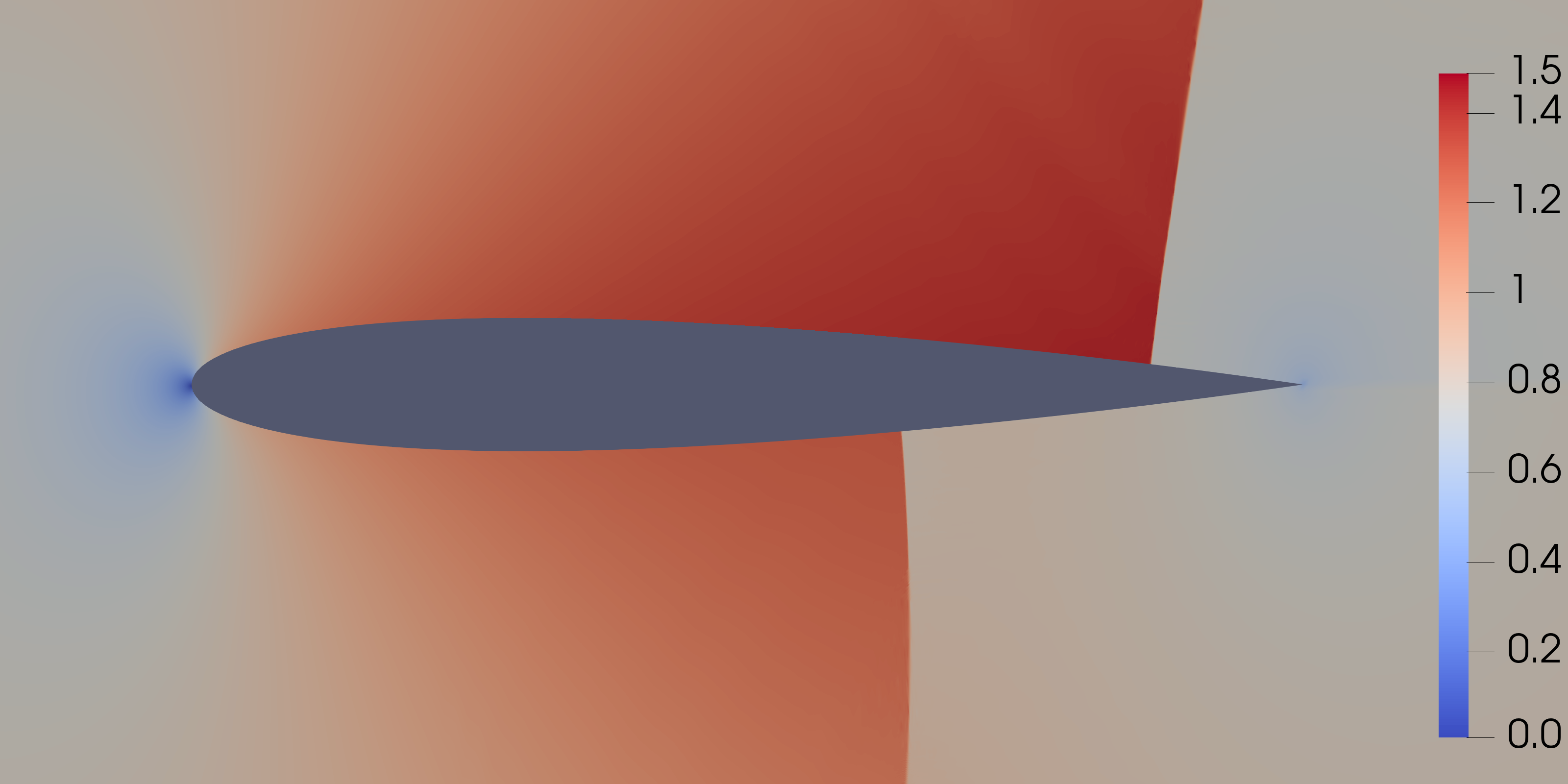}}
 &
\includegraphics[width=0.4\textwidth]{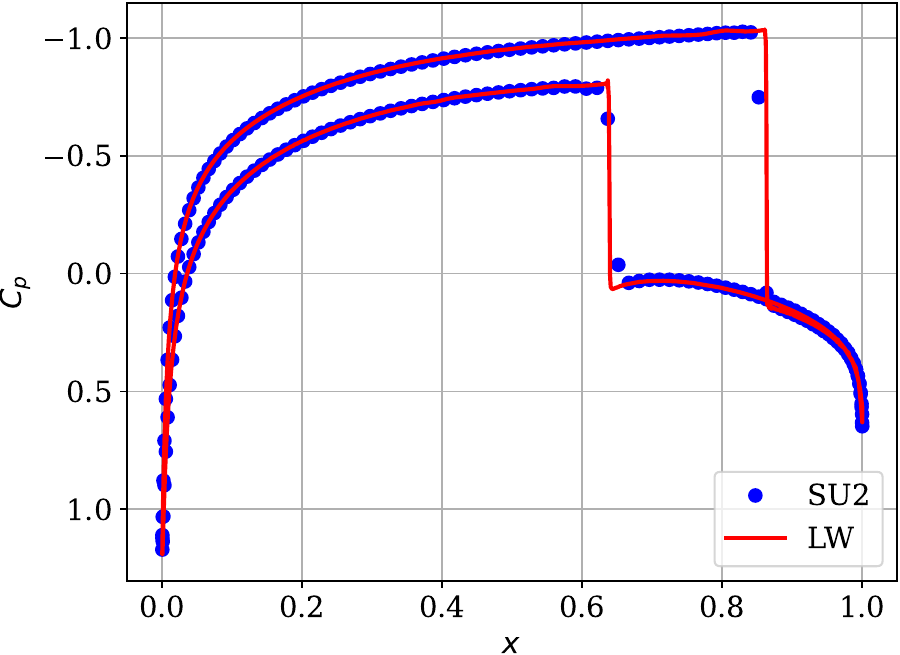} \\
(a) & (b)
\end{tabular}
\caption{\label{fig:naca}Transonic flow over airfoil using degree $N=6$ on adapted mesh (a) Mach number (b)
Coefficient of pressure on the surface of the airfoil}
\end{figure}

\subsection{Performance comparison of time stepping schemes}
In Table~\ref{tab:cfl.error}, we show comparison of total time steps needed by error~ (Algorithm~\ref{alg:time.stepping}) and CFL~\eqref{eq:cfl.time.step} based time stepping methods for test cases where non-Cartesian meshes are used. The total time steps give a complete description of the cost because our experiments have shown that error estimation procedure only adds an additional computational cost of ${\sim}4\%$. The relative and absolute tolerances $\tau_a, \tau_r$ in~\eqref{eq:error.estimator} are taken to be the same, and denoted \tmverbatim{tolE}. The iterations which are redone because of error or admissibility criterion in Algorithm~\ref{alg:time.stepping} are counted as \textit{failed} (shown in Table~\ref{tab:cfl.error} in red) while the rest as \textit{successful} (shown in Table~\ref{tab:cfl.error} in blue). The comparisons are made between the two time stepping schemes as follows - the constant $\Cs$ in~\eqref{eq:cfl.time.step} is experimentally chosen to be the largest which can be used without admissibility violation while error based time stepping is shown with \tmverbatim{tolE = 1e-6} and the best tolerance for the particular test case (which is either \tmverbatim{1e-6} or \tmverbatim{5e-6}). Note that the choice of \tmverbatim{tolE = 1e-6} is made in all the results shown in previous sections. A poor quality (nearly degenerate) mesh~(Figure~\ref{fig:blunt}b) was used in the flow over blunt body~(Section~\ref{sec:blunt}) and thus the CFL based scheme could not run till the final time $t=10$ without admissibility violation for any choice of $\Cs$. However, the error-based time stepping scheme is able to finish the simulation by its ability to redo time steps; although there are many failed time steps as is to be expected. The error-based time stepping scheme is giving superior performance with \tmverbatim{tolE = 1e-6} for the supersonic flow over cylinder and transonic flow over airfoil (curved meshes tests) with ratio of total time steps being \correction{1.51 and 1.55} respectively. However, for the forward facing step test with a straight sided quadrilateral mesh, error based time stepping with \tmverbatim{tolE = 1e-6} takes more time steps than the fine-tuned CFL based time stepping. However, increasing the tolerance to \tmverbatim{tolE = 5e-6} gives the same performance as the CFL based time stepping. By using \tmverbatim{tolE = 5e-6}, the performance of supersonic flow cylinder can be further obtained to get a ratio of \correction{2.08}. These results show the robustness of error-based time stepping and even improved efficiency in meshes with curved elements.
\begin{table}
\begin{tabular}{|c|c|c|c|c|c|}
\hline
& \begin{tabular}{c}
CFL based~\eqref{eq:cfl.time.step}\\
\end{tabular} &
\begin{tabular}{c}
Error Based (Alg. \ref{alg:time.stepping}) \\
(\tmcolor{blue}{Successful} + \tmcolor{red}{failed})
\end{tabular}
&  &  & Ratio\\
\hline
&  & \tmverbatim{tolE=1e-6} & {$\frac{\text{CFL}}{\tmverbatim{tolE=1e-6}}$} & Best
\tmverbatim{tolE} & {$\frac{\text{CFL}}{\text{Best \tmverbatim{tolE}}}$}\\
\hline
FF Step \eqref{sec:forward.step} & 5679021 &
\begin{tabular}{c}
7633954\\
$(\tmcolor{blue}{7633949} + \tmcolor{red}{5})$
\end{tabular} & 0.74  & \begin{tabular}{c}
5598338 \tmverbatim{(5e-6)}\\
(\tmcolor{blue}{5598328} + \tmcolor{red}{10})
\end{tabular} & 1.01\\
\hline
Cylinder (\ref{sec:cylinder}) & 1601725 &
\begin{tabular}{c}
1057844\\
(\tmcolor{blue}{1057710} + \tmcolor{red}{134})
\end{tabular} & 1.51 & \begin{tabular}{c}
768835 \tmverbatim{(5e-6)}\\
(\tmcolor{blue}{760051}+\tmcolor{red}{8784})
\end{tabular} & 2.08\\
\hline
Blunt body (\ref{sec:blunt}) & -  &
\begin{tabular}{c}
30527\\
(\tmcolor{blue}{29766} + \tmcolor{red}{761})
\end{tabular} & - & \begin{tabular}{c}
30527 \tmverbatim{(1e-6)} \\
(\tmcolor{blue}{29766} + \tmcolor{red}{761})
\end{tabular} & - \\
\hline
NACA0012~(\ref{sec:naca})& 7530003  & \begin{tabular}{c}
4842326\\
(\tmcolor{blue}{4842309} + \tmcolor{red}{17})
\end{tabular} & 1.55 & \begin{tabular}{c}
4842326 \tmverbatim{(1e-6)}\\
(\tmcolor{blue}{4842309} + \tmcolor{red}{17})
\end{tabular}& 1.55 \\
\hline
\end{tabular}
\caption{\label{tab:cfl.error} Number of time steps comparing error and CFL
based methods}
\end{table}

\section{Summary and conclusions}\label{sec:conclusions}

The Lax-Wendroff Flux Reconstruction (LWFR) of~{\cite{babbar2022}} has been extended to curvilinear  and dynamic, locally adapted meshes. On curvilinear meshes, it is shown that satisfying the standard metric identities gives free stream preservation for the LWFR scheme. The subcell based blending scheme of~{\cite{babbar2023admissibility}} has been extended to curvilinear meshes along with the provable admissibility preservation of~{\cite{babbar2023admissibility}} based on the idea of appropriately choosing the \tmtextit{blended numerical flux}~{\cite{babbar2023admissibility}} at the element interfaces. Adaptive Mesh Refinement has been introduced for LWFR scheme using the Mortar Element Method (MEM) of~{\cite{Kopriva1996}}.
Fourier stability analysis to compute the optimal CFL number as in~{\cite{babbar2022}} is based on uniform Cartesian meshes and does not apply to curvilinear grids. Thus, in order to use a wave speed based time step computation, the CFL number has to be fine tuned for every problem, especially for curved grids. In order to decrease the fine-tuning process, an embedded errror-based time step computation method was introduced for LWFR by taking difference between two element local evolutions of the solutions using the local time averaged flux approximations - one which is order $N + 1$ and the other truncated to be order $N$. This is the first time error-based time stepping has been introduced for a single stage evolution method for solving time dependent equations. Numerical results using compressible Euler equations were shown to validate the claims. It was shown that free stream condition is satisfied on curvilinear meshes even with non-conformal elements and that the LWFR scheme shows optimal convergence rates on domains with curved boundaries and meshes. The AMR with shock capturing was tested on various problems to show the scheme's robustness and capability to automatically refine in regions comprising of relevant features like shocks and small scale structures. The error based time stepping scheme is able to run with fewer time steps in comparison to the CFL based scheme and with less fine tuning.
\section*{Acknowledgments}

The work of Arpit Babbar and Praveen Chandrashekar is supported by the
Department of Atomic Energy, Government of India, under project
no.~12-R\&D-TFR-5.01-0520.
\section*{Additional data}
The animations of the results presented in this paper can be viewed at
\begin{center}
\href{https://www.youtube.com/playlist?list=PLHg8S7nd3rfvI1Uzc3FDaTFtQo5VBUZER}{www.youtube.com/playlist?list=PLHg8S7nd3rfvI1Uzc3FDaTFtQo5VBUZER}
\end{center}

\appendix
\section{Shock capturing and admissibility
preservation}\label{sec:shock.capturing}

 This
section consists of terminologies for admissibility preservation
and the admissibility preserving blending scheme, some of which is a review
of~{\cite{babbar2023admissibility}}.

\subsection{Admissibility preservation}

For the Euler's equations, since negative density and pressure are
nonphysical, an \tmtextit{admissible solution} is one that belongs to the
\tmtextit{admissible set} $\left\{ \uu : \rho ( \uu ), p ( \uu
) > 0 \right\}$. Since the admissibility preservation approach used in
this work can be used for more general problems, we introduce the general
notations here.

Let $\Uad \subset \re^p$ denote the convex set in which physically correct
solutions of the general conservation law~\eqref{eq:con.law} must belong; we suppose that it
can be written in terms of $K$ constraints as
\begin{equation}
\label{eq:uad.form} \Uad = \{ \uu \in \re^p : p_k (\uu) > 0, 1 \le k \le K\}
\end{equation}
where each admissibility constraint $p_k$ is concave if $p_j > 0$ for all $j <
k$. For Euler's equations, $K = 2$ and $p_1, p_2$ are density, pressure
functions respectively; the density is clearly a concave function of $\uu$ and if the density is positive then it can be easily verified that the pressure is also a concave
function of the conserved variables. The admissibility preserving property,
also known as {\em convex set preservation property}, of the conservation law can be
written as
\begin{equation}
\label{eq:conv.pres.con.law} \uu (\cdummy, t_0) \in \Uad \qquad
\Longrightarrow \qquad \uu (\cdummy, t) \in \Uad, \qquad t > t_0
\end{equation}
and thus we define an admissibility preserving flux reconstruction scheme as
follows.

\begin{definition}
\label{defn:adm.pres}The flux reconstruction scheme is said to be
admissibility preserving if
\[ \uebp^n \in \Uad, \quad \forall e, \bp \qquad \Longrightarrow \qquad
\uebp^{n + 1} \in \Uad, \quad \forall e, \bp \]
where $\Uad$ is the admissibility set of the conservation law.
\end{definition}

To obtain an admissibility preserving scheme, we exploit the weaker
admissibility preservation in means property.

\begin{definition}
\label{defn:mean.pres}The flux reconstruction scheme is said to be
admissibility preserving in the means if
\[ \uebp^n \in \Uad, \quad \forall e, \bp \qquad \Longrightarrow \qquad
\overline{\uu}_e^{n + 1} \in \Uad, \quad \forall e \]
where $\Uad$ is the admissibility set of the conservation law and
$\overline{\uu}_e$ denotes the element mean~\eqref{defn.mean}.
\end{definition}

\subsection{Blending scheme}

In this section, we explain the blending procedure which obtains admissibility
preservation in means property for LWFR scheme on curvilinear grids using
Gauss-Legendre-Lobatto solution points. The procedure is very similar to that
of~{\cite{babbar2023admissibility}} for Cartesian grids where Gauss-Legendre
solution points were used.

Let us write the LWFR update equation~\eqref{eq:lw.update} as
\begin{equation}
\Vu^{H, n + 1}_e = \Vu^n_e - \frac{\Delta t}{\Delta x_e}  \VR^H_e
\label{eq:ho.residual}
\end{equation}
where $\Vu_e$ is the vector of nodal values in the element $\Omega_e$. Suppose we also
have a lower order, non-oscillatory scheme available to us in the form
\begin{equation}
\Vu^{L, n + 1}_e = \Vu^n_e - \frac{\Delta t}{\Delta x_e}  \VR^L_e
\end{equation}
Then a blended scheme is given by
\begin{equation}
\Vu^{n + 1}_e = (1 - \alpha_e)  \Vu^{H, n + 1}_e + \alpha_e  \Vu^{L, n +
1}_e = \Vu^n_e - \frac{\Delta t}{\Delta x_e}  [(1 - \alpha_e) \VR^H_e +
\alpha_e  \VR^L_e] \label{eq:blended.scheme}
\end{equation}
where $\alpha_e \in [0, 1]$ must be chosen based on some local smoothness
indicator. If $\alpha_e = 0$, then we obtain the high order LWFR scheme, while
if $\alpha_e = 1$ then the scheme becomes the low order scheme that is less
oscillatory. In subsequent sections, we explain the details of the lower order
scheme and the design of smoothness indicators.

\subsubsection{Blending scheme in 1-D}

Let us subdivide each element $\Omega_e = [ x_{\emh}, x_{\eph} ]$
into $N + 1$ subcells associated to the solution points $\{x^e_p, p = 0, 1,
\ldots, N\}$ of the LWFR scheme. Thus, we will have $N + 2$ subfaces within
each element $\Omega_e$ denoted by $\{x^e_{\pph}, p = -1, \ldots, N\}$
where $x_{-\frac{1}{2}}^e = x_{\emh} = x_0^e$, $x_{N + \frac{1}{2}}^e =
x_{\eph} = x_N^e$. For maintaining a conservative scheme, the $p^{\tmop{th}}$
subcell is chosen so that
\begin{equation}
x_{\pph}^e - x_{\pmh}^e = w_p \Delta x_e, \qquad 0 \le p \le N \label{eq:subcell.defn}
\end{equation}
where $w_p$ is the $p^{\tmop{th}}$ quadrature weight associated with the
solution points, and $\Delta x_e = x_{\eph} - x_{\emh}$.
Figure~\ref{fig:subcells} gives an illustration of the subcells for degree $N
= 4$ case.

\begin{figure}
\centering
{\includegraphics[width=0.7\textwidth]{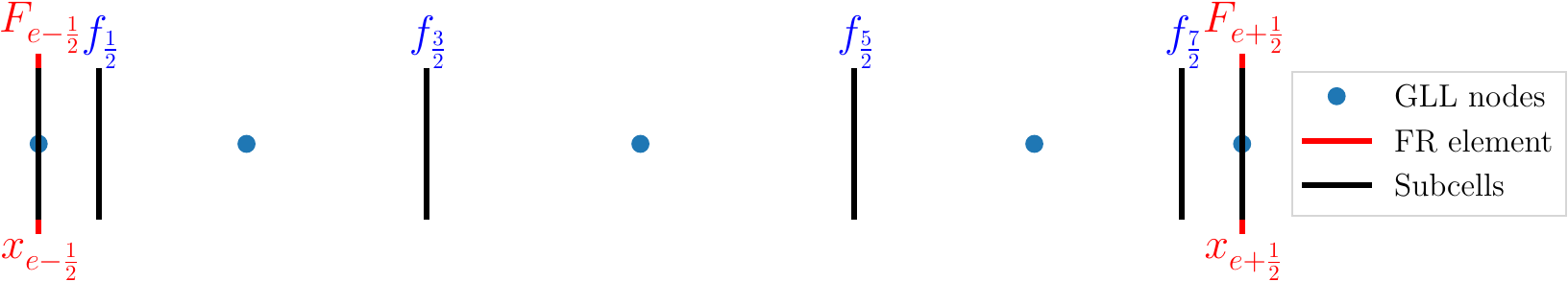}}
\caption{Subcells used by the lower order scheme for degree $N = 4$ \label{fig:subcells}}
\end{figure}

The low order scheme is obtained by updating the solution in each of the
subcells by a finite volume scheme,
\begin{equation}
\label{eq:low.order.update}
\begin{split}
\uez^{L, n + 1} & = \uez^n - \frac{\Delta t}{w_0 \Delta x_e}
[\pf_{\frac{1}{2}}^e - \F_{\emh}] \\
\uep^{L, n + 1} & = \uep^n - \frac{\Delta t}{w_p \Delta x_e}
[\pf_{\pph}^e - \pf_{\pmh}^e], \qquad 1 \le p \le N - 1\\
\ueN^{L, n + 1} & = \ueN^n - \frac{\Delta t}{w_N \Delta x_e}  [\F_{\eph} -
\pf_{\Nmh}^e]
\end{split}
\end{equation}
The fluxes $\pf_{\pph}^e = \pf ( \uu_p^n, \uu^n_{p + 1} )$ are
first order numerical fluxes and $\F_{\eph}$ is the \tmtextit{blended
numerical flux} which is a convex combination of the time averaged numerical
flux~\eqref{eq:rusanov.flux.lw} and a lower order flux $\pf_{\eph} = \pf
(\uu_{\eph}^{n -}, \uu_{\eph}^{n +}) = \pf
(\uu_{e,N}, \uu_{e+1,0})$. The same blended numerical flux
$\F_{\eph}$ is used in the high order LWFR residual~(\ref{eq:ho.residual},
\ref{eq:lwfr.update.curvilinear}); see Remark~1 of~\cite{babbar2023admissibility} for why
it is crucial to do so to ensure conservation. In this work, Rusanov's flux~{\cite{Rusanov1962}} will
be used for the inter-element fluxes and for fluxes in the lower order scheme.
The element mean value obtained by the low order scheme, high order scheme and the blended scheme satisfy the conservative property
\begin{equation}
\overline{\uu}_e^{L, n + 1} = \overline{\uu}_e^{H, n + 1} = \overline{\uu}_e^{n + 1} = \overline{\uu}_e^n - \frac{\Delta t}{\Delta x_e} (\F_{\eph} - \F_{\emh})
\label{eq:low.order.cell.avg.update}
\end{equation}

The inter-element fluxes $\F_{\epmh}$ are used both in the low and high order
schemes at $x_{\eph} = x_N^e, x_{\emh} = x_0^N$ respectively, where both schemes have a solution point and an
element interface. It has to be chosen carefully to balance accuracy, robustness and ensure admissibility preservation. The following natural initial guess is made and then further corrected to enforce admissibility, as explained in Section~\ref{sec:flux.correction}
\begin{equation}
\F_{\eph} = (1 - \alpha_{\eph})  \F_{\eph}^{\tmop{LW}} + \alpha_{\eph}
\pf_{\eph}, \qquad \alpha_{\eph} \in [0, 1] \label{eq:Fblend}
\end{equation}
where $\F_{\eph}^{\tmop{LW}}$ is the high order inter-element time-averaged
numerical flux of the LWFR scheme~\eqref{eq:rusanov.flux.lw} and $\pf_{\eph}$
is a lower order flux at the face $x_{\eph}$ shared between FR elements and
subcells. The coefficient $\alpha_{\eph}$ is given by $\frac{\alpha_e + \alpha_{e+1}}{2}$ where $\alpha_e$ is the blending coefficient computed with a smoothness indicator~(Section~\ref{sec:smooth.ind}). The reader is referred to Section 3.2 of~\cite{babbar2023admissibility} for more details.

\begin{remark}
\label{rmk:common.contri}{\tmdummy}
The contribution to $\VR^L_e, \VR^H_e$ of the flux $\F_{\eph}$ has
coefficients given by $\frac{\Delta t}{w_N \Delta x_e}, \frac{\Delta
t}{\Delta x_e} g_R' (\xi_N)$ respectively, as can be seen
from~(\ref{eq:low.order.update}, \ref{eq:lwfr.update.curvilinear}). Since
we use $g_2$ correction functions with Gauss-Legendre-Lobatto solution
points, we have from equivalence of FR and DG, $g_R' (\xi_N) = \ell_N
(1) / w_N = 1 / w_N$. Thus, the coefficient is the same for both higher
and lower order residuals and we add the contribution without a blending
coefficient. This is different from the case of Gauss-Legendre solution
points used in the blending scheme of~{\cite{babbar2023admissibility}}.
\end{remark}

\subsubsection{Blending scheme for curvilinear grids}

\begin{figure}
\centering
{\includegraphics[width=0.8\textwidth]{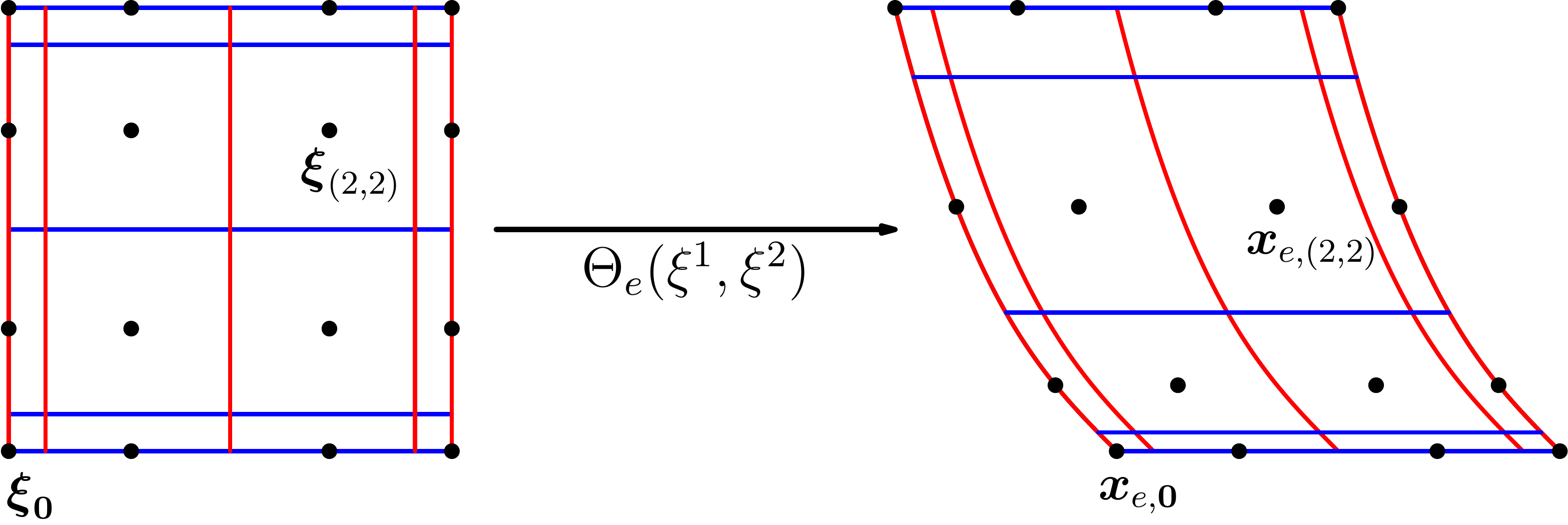}}
\caption{\label{fig:curved.subcells}Subcells in a curved element}
\end{figure}

The subcells for a curved element will be defined by the reference map, as
shown in Figure~\ref{fig:curved.subcells}. As in Appendix B.3
of~{\cite{henneman2021}}, the finite volume formulation on subcells is obtained
by an integral formulation of the transformed conservation
law~\eqref{eq:transformed.conservation.law}. In the reference element, consider subcells $\Cijk$ of size
$w_{\bp} = \prod_{i = 1}^d w_{p_i}$ with solution point $(\xi_{p_i})_{i =
1}^d$ corresponding to the multi-index $\bp = (p_i)_{i = 1}^d$ where $p_i \in
\{ 0, 1 \ldots, N \}$. Fix a subcell $C = \Cijk$ around the solution point $\vxi =
(\xi_{p_i})_{i = 1}^d$ and denote $\vxi_i^{L / R}$ as
in~\eqref{eq:xis.notation}. Integrate the conservation law on the fixed subcell $C$
\[ \int_C J \uu_t  \ud V + \int_C \nabla_{\vxi} \cdot \tf  \ud V = \bzero \]
Next perform one point quadrature in the first term and apply Gauss divergence
theorem on the second term to get
\begin{equation}
J_{e, \bp}  \dv{\uu_{\bp}}{t} w_{\bp} + \int_{\partial C} \tf \cdot \bnr
\ud A = \bzero
\end{equation}
where $\bnr$ is the reference normal vector on the subcell surface. Now
evaluate this surface integral by approximating fluxes in each direction with
numerical fluxes
\begin{equation}
\int_{\partial C} \tf \cdot \bnr  \ud A = \sum_{i = 1}^d
\frac{w_{\bp}}{w_{p_i}}  [ ( \tf^{\delta}_C \cdot \bnr_{R,i}
)^{\ast} ( \vxi_i^R ) + ( \tf^{\delta}_C \cdot \bnr_{L,i}
)^{\ast} ( \vxi_i^L ) ], \qquad \bnr_{R,i} = \be_i, \qquad \bnr_{L,i} = -\be_i
\label{eq:lower.order.flux.curved}
\end{equation}
The explicit lower order method using forward Euler update is given by
\begin{equation}
\uu^{n + 1}_{\bp} = \uu^n_{\bp} - \frac{\Delta t}{J_{e, \bp}}  \sum_{i
= 1}^d \frac{1}{w_{p_i}}  [ ( \tf^{\delta}_{\Cijk} \cdot \bnr_{R,i}
)^{\ast} ( \vxi_i^R ) + ( \tf^{\delta}_{\Cijk} \cdot
\bnr_{L,i} )^{\ast} ( \vxi_i^L ) ]
\label{eq:lower.order.curved}
\end{equation}
For the subcells whose interfaces are not shared by the FR element, the fluxes
are computed, following~{\cite{henneman2021}}, as
\begin{equation}
\label{eq:lo.numflux}
\begin{gathered}
( \tf^{\delta}_{\Cijk} \cdot \bnr_{R,i} )^{\ast} ( \vxi_i^R
) = \left\| (\bn_{R,i})_{\bp} \right\|  \pf^{\ast} \left( \uu_{\bp},
\uu_{\bp_{i +}}, \frac{(\bn_{R,i})_{\bp}}{\left\| (\bn_{R,i})_{\bp} \right\|}
\right) \\
( \tf^{\delta}_{\Cijk} \cdot \bnr_{L,i})^{\ast}
( \vxi_i^L ) = \| (\bn_{L,i})_{\bp} \|  \pf^{\ast} \left(
\uu_{\bp_{i -}}, \uu_{\bp}, \frac{(\bn_{L,i})_{\bp}}{\| (\bn_{L,i})_{\bp}
\|} \right) \\
( \bp_{i \pm} )_m = \begin{cases}
p_m \qquad & m \neq i\\
p_{i \pm 1}  & m = i
\end{cases}
\end{gathered}
\end{equation}
where $(\bn_{s,i})_{\bp}$ is the normal vector of subcell $C_{\bp}$ in direction
$i$ and side $s \in \{ L, R \}$. The numerical fluxes~\eqref{eq:lo.numflux}
are taken to be Rusanov's flux~\eqref{eq:rusanov.flux}
\begin{equation}
\tf^{\ast} ( \uu^-, \uu^+, \bn ) = \frac{1}{2}  [ ( \pf
\cdot \bn ) ( \uu^+ ) + ( \pf \cdot \bn  )
( \uu^- ) ] - \frac{\lambda}{2}  ( \uu^+ - \uu^-
) \label{eq:fo.rusanov}
\end{equation}
At the interfaces shared by FR elements, the first order numerical flux is
computed by setting $\uu^{\pm}$ in~\eqref{eq:fo.rusanov} to element trace
values as in~\eqref{eq:rusanov.flux}. However, the lower order residual needs to be computed using the same inter-element flux as the higher order scheme at interfaces of
the Flux Reconstruction (FR) elements. Thus, for example, for an element $\Omega_e$
at solution point $\vxi = \vxi_{\bp}$ with $\bp = \bzero$, the subcell update
will be given by
\begin{equation}
\uu_{e, \bzero}^{n + 1} = \uu_{e, \bzero}^n - \frac{\Delta t}{J_{e,
\bp}}  \sum_{i = 1}^d \frac{1}{w_{p_i}}  [ (
\tf^{\delta}_{C_{\bzero}} \cdot \bnr_{R,i} )^{\ast} ( \vxi_i^R
) + ( \tF^{\delta}_e \cdot \bnr_{L,i} )^{\ast} ( \vxi_i^L
) ] \label{eq:blended.flux.in.curved}
\end{equation}
where $( \tF^{\delta}_e \cdot \bnr_i )^{\ast} ( \vxi_i^L
)$ is the blended numerical flux and is computed by taking a convex
combination of the lower order flux chosen as in~\eqref{eq:rusanov.flux} and
the time averaged flux~\eqref{eq:rusanov.flux.lw}. An initial guess is made as
in 1-D~\eqref{eq:Fblend} and then further correction is performed to ensure
admissibility, as explained in Section~\ref{sec:flux.correction.curved}. Other
subcells neighbouring the element interfaces will also use the blended
numerical fluxes at the element interfaces and thus have an update
similar to~\eqref{eq:blended.flux.in.curved}. Then, by multiplying each update
equation of each subcell $\bp$ by $w_{\bp}$ and summing over $\bp$, the
conservation property is obtained
\begin{equation}
\overline{\uu}_e^{L, n + 1} = \sum_{\bp} \uebp^{L, n + 1} w_{\bp} =
\overline{\uu}_e^n - \frac{\Delta t}{| \Omega_e |}  \left( \sum_{i =
1}^d \int_{\Oip} ( \tF^{\delta}_e \cdot \bnr_{R,i} )^{\ast}  \ud
S_{\vxi} + \int_{\Oim} ( \tF^{\delta}_e \cdot \bnr_{L,i} )^{\ast}
\ud S_{\vxi} \right) \label{eq:low.order.cell.avg.update.curved}
\end{equation}
Since we also have the conservation property in the higher order
scheme~\eqref{eq:conservation.lw}, the blended scheme will be conservative,
analogous to the 1-D case~\eqref{eq:low.order.cell.avg.update}.

The expressions for normal vectors on the subcells needed to
compute~\eqref{eq:lo.numflux} are taken from Appendix B.4
of~{\cite{henneman2021}} where they were derived by equating the high order
flux difference and Discontinuous Galerkin split form. We directly state the
normal vectors here, denoting $(\bn_{R,i})_{\bp}$ as the outward normal direction in
subcell $\Cijk$ along the positive $i$ direction
\[
(\bn_{R,i})_{\bp} = \II( J \ba^i )_{\bp_{i| 0}} + \sum_{l = 0}^{p_i} w_l
\partial_{\xii}  \II( J \ba^i )_{\bp_{i| \nocomma l}},  \qquad
( \bp_{i \nocomma |l} )_m = \begin{cases}
p_m \qquad & m \neq i\\
p_l &  m = i
\end{cases}
\]
where $\{ w_l \}_{l = 0}^N$ are quadrature weights corresponding
to solution points, $\II$ is the approximation operator for metric terms~\eqref{eq:metric.identity.contravariant.poly}, and $(\bn_{L,i})_{\bp}$ can be obtained by the relation $(\bn_{L,i})_{\bp} = -(\bn_{R,i})_{\bp_{i-}}$, where $\bp_{i-}$ was defined in~\eqref{eq:lo.numflux}.

\paragraph{Free-stream preservation.} To show the free stream preservation of the lower order scheme with the chosen normal vectors, we consider a constant initial state $\uu = \bc$ and show that the finite volume residual will be zero. A constant state implies that the time average of the contravariant flux will be the contravariant flux itself~\eqref{eq:time.avg.is.physical.flux}. Thus, all numerical fluxes including element interface fluxes are first order fluxes like in~\eqref{eq:lo.numflux} and the residual at $\bp$ in direction $i$ is given by
\begin{align*}
\frac{\pf ( \bc )}{w_{p_i}} \cdot ( (\bn_{R,i})_{\bp} +
(\bn_{L,i})_{\bp} ) & = \frac{\pf ( \bc )}{w_{p_i}} \cdot
( \II( J \ba^i )_{\bp_{i \nocomma |0}} + \sum_{l = 0}^{p_i} w_l
\partial_{\xii}  \II( J \ba^i )_{\bp_{i| \nocomma l}} - \II( J \ba^i )_{\bp_{i \nocomma |0}} - \sum_{l = 0}^{p_i - 1} w_l
\partial_{\xii}  \II( J \ba^i )_{\bp_{i \nocomma |l}} )\\
& = \pf ( \bc ) \cdot \partial_{\xii}  \II( J \ba^i
)_{\bp}
\end{align*}
The residuals in other directions give similar terms and summing them gives
\[ \pf ( \bc ) \cdot \sum_{i = 1}^d \pdv{}{\xii}  \II ( J
\ba^i )_{\bp} = \bzero \]
by the metric identities, thus satisfying the free stream preservation condition.

\subsubsection{Smoothness indicator}\label{sec:smooth.ind}

The smoothness of the numerical solution is assessed by writing the degree $N$ polynomial within each element in terms of an orthonormal basis like Legendre polynomials and then analyzing the decay of the coefficients~{\cite{Persson2006,klockner2011,henneman2021}}. For a system of PDE, the orthonormal expansion of a derived quantity $q(\uu)$ is used; a good choice for Euler's equations is the product of density and pressure~\cite{henneman2021} which depends on all the conserved quantities.

Let $q = q (\uu)$ be the quantity used to measure the solution smoothness. With
$\{ L_j \}_{j \subindex 0}^N$ being the 1-D Legendre polynomial basis of
degree $N$, taking tensor product gives the degree $N$ Legendre basis
\[ \Leg_{\bp} ( \vxi ) = \prod_{i = 1}^d \Leg_{p_i} ( \xii
), \qquad p_i \in \{ 0, 1, \ldots, N \} \]
The Legendre basis representation of $q$ can be obtained as
\[ q_h (\vxi) = \sum_{\bp} \hat{q}_{\bp} L_{\bp} (\vxi), \quad \vxi \in \Oo,
\qquad \hat{q}_{\bp} = \int_{\Oo} q (\uu^{\delta} (\vxi)) L_j (\vxi)  \ud
\vxi \]
The Legendre coefficients $\left\{ \hat{q}_{\bp} \right\}$ are computed using
the quadrature induced by the solution points,
\[ \hat{q}_{\bp} = \sum_{\bq} q (\uebq) L_{\bp} (\vxi_{\bq}) w_{\bq} \]
Define
\[ \mathbb{S}_K = \sum_{\bp, p_i \leq K} \hat{q}_{\bp}^2 \]
which the measures the "energy" in $q_h$. Then the energy contained in highest modes relative to the total energy of the
polynomial is computed as follows
\[ \en = \max \left\{ \frac{\mathbb{S}_N -\mathbb{S}_{N - 1}}{\mathbb{S}_N},
\frac{\mathbb{S}_{N - 1} -\mathbb{S}_{N - 2}}{\mathbb{S}_{N - 1}} \right\}
\]
In 1-D, this simplifies to the expression
of~{\cite{henneman2021,babbar2023admissibility}}
\[ \en = \max \left \{ \frac{\hat{q}_N^2}{\sum_{j = 0}^N \hat{q}_j^2},
\frac{\hat{q}_{N - 1}^2}{\sum_{j = 0}^{N - 1} \hat{q}_j^2} \right \} \]
The $N^{\tmop{th}}$ Legendre coefficient $\hat{q}_N$ of a function which is in
the Sobolev space $H^2$ decays as $O (1 / N^2)$ (see Chapter 5, Section 5.4.2
of~{\cite{Canuto2007}}). We consider smooth functions to be those whose
Legendre coefficients $\hat{q}_N$ decay at rate proportional to $1 / N^2$ or
faster so that their squares decay proportional to $1 /
N^4$~{\cite{Persson2006}}. Thus, the following dimensionless threshold for
smoothness is proposed in~{\cite{henneman2021}}
\[ \thresh (N) = a \cdot 10^{- c (N + 1)^{\nosymbol 4}} \]
where parameters $a = \half$ and $c = 1.8$ are obtained through numerical
experiments. To convert the highest mode energy indicator $\en$ and threshold
value $\thresh$ into a value in $[0, 1]$, the logistic function is used
\[ \tilde{\alpha} (\en) = \frac{1}{1 + \exp ( - \frac{s}{\thresh} (\en -
\thresh) )} \]

The sharpness factor $s$ was chosen to be $s = 9.21024$ so that blending
coefficient equals $\alpha = 0.0001$ when highest energy indicator $\en = 0$.
In regions where $\tilde{\alpha} = 0$ or $\tilde{\alpha} = 1$, computational
cost can be saved by performing only the lower order or higher order scheme
respectively. Thus, the values of $\alpha$ are clipped as
\[ \alpha_e \assign \begin{cases}

0, \quad & \text{if } \tilde{\alpha} < \alpha_{\min}\\
\tilde{\alpha}, & \text{if } \alpha_{\min} \le \tilde{\alpha} \le 1 -
\alpha_{\min}\\
1, & \text{if } 1 - \alpha_{\min} < \tilde{\alpha}
\end{cases} \]
with $\alpha_{\min} = 0.001$. Finally, since shocks can spread to the
neighbouring cells as time advances, some smoothening of $\alpha$ is performed as
\begin{equation}
\alpha_e^{\tmop{final}} = \max_{E \in \mathcal{E}_e} \left\{ \alpha_e, \half
\alpha_E \right\} \label{eq:smoothing}
\end{equation}
where $\mathcal{E}_e$ denotes the set of elements sharing a face with $\Omega_e$.

\subsection{Flux limiter for admissibility
preservation}\label{sec:flux.correction}

We first review the flux limiting process for admissibility preservation from~{\cite{babbar2023admissibility}} for 1-D and then do a natural extension to curvilinear meshes. The first step in obtaining an admissibility preserving blending scheme is to ensure that the lower order scheme preserves the admissibility set $\Uad$. This is always true if all the fluxes in the lower order method are computed with an admissibility preserving low order finite volume method. But the LWFR scheme uses a time average numerical flux and maintaining conservation requires that we use the same numerical flux at the element interfaces for both lower and higher order schemes (Remark~1 of~\cite{babbar2023admissibility}). To maintain accuracy and admissibility, we carefully choose a blended numerical flux $\F_{\eph}$ as in~\eqref{eq:Fblend} but this choice may not ensure admissibility and further limitation is required. Our proposed procedure for choosing the blended numerical flux will give us an admissibility preserving lower order scheme. As a result of using the same numerical flux at element interfaces in both high and low order schemes, element means of both schemes are the same (Theorem~\ref{thm:lwfr.admissibility}). A consequence of this is that our scheme now preserves admissibility of element means and thus we can use the scaling limiter of~{\cite{Zhang2010b}} to get admissibility at all solution points.

The theoretical basis for flux limiting can be summarized in the following
Theorem~\ref{thm:lwfr.admissibility}.

\begin{theorem}
\label{thm:lwfr.admissibility}Consider the LWFR blending
scheme~\eqref{eq:blended.scheme} where low and high order schemes use the
same numerical flux $( \tF^{\delta}_e \cdot \bnr_i )^{\ast}
( \vxi_i^s )$ at every element interface and the lower order
residual is computed using the first order finite volume
scheme~\eqref{eq:lower.order.curved}. Then the following can be said about
admissibility preserving in means property (Definition~\ref{defn:mean.pres})
of the scheme:
\begin{enumerate}
\item element means of both low and high order schemes are same, and thus
the blended scheme~\eqref{eq:blended.scheme} is admissibility preserving
in means if and only if the lower order scheme is admissibility preserving
in means;

\item if the blended numerical flux $( \tF^{\delta}_e \cdot \bnr_i
)^{\ast} ( \vxi_i^s )$ is chosen to preserve the
admissibility of lower-order updates at solution points adjacent to the
interfaces, then the blending scheme~\eqref{eq:blended.scheme} will
preserve admissibility in means.
\end{enumerate}
\end{theorem}

\begin{proof}
By~(\ref{eq:conservation.lw}, \ref{eq:low.order.cell.avg.update.curved}), element means are the same for both low and high order schemes. Thus, admissibility in means of one implies the same for the other, proving the first claim. For the second claim, note that our assumptions imply $\uebp^{L, n + 1}$ given by~(\ref{eq:lower.order.curved}, \ref{eq:blended.flux.in.curved}) are in $\Uad$ for all $\bp$. Therefore, we obtain admissibility in means property of the lower order scheme by~\eqref{eq:low.order.cell.avg.update.curved} and thus admissibility in
means for the blended scheme.
\end{proof}

\subsubsection{Flux limiter for admissibility preservation in
1-D}\label{sec:flux.limiter.1d}

Flux limiting ensures that the update obtained by the lower order scheme will
be admissible so that, by Theorem~\ref{thm:lwfr.admissibility}, admissibility
in means is obtained. The procedure of flux limiting will be explained for the element
$\Omega_e = [ x_{\emh}, x_{\eph} ]$. The lower order scheme is
computed with a first order finite volume method so that admissibility is
already ensured for inner solution points; i.e., we already have
\[ \uep^{L, n + 1} \in \Uad, \qquad 1 \le p \le N - 1 \]
The admissibility for the first ($p = 0$) and last solution points ($p = N$)
will be enforced by appropriately choosing the inter-element flux $\F_{\eph}$.
The first step is to choose a candidate for $\F_{\eph}$ which is heuristically
expected to give reasonable control on spurious oscillations while maintaining
accuracy in smooth regions, e.g.,
\[ \F_{\eph} = (1 - \alpha_{\eph})  \F^{\tmop{LW}}_{\eph} + \alpha_{\eph}
\pf_{\eph}, \qquad \alpha_{\eph} = \frac{\alpha_e + \alpha_{e + 1}}{2} \]
where $\pf_{\eph}$ is the lower order flux at the face $\eph$ shared between
FR elements and subcells, and $\alpha_e$ is the blending
coefficient~\eqref{eq:blended.scheme} based on an element-wise smoothness
indicator (Section~\ref{sec:smooth.ind}).

The next step is to correct $\F_{\eph}$ to enforce the admissibility
constraints. The guiding principle of this approach is to perform the
correction within the face loops, minimizing storage requirements and
additional memory reads. The lower order updates in subcells neighbouring the
$\eph$ face with the candidate flux are
\begin{equation}
\begin{split}
\atu_0^{n + 1} & = \uepoz^n - \frac{\Delta t}{w_0 \Delta x_{e + 1}}
(\pf^{e + 1}_{\frac{3}{2}} - \F_{\eph})
\label{eq:low.order.tilde.update}\\
\atu_N^{n + 1} & = \ueN^n - \frac{\Delta t}{w_N \Delta x_e}  (\F_{\eph} -
\pf^e_{\Nph})
\end{split}
\end{equation}
To correct the interface flux, we will again use the fact that low order
finite volume flux $\pf_{\eph} = \pf (\uu_{\eph}^-, \uu_{\eph}^+)$ preserves
admissibility, i.e.,
\begin{align*}
\utilow_0 & = \uepoz^n - \frac{\Delta t}{w_0 \Delta x_{e + 1}}  (\pf^{e +
1}_{\frac{3}{2}} - \pf_{\eph}) \in \Uad\\
\utilow_N & = \ueN^n - \frac{\Delta t}{w_N \Delta x_e}  (\pf_{\eph} -
\pf^e_{\Nph}) \in \Uad
\end{align*}
Let $\{p_k, 1 \le 1 \le K\}$ be the admissibility
constraints~\eqref{eq:uad.form} of the conservation law. The numerical flux is
corrected by iterating over the admissibility constraints as explained in
 Algorithm~\ref{alg:blended.flux}

\begin{algorithm}
\caption{Computation of blended flux $\F_{\eph}$\label{alg:blended.flux}}
\textbf{Input:} $\F_{\eph}^{\tmop{LW}}, \pf_{\eph}, \pf^{e + 1}_{\frac{3}{2}}, \pf^e_{\Nph}, \uepoz^n, \ueN^n, \alpha_e, \alpha_{e + 1}$ \\
\textbf{Output:} $\F_{\eph}$
\begin{algorithmic} 
\State $\alpha_{\eph} = \frac{\alpha_e + \alpha_{e + 1}}{2}$

\State $\F_{\eph} \gets (1 - \alpha_{\eph}) \F_{\eph}^{\tmop{LW}} + \alpha_{\eph}  \pf_{\eph}$
\Comment{Heuristic guess to control oscillations}
\State {$\atu_0^{n + 1} \gets \uepoz^n - \frac{\Delta t}{w_0 \Delta x_{e + 1}}  (\pf^e_{\frac{3}{2}} - \F_{\eph})$} \Comment{FV inner updates with guessed $\F_{\eph}$}

\State $\atu_N^{n + 1} \gets \ueN^n - \frac{\Delta t}{w_N \Delta x_e}  (\F_{\eph} - \pf^e_{\Nph})$
\State {$\utilow_0 = \uepoz^n - \frac{\Delta t}{w_0 \Delta x_{e +  1}}  (\pf^{e + 1}_{\frac{3}{2}} - \pf_{\eph})$}\Comment{FV inner updates with $\pf_{\eph}$}

\State {$\utilow_N = \ueN^n - \frac{\Delta t}{w_N \Delta x_e}  (\pf_{\eph} - \pf^e_{\Nph})$}
\For{$k = 1 : K$} \Comment{Correct $\F_{\eph}$ for $K$ admissibility constraints}
\State \correction{$\epsilon_0 \gets \frac{1}{10} p_k (\utilow_0)$}

\State \correction{$\epsilon_N \gets \frac{1}{10} p_k (\utilow_N)$}

\State {$\theta \gets \min \left ( \min_{j = 0, N} \left| \frac{\epsilon_j - p_k (\atu_j^{\text{low}, n + 1}))}{p_k (\atu_j^{n + 1}) - p_k (\atu_j^{\text{low}, n + 1})} \right|, 1 \right )$}

\State {$\F_{\eph} \gets \theta \F_{\eph} + (1 - \theta)  \pf_{\eph}$}

\Comment FV inner updates with guessed $\F_{\eph}$

\State {$\atu_0^{n + 1} \gets \uepoz^n - \frac{\Delta t}{w_0 \Delta x_{e + 1}}  (\pf^{e + 1}_{\frac{3}{2}} - \F_{\eph})$}

\State {$\atu_N^{n + 1} \gets \ueN^n - \frac{\Delta t}{w_N\Delta x_e}  (\F_{\eph} - \pf^e_{\Nph})$}
\EndFor
\end{algorithmic}
\end{algorithm}

By concavity of $p_k$, after the $k^{\tmop{th}}$ iteration, the updates
computed using flux $\F_{\eph}$ will satisfy
\begin{equation}
\begin{split}
p_k (\atu_j^{n + 1}) & = p_k (\theta (\atu_j^{n + 1})^{\tmop{prev}} + (1 -
\theta) \utilow_j)\\
& \ge \theta p_k ((\atu_j^{n + 1})^{\tmop{prev}}) + (1 -
\theta) p_k (\utilow_j) \\
&\ge \epsilon_j, \quad j = 0, N
\end{split}
\end{equation}
satisfying the $k^{\tmop{th}}$ admissibility constraint; here $(\atu_j^{n +
1})^{\tmop{prev}}$ denotes $\atu_j^{n + 1}$ before the $k^{\tmop{th}}$
correction and the choice of $\epsilon_j = \frac{1}{10} p_k (\utilow_j)$ is
made following~{\cite{ramirez2021}}. After the $K$ iterations, all
admissibility constraints will be satisfied and the resulting flux $\F_{\eph}$
will be used as the interface flux keeping the lower order updates and thus
the element means admissible. Thus, by Theorem~\ref{thm:lwfr.admissibility},
the choice of blended numerical flux gives us admissibility preservation in
means. We now use the scaling limiter of~{\cite{Zhang2010b}} to obtain an
admissibility preserving scheme as defined in Definition~\ref{defn:adm.pres}.
An overview of the complete residual computation of Lax-Wendroff Flux
Reconstruction scheme can be found in Algorithm~\ref{alg:lw.residual}.

\subsubsection{Flux limiter for admissibility preservation on curved
meshes}\label{sec:flux.correction.curved}

Consider the calculation of the blended numerical flux for a corner
solution point of the element, see Figure~\ref{fig:curved.subcells}. A corner
solution point is adjacent to interfaces in all $d$ directions, making its
admissibility preservation procedure different from 1-D. In particular, let us
consider the corner solution point $\bp = \bzero$ and show how we can apply
the 1-D procedure in Section~\ref{sec:flux.limiter.1d} to ensure admissibility
at such points. The same procedure applies to other corner and non-corner
points. The lower order update at the corner is given
by~\eqref{eq:blended.flux.in.curved}
\begin{equation}
\begin{split}
\atu_{e, \bzero}^{n + 1} & = \uu_{e, \bzero}^n - \frac{\Delta t}{J_{e, \bp}} \sum_{i = 1}^d
\frac{1}{w_{p_i}}  [ (
\tf^{\delta}_{C_{\bzero}} \cdot \bnr_{R,i} )^{\ast} ( \vxi_i^R
) + ( \atF^{\delta}_e \cdot \bnr_{L,i} )^{\ast} (
\vxi_i^L ) ]
\end{split}
\end{equation}
where $\bnr_i = \be_i$ is the reference normal vector on the subcell interface
in direction $i$, $( \tf^{\delta}_{C_{\bzero}} \cdot \bnr_{R,i}
)^{\ast}$ denotes the lower order
flux~\eqref{eq:lower.order.flux.curved} at the subcell $C_{\bzero}$
surrounding $\vxi_{\bzero}$, $( \atF^{\delta}_e \cdot \bnr_{L,i}
)^{\ast} ( \vxi_i^L )$ is the initial guess candidate for the
blended numerical flux. Pick $k_i > 0$ such that $\sum_{i = 1}^d k_i = 1$ and
\begin{equation}
\atu_i^{\text{low}, n + 1} : = \uu_{e, \bzero}^n - \frac{\Delta t}{k_i
w_{p_i} J_{e, \bp}}  [ ( \tf^{\delta}_{C_{\bzero}} \cdot \bnr_{R,i}
)^{\ast} ( \vxi_i^R ) + ( \tf \cdot \bnr_{L,i}
)^{\ast} ( \vxi_i^L ) ], \qquad 1 \leq i \leq d
\label{eq:low.update.2d}
\end{equation}
satisfy
\begin{equation}
\atu_i^{\text{low}, n + 1} \in \Uad, \qquad 1 \leq i \leq d
\label{eq:2d.low.update.admissibility.condn}
\end{equation}
where $( \tf \cdot \bnr_i )^{\ast} ( \vxi_i^L )$ is the
first order finite volume flux computed at the FR element interface.

The $\{ k_i \}$ that ensure~\eqref{eq:2d.low.update.admissibility.condn} will
exist provided the appropriate CFL restrictions are satisfied because the
lower order scheme using the first order numerical flux at element interfaces
is admissibility preserving. The choice of $\{ k_i \}$ should be made so
that~\eqref{eq:2d.low.update.admissibility.condn} is satisfied with the least
time step restriction. However, we make the trivial choice of equal $k_i$'s
motivated by the experience of~{\cite{babbar2023admissibility}}, where it was
found that even this choice does not impose any additional time step constraints over
the Fourier stability limit. After choosing $k_i$'s, we have reduced the
update to 1-D and can repeat the same procedure as in
Algorithm~\ref{alg:blended.flux} where for all directions $i$, the
neighbouring element is chosen along the normal direction. After the flux
limiting is performed following the Algorithm~\ref{alg:blended.flux}, we obtain $(\atF^{\delta}_{e}{\cdot}{\bnr}_{L,i})^{\ast}(\vxi_{i}^L)$
such that
\begin{equation}
\atu_i^{n + 1} : = \uu_{e, \bzero}^n - \frac{\Delta t}{k_i w_{p_i}
J_{e, \bp}}  [ ( \tf^{\delta}_{C_{\bzero}} \cdot \bnr_{R,i}
)^{\ast} ( \vxi_i^R ) + ( \atF^{\delta} \cdot \bnr_{L,i}
)^{\ast} ( \vxi_i^L ) ] \in \Uad
\label{eq:2d.adm.numflux.desired}
\end{equation}
Then, we will get
\begin{equation}
\sum_{i = 1}^d k_i  \atu_i^{n + 1} = \atu_{e, \bzero}^{n + 1} \in
\Uad \label{eq:2d.xy.implies.admissibility}
\end{equation}
along with admissibility of all other corner and non-corner solution points
where the flux $( \atF^{\delta} \cdot \bnr_i )^{\ast} (
\vxi_i^R )$ is used. Finally, by Theorem~\ref{thm:lwfr.admissibility},
admissibility in means (Definition~\ref{defn.mean}) is obtained and the
scaling limiter of~{\cite{Zhang2010b}} can be used to obtain an admissibility
preserving scheme (Definition~\ref{defn:adm.pres}).


\section{Conservation property of LWFR on curvilinear grids}\label{sec:app.lwfr.conservative}
In order to show that the LWFR scheme is conservative, multiply~\eqref{eq:lwfr.update.curvilinear} with $J_{e,\bp} w_{\bp}$ and sum over $\bp \in \Nnd$ to get, using the exactness of quadrature
\begin{equation}
\label{eq:semi.integral.form}
\begin{split}
 \overline{\uu}_e^{n + 1} = \overline{\uu}_e^n & - \frac{\Delta t}{|\Omega_e|}
\int_{\Oo}\nabla_{\vxi} \cdot \tF^{\delta}_e ( \vxi ) \ud \vxi \\
& - \frac{\Delta t}{|\Omega_e|}  \int_{\Oo}\sum_{i = 1}^d ( (
\tF_e \cdot \bnr_{R,i} )^{\ast} - \tF^{\delta}_e \cdot \bnr_{R,i} )
( \vxi_i^R ) g_R' (\xi_{p_i}) -
 ( ( \tF_e \cdot \bnr_{L,i} )^{\ast} -
\tF^{\delta}_e \cdot \bnr_{L,i} ) ( \vxi_i^L ) g_L' (\xi_{p_i}) \ud \vxi
\end{split}
\end{equation}
where $\vxi_i^s$ are as defined in~\eqref{eq:xis.notation}. Then, note the following integral identities that are an application of Fubini's theorem followed by fundamental theorem of Calculus
\begin{align*}
\int_{\Oo}\partial_{\xi^i} \cdot \tF^{\delta}_e ( \vxi ) \ud \vxi &= \int_{\partial \Omega_{o,i}^L} [\tF^{\delta}_e \cdot \bnr_{L,i}] \ud S_{\vxi} + \int_{\partial \Omega_{o,i}^R} [\tF^{\delta}_e \cdot \bnr_{R,i}]\ud S_{\vxi} \\
\int_{\Oo} ((
\tF_e \cdot \bnr_{R,i} )^{\ast} - \tF^{\delta}_e \cdot \bnr_{R,i})
( \vxi_i^R )g_R' (\xi_{p_i}) \ud \vxi &= \int_{\partial \Omega_{o,i}^R} [(
\tF_e \cdot \bnr_{R,i} )^{\ast} - \tF^{\delta}_e \cdot \bnr_{R,i}] \ud S_{\vxi} \\
\int_{\Oo} ((
\tF_e \cdot \bnr_{L,i} )^{\ast} - \tF^{\delta}_e \cdot \bnr_{L,i})
( \vxi_i^L )g_L' (\xi_{p_i}) \ud \vxi &= -\int_{\partial \Omega_{o,i}^L} [(
\tF_e \cdot \bnr_{L,i} )^{\ast} - \tF^{\delta}_e \cdot \bnr_{L,i}]\ud S_{\vxi}
\end{align*}
where $\Ois$ is as in Figure~\ref{fig:ref.map} and we used $g_L(-1) = g_R(1) = 1$, $g_L(-1) = g_R(1) = 0$. Then substituting these identities into~\eqref{eq:semi.integral.form} gives us the conservative update of the cell average~\eqref{eq:conservation.lw}.

\bibliographystyle{siam}
\bibliography{references}

\end{document}